\documentclass[10pt]{article}
\usepackage[utf8]{inputenc}
\usepackage{amsmath, amssymb, geometry}
\usepackage{graphicx}
\usepackage{comment}
\usepackage{mathrsfs}
\usepackage{color}
\usepackage{bm}
\usepackage{xcolor}

\usepackage{hyperref}
\hypersetup{
    colorlinks = true,     
    linkcolor  = blue!70!black,     
    citecolor  = green!50!black, 
    urlcolor   = purple!80!black   
}

\usepackage[final]{showkeys}
\usepackage{physics}
\usepackage{tikz}\usepackage{footmisc}
\usepackage{tabularray}
\usepackage{bbm}
\usepackage{cancel}
\usepackage{diagbox}
\usepackage{enumitem}
\usepackage[normalem]{ulem}\usepackage{footmisc}
\usepackage{cite}
\allowdisplaybreaks[4]
\usepackage{orcidlink}

\usepackage{footmisc}
\DefineFNsymbols{nodaggers}{%
  * 
  {\S}  
  {\P}  
  {**} 
}
\setfnsymbol{nodaggers}

\usepackage{amsthm}
\newtheorem{lemma}{Lemma}

\newtheorem{remark}{Remark}
\newtheorem{theorem}{Theorem}
\newtheorem{corollary}{Corollary}
\newtheorem{proposition}{Proposition}

\DeclareMathOperator{\Id}{Id}

\newcommand{\Ppt}{\Pi^{\partial_t^2}}
\newcommand{\Pg}{\Pi^{\nabla_{\bx}}}

\newcommand{\bx}{{\boldsymbol{x}}}
\newcommand{\A}{\mathcal{A}}
\newcommand{\V}{\mathcal{V}}

\newcommand{\R}{\mathbb{R}}

\newcommand{\jx}{j_{\bx}}
\newcommand{\jt}{j_t}
\newcommand{\Vx}{V^{\bx}_{j_{\bx}}(\Omega)}
\newcommand{\Vt}{V_{j_t}^t(0,T)}

\newcommand{\Pinfjt}{P_{\infty,j_t}u}
\newcommand{\Pjxinf}{P_{j_{\bx},\infty}u}

\newcommand{\vanishingconditions}{{discrete compatibility conditions\ }}
\newcommand{\Vanishingconditions}{{Discrete compatibility conditions\ }}
\newcommand{\vanishingassumptions}{{discrete compatibility assumptions\ }}

\newcommand{\eremk}{\hbox{}\hfill\rule{0.8ex}{0.8ex}}

\title{A space--time sparse-grid method for the wave equation\footnote{\emph{Funding:} AM and CP acknowledge support from the PRIN project ``ASTICE'' (202292JW3F) funded by the European Union -- NextGenerationEU. This research was also funded in part by the Austrian Science Fund (FWF) project~10.55776/F65. MF, AM, and CP are members of the Gruppo Nazionale Calcolo Scientifico-Istituto Nazionale di Alta Matematica (GNCS-INdAM).}}

\author{Matteo~Ferrari\thanks{Fakult\"at f\"ur Mathematik, Universit\"at Wien, 
Vienna, Austria (\href{mailto:matteo.ferrari@univie.ac.at}{matteo.ferrari@univie.ac.at}, \href{mailto:ilaria.perugia@univie.ac.at}{ilaria.perugia@univie.ac.at})}\ \orcidlink{0000-0002-2577-1421} \and Andrea Moiola\thanks{Dipartimento di Matematica, Universit\`a di Pavia, 
Pavia, Italy (\href{mailto:andrea.moiola@unipv.it}{andrea.moiola@unipv.it}, \href{mailto:chiara.perinati01@universitadipavia.it}{chiara.perinati01@universitadipavia.it})}\ \orcidlink{0000-0002-6251-4440} \and Chiara Perinati\footnotemark[3]\ \orcidlink{0009-0002-8819-928X} \and Ilaria Perugia\footnotemark[2]\ \orcidlink{0000-0003-1368-2883}}

\date{}

\begin{document}

\maketitle

\begin{abstract}
\noindent
We develop a fast space--time numerical scheme for approximating solutions to the linear wave equation.
The approach is based on the sparse-grid combination technique applied to a coercive space--time discretization. Designed for tensor-product space--time discretizations, the method enables efficient parallelization of the resulting solver. We provide a rigorous theoretical analysis establishing convergence rates and computational complexity estimates. Numerical experiments validate the theoretical estimates and demonstrate the efficiency of the proposed method.
\end{abstract}


\section{Introduction}

In recent years, an increasing amount of literature has been devoted to the design of \textit{conforming} space--time Galerkin methods for hyperbolic problems (we refer to \cite[\S 1]{FerrariPerugiaZampa2026} for an overview of recent developments, and to~\cite{Gomez2026} for a comprehensive survey of space--time methods). Concurrently, progress in developing fast solvers has allowed space--time methods to compete with traditional method-of-lines approaches that combine spatial finite element discretization with ODE time integrators. Specifically, efficient strategies have been designed to exploit the space--time tensor-product structure, enabling the efficient resolution of the resulting linear systems through their Kronecker product properties. For instance, \cite{Zank2025} investigates various fast direct solution techniques for second-order-in-time numerical schemes based on continuous piecewise polynomials (see \cite{LoliSangalli2025} for an isogeometric analysis approach), while \cite[\S 5.1]{FerrariFraschiniLoliPerugia2025} designs an analogous solver for a first-order-in-time formulation. In this paper, we design a fast conforming space--time numerical scheme that, similarly to \cite{FerrariFraschiniLoliPerugia2025,LoliSangalli2025,Zank2025}, takes advantage of this tensor-product structure; however, rather than relying on the Kronecker properties of the algebraic system, our approach is based on a sparse-grid discretization technique.

\paragraph{Model problem.} Let $Q_T:=\Omega \times (0,T)$ be a space--time cylinder, where $\Omega \subset \R^d$, $d=1,2,3$, is a  bounded, Lipschitz domain with boundary $\partial \Omega$, and  $T>0$ is a given final time. We consider the following initial-boundary value problem for the linear wave equation: find $u: Q_T \rightarrow \mathbb{R}$ such that
\begin{equation} \label{eq:1}
\begin{cases}
    \partial_t^2 u(\bx,t) - \nabla_{\bx} \cdot (c^2(\bx)\nabla_{\bx} u(\bx,t)) = f(\bx,t) &(\bx,t)\in Q_T,
    \\ u(\bx,t) = 0 & (\bx,t)\in\partial \Omega \times (0,T),
    \\ u(\bx,0) = u_0(\bx), \quad \partial_t u(\bx,0) = v_0(\bx) & \bx\in\Omega.
\end{cases}
\end{equation}
Here, the source term~$f$ belongs to~$L^2(Q_T)$, and the wave velocity~$c \in L^\infty(\Omega)$ satisfies~$c = c(\bx) \ge c_0$ a.e., for some $c_0 > 0$. Moreover, the initial data satisfy $u_0 \in H^1_0(\Omega)$ and $v_0 \in L^2(\Omega)$. We are interested in a numerical scheme that approximates a weak solution to \eqref{eq:1}.

\paragraph{Conforming space--time Galerkin scheme.} The underlying space--time scheme we consider was designed and analyzed in \cite{FerrariPerugia2026}. This scheme can be viewed as a variant of the discontinuous Galerkin--continuous Galerkin (DG--CG) method in \cite{Walkington2014,DongMascottoWang2026}, which, in turn, evolved from the classical space--time DG method of~\cite{HughesHulbert1988}. It is based on a second-order-in-time variational formulation of \eqref{eq:1}, where integration by parts is performed in space but not in time. A consistent penalty term is added to weakly enforce the initial condition on the first time derivative of the wave field at $t = 0$; furthermore, time-dependent exponential weights are used in the $L^2$ space--time scalar products to guarantee coercivity of the formulation in the energy norm $H^1(0,T;L^2(\Omega)) \cap L^2(0,T;H^1_0(\Omega))$, providing stability for any choice of discrete tensor-product spaces, with $C^1$-regular temporal component. Since coercivity and continuity are not satisfied with respect to the same norm, quasi-optimality does not follow automatically for any choice of discrete subspaces. However, optimal convergence rates with respect to the mesh size can be still established for specific tensor-product discretizations. For instance, \cite[Appendix A]{FerrariPerugia2026} shows this for $C^1$ splines in time of even polynomial degree on uniform meshes, and an analogous result is expected when using maximal-regularity splines in time.

\paragraph{Contribution.} In the present paper, building on the coercive structure of the scheme of~\cite{FerrariPerugia2026},  we establish optimal convergence rates for a sparse space--time version based on the \textit{combination technique} \cite{GriebelSchneiderZenger1992, BungartzGriebel2004, griebel2014convergence}. This provides an optimal-order convergent approximation to the solution of \eqref{eq:1} with a number of degrees of freedom that scales like that of a single stationary elliptic problem in $\Omega$ on the finest spatial grid.

\smallskip
\noindent
Given nested sequences of spatial $\{V^{\bx}_{j_{\bx}}(\Omega)\}_{j_{\bx}=0,\ldots,J}$ and temporal $\{V_{j_t}^t(0,T)\}_{j_t=0,\ldots,J}$ refinements, our goal is to compute an approximate solution at the finest full-grid space $V_J^{\bx}(\Omega) \otimes V_J^t(0,T)$ for a fixed, large $J \in \mathbb{N}$. Instead of computing this approximate solution directly (the full-grid solution), the combination technique constructs an approximation by combining solutions computed on suitably chosen coarser refinement levels. This \textit{sparse-grid} solution preserves the approximation properties of the \textit{full-grid} solution with a considerably reduced total number of degrees of freedom. The method naturally enables a parallelizable space--time solver, as the problems associated with different coarser refinement levels are independent and can be solved in parallel. To the best of our knowledge, this is one of the few conforming space--time methods for the wave equation that admits efficient parallelization while also being supported by a rigorous theoretical analysis. In this context, we also mention the approach of~\cite[\S 8.1.3]{Zank2025}, which relies on the fast diagonalization method \cite{LynchRiceThomas1964} applied to the Galerkin--Bubnov scheme from \cite{LoscherSteinbachZank2022}. Although this approach has demonstrated good numerical performance, a stability analysis of the corresponding conforming space--time discretization is currently missing. 
Non-conforming schemes for hyperbolic systems that are easily parallelizable and whose convergence analysis is available are those based on tent-pitching, e.g., \cite{MoRi05}.

\smallskip
\noindent
Since the combination technique requires solving multiple discrete problems on strongly anisotropic space--time meshes, applying it to an unconditionally stable formulation is crucial. However, as highlighted above, for the space--time formulation employed herein, quasi-optimality is not inherently guaranteed for arbitrary choices of conforming discrete spaces. As a consequence, analyzing directly the Galerkin method in the sparse-grid space as done in~\cite{GriebelHarbrecht2013} for elliptic problems, \cite{HarbrechtSchwabZank2025} for parabolic problems, and~\cite{BanjaiMelenkNochettoOtarolaSalgadoSchwab2019} for fractional diffusion via higher-dimensional extension, may lead, in the present setting, to reduced convergence rates. Moreover, this procedure can produce an approximate solution different from the combination technique solution, even though both belong to the same space, in contrast to the situation in~\cite{HarbrechtSiebnmorgenPeters2013}, where the two coincide. Therefore, our convergence analysis targets the combination technique solution itself directly, integrating the framework of~\cite{griebel2014convergence} with the space--time analytical techniques developed in~\cite{FerrariPerugia2026}. In contrast to the elliptic setting considered in~\cite{griebel2014convergence}, where standard $L^2$ projections are used, our analysis for the linear wave equation requires tailored space- and time-projection operators.

\smallskip
\noindent
We establish optimal convergence rates of the sparse-grid solution for arbitrary polynomial approximation degrees in both space and time, under suitable regularity assumptions and \vanishingassumptions on the analytical solution at~$t=0$ and on~$\partial\Omega$. Although these conditions do not appear to be strictly sharp, numerical experiments show that relaxing some of them can result in a reduced convergence order. 

\smallskip
\noindent
The combination technique has been  applied to space--time discontinuous Galerkin (DG) methods in \cite{BansalMoiolaPerugiaSchwab2021}. However, since stability in that setting is established only with respect to a mesh-dependent DG seminorm, the convergence analysis of the combination technique presents intrinsic difficulties. Similar challenges would likely arise if one were to apply the combination formula to the conforming space--time method of~\cite{FerrariPerugiaZampa2026}, for which the inf-sup stability relies on a mesh-dependent norm. We finally mention that, when $d>1$ and~$\Omega$ is a tensor product of domains in some coordinate system, the combination technique could be employed in the components of the space variable as well.

\paragraph{Outline.} The paper is organized as follows. In Section~\ref{sec:spacetimemethod}, we recall the space--time discretization method of~\cite{FerrariPerugia2026}, together with the corresponding error estimates. In Section~\ref{sec:combinationformula}, we introduce the space--time sparse-grid method based on the combination technique of~\cite{griebel2014convergence}, discuss its computational complexity, and state the main convergence result in Theorem~\ref{th:1}. Preliminary results on semidiscretization operators required for the proof of Theorem~\ref{th:1} are established in Section~\ref{sec:semidiscr}, while the proof itself is presented in Section~\ref{sec:convergence}. Finally, Section~\ref{sec:numericalexp} reports numerical experiments that validate the theoretical results and illustrate the performance of the proposed method. A summary of the main notation and symbols used throughout the paper is reported in Table~\ref{tab:3}.

\subsection*{Notation}\label{sec:notation}

For a bounded domain~$D\subset\R^d$ and~$k\in\mathbb{N}$, we denote by~$H^k(D)$ the Sobolev space of index~$k$ endowed with the standard Sobolev norm $\|\cdot\|_{H^k(D)}$~\cite{Brezis2010}. We also set
\begin{equation*}
    H_0^1(\Omega) := \{ u \in H^1(\Omega): u(\bx) = 0, \bx \in \partial\Omega\}, \quad\quad H^2_{0,\bullet}(0,T) := \{ u \in H^2(0,T) : u(0)=0 \}.
\end{equation*}
For $s_{\bx} \ge 1$ and $s_t \ge 2$, we also define, with nonstandard notation,
\begin{align*}
    & H^{s_{\bx}}_{0}(\Omega) := H^1_{0}(\Omega) \cap H^{s_{\bx}}(\Omega), \quad H^{s_t}_{0,\bullet}(0,T) := H^2_{0,\bullet}(0,T) \cap H^{s_t}(0,T).
\end{align*}
In the paper, we identify Bochner-type function spaces $H^k(0,T; H^\ell(\Omega))$ for nonnegative integers $k,\ell$ with tensor product of Hilbert spaces, i.e.,
\begin{equation*}
    H^k(0,T; H^\ell(\Omega)) \cong H^\ell(\Omega) \otimes H^k(0,T),
\end{equation*}
where $\cong$ denotes isometric isomorphism and $\otimes$ the Hilbertian tensor product (see, e.g.,  \cite[Theorem~1, Chapter~12, \S~7]{Aubin1979}).

\noindent 
For $u, v \in L^2(0,T)$, we define the exponentially-weighted scalar product and norm by
\begin{equation*}
    (u,v)_{L^2_e(0,T)} := \int_0^T  u(t)\, v(t)\, e^{-t/T} \, \dd t, \qquad \| u \|_{L^2_e(0,T)} := \sqrt{(u,u)_{L^2_e(0,T)}}.
\end{equation*}
This norm is clearly equivalent to the classical $L^2(0,T)$ norm, with equivalence constants independent of~$T$. Indeed, we have
\begin{equation*}
    \frac{1}{e} \| u \|_{L^2(0,T)} \le \| u \|_{L_e^2(0,T)} \le \| u \|_{L^2(0,T)} \quad \text{for all~} u \in L^2(0,T).
\end{equation*}
By integration by parts, the scalar product satisfies
\begin{equation} \label{eq:2}
    (u, v')_{L_e^2(0,T)} = -(u', v)_{L_e^2(0,T)} + \frac{1}{T}(u, v)_{L_e^2(0,T)} + \frac{1}{e}u(T)v(T) - u(0)v(0) \quad \text{for all~} u, v \in H^1(0,T).
\end{equation}
A key consequence is the following property, which motivates the use of the exponential weight in the method:
\begin{equation} \label{eq:3}
    (w, w')_{L_e^2(0,T)} + 
    |w(0)|^2 = \frac{1}{2T}\|w\|^2_{L_e^2(0,T)} + \frac{1}{2e}|w(T)|^2+\frac12 
    |w(0)|^2  \quad \text{for all } w \in H^1(0,T).
\end{equation}
Similarly, for space--time functions $u, v \in L^2(Q_T)$, we define the exponentially-weighted scalar product and norm
\begin{equation*}
    (u,v)_{L^2_e(Q_T)} := \int_0^T \int_\Omega u(\bx,t)\, v(\bx,t)\, e^{-t/T} \, \dd \bx \, \dd t, \qquad \| u \|_{L^2_e(Q_T)} := \sqrt{(u,u)_{L^2_e(Q_T)}}.
\end{equation*}
We report in Table~\ref{tab:3} the symbols used in the paper.

\section{Space--time continuous and discrete variational formulations}
\label{sec:spacetimemethod}

In this section, we briefly review the space–time method introduced in~\cite{FerrariPerugia2026}. We begin by introducing the underlying continuous variational formulation, and then proceed to describe the numerical scheme.

\subsection{Continuous variational formulation}\label{sec:well-posendess}

We define the space 
\begin{equation} \label{eq:4}
    \V(Q_T) := \big(H_0^1(\Omega) \otimes H^1(0,T)\big) \cap \big(L^2(\Omega) \otimes H_{0,\bullet}^2(0,T)\big),
\end{equation}
where~$\otimes$ is the Hilbertian tensor product, and the norm on $\V(Q_T)$
\begin{equation} \label{eq:5}
\begin{aligned} 
    \| u \|^2_{\mathcal{V}(Q_T)} := \frac{1}{T}\big(\| \partial_t u \|_{L^2(Q_T)}^2 + \| c \nabla_{\bx} u \|^2_{L^2(Q_T)} \big)  + \| \partial_t u (\cdot,T)\|^2_{L^2(\Omega)}+ \| \partial_t u(\cdot,0)\|^2_{L^2(\Omega)} + \| c \nabla_{\bx} u(\cdot,T)\|^2_{L^2(\Omega)}.
\end{aligned}
\end{equation}
Defining the bilinear form $\A : \V(Q_T) \times \V(Q_T) \to \R$
\begin{equation} \label{eq:6}
\begin{aligned}
    \A(u,w) := (\partial_t^2 u, \partial_t w)_{L_e^2(Q_T)}  + (\partial_t u(\cdot,0), \partial_t w(\cdot,0))_{L^2(\Omega)} 
    + (c^2\nabla_{\bx} u, \nabla_{\bx} \partial_t w)_{L_e^2(Q_T)},
\end{aligned}
\end{equation}
we consider the following weak formulation of~\eqref{eq:1}: find $u \in \V(Q_T)$ such that
\begin{equation} \label{eq:7}
    \A(u,w) = (f, \partial_t w)_{L^2_e(Q_T)} + (v_0, \partial_t w(\cdot,0))_{L^2(\Omega) } - (c^2 \nabla_{\bx} u_0, \nabla_{\bx} \partial_t w)_{L_e^2(Q_T)} \quad \text{for all } w \in \V(Q_T).
\end{equation}
The following coercivity property holds true (see \cite[Lemma 2.1]{FerrariPerugia2026}\footnote{In \cite[Lemma 2.1]{FerrariPerugia2026}, the bilinear form $\A$ is defined with the standard scalar product, and the operator $\mathcal{L}_T u$ satisfies
$\partial_t \mathcal{L}_T u(t) = e^{-t/T} \partial_t u(t)$.}):
\begin{equation} \label{eq:8}
    \mathcal{A}(u,u) \ge \frac{1}{2e}  \| u \|^2_{\mathcal{V}(Q_T)} \quad \text{for all~} u \in \V(Q_T),
\end{equation}
while a continuity property of the type~$|\mathcal{A}(u,v)|\le C \| u \|_{\mathcal{V}(Q_T)}\| v \|_{\mathcal{V}(Q_T)}$ for all~$u,v\in \mathcal{V}(Q_T)$ 
is not satisfied.
Nevertheless, well-posedness of problem \eqref{eq:7} in $\mathcal{V}(Q_T)$ has been shown in \cite[Theorem 2.1]{FerrariPerugia2026} under the following assumptions on the data:
\begin{equation*}
\begin{aligned}
    f \in L^2(\Omega) \otimes H^1(0,T), \quad f(\cdot,0) & \in L^2(\Omega), 
    \\ u_0 \in H_0^1(\Omega), \quad \nabla_{\bx} \cdot(c^2 \nabla_{\bx} u_0) \in L^2(\Omega), & \quad v_0 \in H^1_0(\Omega).
\end{aligned}
\end{equation*}

\medskip
\noindent
Then, the weak solution to the wave equation \eqref{eq:1} is $u(\bx,t)+u_0(\bx)$, where~$u_0$ is also used to denote its constant-in-time extension to~$Q_T$.

\subsection{Space--time full-grid Galerkin discretization}
\label{sec:numericalscheme}

For the space discretization, we consider discrete spaces
\begin{equation*}
    \Vx\subset H_0^1(\Omega), \qquad N_{j_{\bx}}^{\bx} := \dim(V_{j_{\bx}}^{\bx}(\Omega)),
\end{equation*}
depending on mesh parameters~$h_{\jx} := h_0^{\bx} \, 2^{-\jx}$ with $h_0^{\bx}>0$ and $\jx \in \mathbb{N}$. For the time discretization, we consider discrete spaces
\begin{equation*}
    \Vt\subset H_{0,\bullet}^2(0,T), \qquad N_{j_t}^t := \dim(V_{j_t}^t(0,T)),
\end{equation*}
depending on mesh parameters $h_{\jt} := h_0^t \, 2^{-\jt}$ with $h_0^t >0$ and $\jt \in \mathbb{N}$. We assume that these spaces form nested sequences 
\begin{equation*}
    V_0^{\bx}(\Omega) \subset V_1^{\bx}(\Omega) \subset \ldots \subset V_{j_{\bx}}^{\bx}(\Omega) \subset \ldots \subset H_0^1(\Omega), \quad V_0^t(0,T) \subset V_1^t(0,T) \subset \ldots \subset V_{j_t}^t(0,T) \subset \ldots \subset H_{0,\bullet}^2(0,T).
\end{equation*}
The Galerkin approximation of \eqref{eq:7} reads as follows: find~$u_{j_{\bx},j_t}\in \Vx \otimes \Vt$ such that
\begin{equation} \label{eq:9}
    \mathcal{A}(u_{j_{\bx},j_t}, w_{j_{\bx},j_t}) = (f, \partial_t w_{j_{\bx},j_t})_{L^2_e(Q_T)} + (v_0, \partial_t w_{j_{\bx},j_t}(\cdot,0))_{L^2(\Omega)} - (c^2 \nabla_{\bx} u_0, \nabla_{\bx} \partial_t w_{j_{\bx},j_t})_{L^2_e(Q_T)}
\end{equation}
for all~$w_{j_{\bx},j_t} \in \Vx  \otimes \Vt$. We call the solution $u_{j_\bx,j_t}$ of \eqref{eq:9} a ``full-grid'' approximation of \eqref{eq:7}.

\medskip
\noindent 
Given $u \in \mathcal{V}(Q_T)$, we define $P_{j_{\bx},j_t} u \in \Vx \otimes \Vt$ as the solution of the discrete problem
\begin{equation} \label{eq:10}
     \mathcal{A}(P_{j_{\bx},j_t} u, w_{j_{\bx},j_t}) = \mathcal{A} (u, w_{j_{\bx},j_t})
     \quad\quad \text{for all~} w_{j_{\bx},j_t} \in \Vx  \otimes \Vt.
\end{equation}
From the coercivity in \eqref{eq:8}, this problem is well-posed for any choice of discrete spaces $\Vx$ and $\Vt$. Moreover, if $u$ solves the continuous problem~\eqref{eq:7}, then $P_{j_{\bx},j_t} u$ solves the discrete problem~\eqref{eq:9}.
This also defines the Galerkin projection operators~$P_{j_{\bx},j_t}: \V(Q_T)\to \Vx \otimes \Vt$.

\medskip
\noindent
We introduce the following projection operators:
\begin{itemize}
\item 
$\Pg_{j_{\bx}} : H^1(\Omega) \to \Vx$, defined as the solution to
\begin{equation} \label{eq:11}
    (c^2 \nabla_{\bx} (\Pg_{j_{\bx}} - \Id^{\bx}) w, \nabla_{\bx} w_{j_{\bx}})_{L^2(\Omega)} = 0 \qquad \text{for all~} w_{\jx} \in \Vx,
\end{equation}
\item 
$\Ppt_{j_t} : H^2(0,T) \to \Vt$, defined as the solution to 
\begin{equation} \label{eq:12}
\begin{aligned}
    (((\Ppt_{j_t}-\Id^t) w)'',  w_{j_t}')_{L_e^2(0,T)} + ( (\Ppt_{j_t} - \Id^t) w)'(0)  w_{j_t}'(0) = 0 \qquad \text{for all~} w_{\jt} \in \Vt.
\end{aligned} 
\end{equation}
\end{itemize}
We emphasize that the non-standard projection operator~$\Pi_{j_t}^{\partial_t^2}$ is well-defined, a property that follows directly from~\eqref{eq:3}. Let us assume that, associated with the families of spaces $\{\Vx\}_{j_{\bx}}$ and $\{\Vt\}_{j_t}$, there are approximation degrees $p_{\bx}, p_t \in \mathbb N$ such that, if $w \in H^{p_{\bx}+1}(\Omega)$, then 
\begin{equation} \label{eq:13}
\begin{aligned}
    \| c\nabla_{\bx} (\Pi_{j_{\bx}}^{\nabla_{\bx}}-\Id^{\bx}) w \|_{L^2(\Omega)} \lesssim 2^{-j_{\bx} p_{\bx}} \| w \|_{H^{p_{\bx}+1}(\Omega)}, \quad \| (\Pi_{j_{\bx}}^{\nabla_{\bx}}-\Id^{\bx}) w \|_{L^2(\Omega)} \lesssim 2^{-j_{\bx} (p_{\bx}+1)} \| w \|_{H^{p_{\bx}+1}(\Omega)},
\end{aligned}
\end{equation}
and, if $w \in H^{p_t+3}(0,T)$, then 
\begin{equation} \label{eq:14}
\begin{aligned}
    \| \partial_t (\Pi_{j_t}^{\partial_t^2}-\Id^t) w \|_{L^2(0,T)} \lesssim 2^{-j_t p_t} \| w \|_{H^{p_t+3}(0,T)}, \quad \|  (\Pi_{j_t}^{\partial_t^2}-\Id^t) w \|_{L^2(0,T)} \lesssim 2^{-j_t(p_t+1)} \| w \|_{H^{p_t+3}(0,T)}.
\end{aligned}
\end{equation}
Here and in the rest of the paper, we use the notation $``\lesssim"$ for $``\lesssim C \,"$, whenever the constant $C>0$ is independent of $ j_{\bx}, j_t$, and of any functions to which the estimate is applied.
\begin{remark}[Validity of assumptions~\eqref{eq:13} and~\eqref{eq:14}]\label{rem:1}
Assumption \eqref{eq:13} on $V^{\bx}_{j_{\bx}}(\Omega)$ holds, for instance, for conforming finite elements of degree~$p_{\bx}\ge 1$ on shape-regular simplicial meshes with mesh width $h_{j_\bx}$, provided that $\Omega$ and~$c$ are such that the elliptic regularity assumptions is satisfied for the Dirichlet--Poisson problem with operator~$-\nabla_{\bx} \cdot (c^2(\bx)\nabla_{\bx})$ (see, e.g., \cite[Theorem~22.6]{ErnGuermond2021a}). 
Assumption~\eqref{eq:14} on $V^{t}_{j_t}(0,T)$ holds for $C^1$ splines on a uniform mesh with mesh width $h_{j_t}$ with an even polynomial degree $p_t \ge 2$ (see \cite[Proposition~4.4 and Corollary~4.1]{FerrariPerugia2026}). It is further conjectured that the assumption is satisfied for all $C^{k_t}$ splines with polynomial degree $p_t \ge 2$ such that $p_t - k_t$ is odd; see Table \ref{tab:validity}.
\begin{table}[ht!]
\centering
\begin{tabular}{l|c|c}
    temporal spline type & even $p_t$ & odd $p_t$ \\ \hline $C^1$ splines & $\checkmark$ & $\times$
    \\ $C^{p_t-1}$ splines (maximal-regularity) & $\circ$  & $\circ$
    \\ $C^{k_t}$ splines, odd $k_t$ & $\circ$ & $\times$
    \\ $C^{k_t}$ splines, even $k_t$ & $\times$ & $\circ$
\end{tabular}
\caption{Validity of assumption~\eqref{eq:14}. Symbols indicate proven validity ($\checkmark$), conjectured validity ($\circ$), and not valid ($\times$), depending on spline regularity and polynomial degree parity.}
\label{tab:validity}
\end{table}
\eremk
\end{remark}
\noindent 
Assuming well-posedness of problem~\eqref{eq:7} (see Section~\ref{sec:well-posendess}), \cite[Theorem 4.4]{FerrariPerugia2026} shows that, if the solution~$u$ satisfies
\begin{equation} \label{eq:15}
    u \in (H_0^{p_{\bx}+1}(\Omega) \otimes H^2(0,T)) \cap \big(H^2(\Omega) \otimes H^{p_t+3}_{0,\bullet}(0,T)\big),
\end{equation}
then the following estimate holds:
\begin{equation} \label{eq:16}
\begin{aligned}
    \| (P_{j_{\bx},j_t} - \Id) u\|_{L^2(Q_T)} & \lesssim 2^{-j_{\bx}(p_{\bx}+1)} \| u \|_{H^{p_{\bx}+1}(\Omega) \otimes H^2(0,T)}  + 2^{-j_t(p_t+1)} \| u \|_{H^2(\Omega) \otimes H^{p_t+3}(0,T)}. 
\end{aligned}
\end{equation}
Estimate~\eqref{eq:16} follows from the key intermediate estimate\footnote{
Due to their tensor-product structure, the  extensions $\Pi_{j_\bx}^{\nabla_\bx}\otimes \Id^t$ and $\Id^{\bx} \otimes \Pi_{j_t}^{\partial_t^2}$  to space--time functions of the projectors $\Pi_{j_\bx}^{\nabla_\bx}$ and $\Pi_{j_t}^{\partial_t^2}$ satisfy the following commutativity properties:
\begin{equation*}
\begin{aligned}
    \mathcal{G}_t (\Pi^{\nabla_{\bx}}_{\jx}\otimes \Id^t)w = (\Pi^{\nabla_{\bx}}_{\jx}\otimes \Id^t) \mathcal{G}_t  w, \quad \quad 
    \mathcal{G}_{\bx}(\Id^{\bx}\otimes \Pi^{\partial_t^2}_{\jt})w = (\Id^{\bx}\otimes \Pi^{\partial_t^2}_{\jt}) \mathcal{G}_{\bx} w,
\end{aligned}
\end{equation*}
for any sufficiently regular function $w(\bx,t)$, where $\mathcal{G}_t$ and $\mathcal{G}_{\bx}$ are differential operators acting only on the variable $t$ and $\bx$, respectively.
}
\begin{equation} \label{eq:17}
\begin{aligned}
    \| (P_{j_{\bx},j_t} - \Pi_{j_{\bx}}^{\nabla_{\bx}}\otimes  \Pi_{j_t}^{\partial_t^2}) u \|_{\mathcal{V}(Q_T)} \lesssim \|(\Pi_{j_{\bx}}^{\nabla_{\bx}} \otimes \Id^t-\Id) \partial_t^2 u \|_{L^2(Q_T)} & + \|(\Pi_{j_{\bx}}^{\nabla_{\bx}} - \Id^{\bx}) v_0 \|_{L^2(\Omega)}
    \\ & + \|(\Id^{\bx} \otimes \Pi_{j_t}^{\partial_t^2}-\Id) \nabla_{\bx} \cdot (c^2 \nabla_{\bx} u) \|_{L^2(Q_T)},
\end{aligned}
\end{equation}
(see~\cite[Lemma 4.1]{FerrariPerugia2026}) combined with the Poincar\'e inequality and the approximation properties~\eqref{eq:13} and~\eqref{eq:14} of the projectors.
In particular, when $j_{\bx}=j_t=J$, estimate \eqref{eq:16} becomes
\begin{equation} \label{eq:18}
    \| (P_{J,J}-\Id)u\|_{L^2(Q_T)} \lesssim 2^{-J\min\{p_{\bx}+1,p_t+1\}} \| u \|_{(H^{p_{\bx}+1}(\Omega) \otimes H^2(0,T)) \cap (H^2(\Omega) \otimes H^{p_t+3}(0,T))}. 
\end{equation}
In the following section, we introduce a technique that maintains the same order of accuracy while significantly reducing the computational cost and enabling parallelization of the numerical scheme.

\section{The space--time sparse-grid method}\label{sec:combinationformula}

In this section, we introduce the space--time sparse-grid method based on the combination technique, and analyze its computational complexity in terms of the number of degrees of freedom (DoFs). We then state a convergence result for the resulting approximation, with the proof deferred to Section~\ref{sec:convergence}.

\subsection{The combination technique}
For a fixed maximal level $J \ge 0$, the \emph{combination formula}~\cite{griebel2014convergence} defines the approximation 
\begin{equation} \label{eq:19}
    u^{CF}_J := \sum_{j_{\bx}=0}^J (P_{j_{\bx},J-j_{\bx}} - P_{j_{\bx}-1,J-j_{\bx}})u,
\end{equation}
where $P_{j_{\bx},j_t}$ are the operators defined in \eqref{eq:10}, with the convention~$P_{-1,j_t} u := 0$. The combination formula solution $u^{CF}_J$ is obtained by computing~$2J+1$ Galerkin solutions of \eqref{eq:9}, one for each of the discrete spaces $V_{j_\bx}^\bx(\Omega)\otimes V_{J-j_\bx}^t(0,T)$ with $j_\bx=0,\ldots, J$, and 
$V_{j_\bx-1}^\bx(\Omega)\otimes V_{J-j_\bx}^t(0,T)$ with $j_\bx=1,\ldots, J$. The $2J+1$ solutions~$P_{j_{\bx},j_t}u$ are then combined with $\pm1$ coefficients as in formula~\eqref{eq:19}. The coarsest and finest meshes used have mesh parameters $h_0^\bx$ and $h_0^\bx2^{-J}$ in space, respectively, and $h_0^t$ and $h_0^t2^{-J}$ in time.
Figure~\ref{fig:sparse_combination} illustrates the combination of contributions from the different space and time resolution levels; see also Figure~\ref{fig:MeshesEx1} below for the meshes used in a concrete example. 

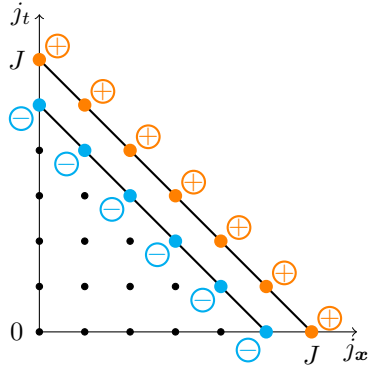
\begin{figure}[ht]\centering
\begin{tikzpicture}[scale=.6]

\draw[->](0,0)--(7,0); 
\draw[->](0,0)--(0,7); 
\draw[thick](0,6)--(6,0);
\draw[thick](0,5)--(5,0);

\draw(-.5,0)node{$0$};
\draw(-.5,6)node{$J$};
\draw(6,-0.5)node{$J$};
\draw(7,-.4)node{$j_{\boldsymbol{x}}$};
\draw(-.4,7)node{$j_t$};

\draw[fill](0,0)circle(.07);
\draw[fill](1,0)circle(.07);
\draw[fill](2,0)circle(.07);
\draw[fill](3,0)circle(.07);
\draw[fill](4,0)circle(.07);
\draw[fill](0,1)circle(.07);
\draw[fill](1,1)circle(.07);
\draw[fill](2,1)circle(.07);
\draw[fill](3,1)circle(.07);
\draw[fill](0,2)circle(.07);
\draw[fill](1,2)circle(.07);
\draw[fill](2,2)circle(.07);
\draw[fill](0,3)circle(.07);
\draw[fill](1,3)circle(.07);
\draw[fill](0,4)circle(.07);

\draw[ultra thick,fill,cyan](5,0)circle(.1);
\draw[thick,cyan](4.6,-0.4)circle(.25)node{$-$};
\draw[ultra thick,fill,cyan](4,1)circle(.1);
\draw[thick,cyan](3.6,0.7)circle(.25)node{$-$};
\draw[ultra thick,fill,cyan](3,2)circle(.1);
\draw[thick,cyan](2.6,1.7)circle(.25)node{$-$};
\draw[ultra thick,fill,cyan](2,3)circle(.1);
\draw[thick,cyan](1.6,2.7)circle(.25)node{$-$};
\draw[ultra thick,fill,cyan](1,4)circle(.1);
\draw[thick,cyan](0.6,3.7)circle(.25)node{$-$};
\draw[ultra thick,fill,cyan](0,5)circle(.1);
\draw[thick,cyan](-0.4,4.7)circle(.25)node{$-$};

\draw[ultra thick,fill,orange](6,0)circle(.1);
\draw[thick,orange](6.4,0.3)circle(.25)node{$+$};
\draw[ultra thick,fill,orange](5,1)circle(.1);
\draw[thick,orange](5.4,1.3)circle(.25)node{$+$};
\draw[ultra thick,fill,orange](4,2)circle(.1);
\draw[thick,orange](4.4,2.3)circle(.25)node{$+$};
\draw[ultra thick,fill,orange](3,3)circle(.1);
\draw[thick,orange](3.4,3.3)circle(.25)node{$+$};
\draw[ultra thick,fill,orange](2,4)circle(.1);
\draw[thick,orange](2.4,4.3)circle(.25)node{$+$};
\draw[ultra thick,fill,orange](1,5)circle(.1);
\draw[thick,orange](1.4,5.3)circle(.25)node{$+$};
\draw[ultra thick,fill,orange](0,6)circle(.1);
\draw[thick,orange](0.4,6.3)circle(.25)node{$+$};

\end{tikzpicture}
\caption{Diagram of the combination technique \eqref{eq:19}.}
\label{fig:sparse_combination}
\end{figure}

\noindent 
For $j_{\bx}, j_t \ge 0$, we define the \emph{space--time detail} operator~$\Delta_{j_{\bx},j_t}^P : \mathcal{V}(Q_T) \to V_{j_{\bx}}^{\bx}(\Omega) \otimes V_{j_t}^t(0,T)$ as 
\begin{equation} \label{eq:20}
    \Delta_{j_{\bx},j_t}^P u := (P_{j_{\bx},j_t} - P_{j_{\bx}-1,j_t}) u - (P_{j_{\bx},j_t-1} - P_{j_{\bx}-1,j_t-1}) u,
\end{equation}
with the convention that $P_{j_{\bx},-1} u = P_{-1,j_t} u = P_{-1,-1} u = 0$. Consequently, $u_J^{CF}$ in~\eqref{eq:19} can be equivalently expressed as
\begin{equation} \label{eq:21}
    u_J^{CF} = \sum_{j_{\bx}+j_t \le J} \Delta_{j_{\bx},j_t}^P u.
\end{equation}
For the full-grid solution~$P_{J,J}u$, we have
\begin{equation} \label{eq:21bis}
    P_{J,J}u= \sum_{j_{\bx},j_t \le J} \Delta_{j_{\bx},j_t}^P u, \qquad P_{J,J}u - u_J^{CF} = \sum_{\substack{j_{\bx},j_t \le J  \\ j_{\bx}+j_t > J}} \Delta_{j_{\bx},j_t}^P u.
\end{equation}
These relations are represented in Figure~\ref{fig:hierarchical}.
Each node $(\jx,\jt)$ represents the Galerkin projection~$P_{\jx,\jt}u$, and the $\pm$ signs next to them indicate the~$\pm1$ coefficients with which~$P_{\jx,\jt}u$ enters the stencil~\eqref{eq:20} of the adjacent space--time detail operators.  For the full-grid approximation, when summing over all $\jx,\jt \leq J$, all nodal contributions cancel except for~$P_{J,J}u$, which has coefficient $+1$ (see Figure \ref{fig:hierarchical}, left). For the sparse-grid approximation, when summing over $\jx+\jt \leq J$, all nodal contributions cancel except those on the first lower diagonal $\jx+\jt=J-1$ with coefficient $-1$, and those on the diagonal $\jx+\jt=J$ with coefficient $+1$, giving the combination formula~\eqref{eq:19} (see Figure \ref{fig:hierarchical}, right).

\begin{figure}[ht]
\centering
\begin{minipage}{0.4\textwidth}
\centering
\begin{tikzpicture}[scale=.7]

\draw[->](0,0)--(7,0); 
\draw[->](0,0)--(0,7); 
\draw[-](0,1.5)--(6,1.5); 
\draw[-](0,3)--(6,3); 
\draw[-](0,4.5)--(6,4.5); 
\draw[-](0,6)--(6,6); 
\draw[-](1.5,0)--(1.5,6); 
\draw[-](3,0)--(3,6); 
\draw[-](4.5,0)--(4.5,6); 
\draw[-](6,0)--(6,6); 

\draw[-,dashed](0,0)--(0,-1.5); 
\draw[-,dashed](1.5,0)--(1.5,-1.5); 
\draw[-,dashed](3,0)--(3,-1.5); 
\draw[-,dashed](4.5,0)--(4.5,-1.5); 
\draw[-,dashed](6,0)--(6,-1.5); 
\draw[-,dashed](0,0)--(-1.5,0); 
\draw[-,dashed](0,1.5)--(-1.5,1.5); 
\draw[-,dashed](0,3)--(-1.5,3); 
\draw[-,dashed](0,4.5)--(-1.5,4.5); 
\draw[-,dashed](0,6)--(-1.5,6); 

\draw(7,-.4)node{$j_{\boldsymbol{x}}$};
\draw(-.4,7)node{$j_t$};

\draw[fill](0,0)circle(.07);
\draw[thick](0.2,0.2)node{$+$};
\draw[fill](1.5,0)circle(.07);
\draw[thick](1.3,0.2)node{$-$};
\draw[thick](1.7,0.2)node{$+$};
\draw[fill](3,0)circle(.07);
\draw[thick](2.8,0.2)node{$-$};
\draw[thick](3.2,0.2)node{$+$};
\draw[fill](4.5,0)circle(.07);
\draw[thick](4.3,0.2)node{$-$};
\draw[fill](6,0)circle(.07);

\draw[fill](0,1.5)circle(.07);
\draw[thick](0.2,1.3)node{$-$};
\draw[fill](1.5,1.5)circle(.07);
\draw[thick](1.3,1.3)node{$+$};
\draw[thick](1.7,1.7)node{$+$};
\draw[thick](1.3,1.7)node{$-$};
\draw[thick](1.7,1.3)node{$-$};
\draw[thick](0.2,1.7)node{$+$};
\draw[thick](2.8,1.7)node{$-$};
\draw[fill](3,1.5)circle(.07);
\draw[thick](2.8,1.3)node{$+$};
\draw[thick](3.2,1.3)node{$-$};
\draw[fill](4.5,1.5)circle(.07);
\draw[thick](4.3,1.3)node{$+$};

\draw[fill](0,3)circle(.07);
\draw[thick](0.2,2.8)node{$-$};
\draw[thick](0.2,3.2)node{$+$};
\draw[fill](1.5,3)circle(.07);
\draw[thick](1.3,2.8)node{$+$};
\draw[thick](1.7,2.8)node{$-$};
\draw[thick](2.8,2.8)node{$+$};
\draw[fill](3,3)circle(.07);

\draw[fill](0,4.5)circle(.07);
\draw[thick](1.3,4.3)node{$+$};
\draw[thick](1.3,3.2)node{$-$};
\draw[fill](1.5,4.5)circle(.07);
\draw[thick](0.2,4.3)node{$-$};

\draw[fill](0,6)circle(.07);

\draw[fill](6,6)circle(.07);
\draw[fill](6,4.5)circle(.07);
\draw[fill](6,3)circle(.07);
\draw[fill](6,1.5)circle(.07);
\draw[fill](4.5,6)circle(.07);
\draw[fill](4.5,4.5)circle(.07);
\draw[fill](4.5,3)circle(.07);
\draw[fill](3,6)circle(.07);
\draw[fill](3,4.5)circle(.07);
\draw[fill](1.5,6)circle(.07);

\draw[thick](5.8,-0.2)node{$+$};
\draw[thick](4.3,-0.2)node{$+$};
\draw[thick](4.7,-0.2)node{$-$};
\draw[thick](3.2,-0.2)node{$-$};
\draw[thick](2.8,-0.2)node{$+$};
\draw[thick](1.7,-0.2)node{$-$};
\draw[thick](1.3,-0.2)node{$+$};
\draw[thick](0.2,-0.2)node{$-$};
\draw[thick](-0.2,-0.2)node{$+$};
\draw[thick](-0.2,1.3)node{$+$};
\draw[thick](-0.2,0.2)node{$-$};
\draw[thick](-0.2,2.8)node{$+$};
\draw[thick](-0.2,1.7)node{$-$};
\draw[thick](-0.2,4.3)node{$+$};
\draw[thick](-0.2,4.7)node{$-$};
\draw[thick](-0.2,3.2)node{$-$};
\draw[thick](-0.2,5.8)node{$+$};

\draw[thick](1.3,4.7)node{$-$};
\draw[thick](1.3,5.8)node{$+$};
\draw[thick](0.2,4.7)node{$+$};
\draw[thick](0.2,5.8)node{$-$};
\draw[thick](1.7,4.7)node{$+$};
\draw[thick](1.7,5.8)node{$-$};
\draw[thick](2.8,4.7)node{$-$};
\draw[thick](2.8,5.8)node{$+$};
\draw[thick](1.7,4.3)node{$-$};
\draw[thick](1.7,3.2)node{$+$};
\draw[thick](2.8,4.3)node{$+$};
\draw[thick](2.8,3.2)node{$-$};

\draw[thick](3.2,5.8)node{$-$};
\draw[thick](3.2,4.7)node{$+$};
\draw[thick](3.2,4.3)node{$-$};
\draw[thick](3.2,3.2)node{$+$};
\draw[thick](3.2,2.8)node{$-$};
\draw[thick](3.2,1.7)node{$+$};

\draw[thick](4.3,5.8)node{$+$};
\draw[thick](4.3,4.7)node{$-$};
\draw[thick](4.3,4.3)node{$+$};
\draw[thick](4.3,3.2)node{$-$};
\draw[thick](4.3,2.8)node{$+$};
\draw[thick](4.3,1.7)node{$-$};

\draw[thick](4.7,5.8)node{$-$};
\draw[thick](4.7,4.7)node{$+$};
\draw[thick](4.7,4.3)node{$-$};
\draw[thick](4.7,3.2)node{$+$};
\draw[thick](4.7,2.8)node{$-$};
\draw[thick](4.7,1.7)node{$+$};
\draw[thick](4.7,1.3)node{$-$};
\draw[thick](4.7,0.2)node{$+$};

\draw[thick](5.8,5.8)node{$+$};
\draw[thick](5.8,4.7)node{$-$};
\draw[thick](5.8,4.3)node{$+$};
\draw[thick](5.8,3.2)node{$-$};
\draw[thick](5.8,2.8)node{$+$};
\draw[thick](5.8,1.7)node{$-$};
\draw[thick](5.8,1.3)node{$+$};
\draw[thick](5.8,0.2)node{$-$};

\draw[thick](0.75,0.75)node{$\Delta^P_{1,1}$};
\draw[thick](2.25,0.75)node{$\Delta^P_{2,1}$};
\draw[thick](3.75,0.75)node{$\Delta^P_{3,1}$};
\draw[thick](0.75,2.25)node{$\Delta^P_{2,1}$};
\draw[thick](0.75,3.75)node{$\Delta^P_{3,1}$};
\draw[thick](2.25,2.25)node{$\Delta^P_{2,2}$};
\draw[thick](0.75,-0.75)node{$\Delta^P_{1,0}$};
\draw[thick](2.25,-0.75)node{$\Delta^P_{2,0}$};
\draw[thick](3.75,-0.75)node{$\Delta^P_{3,0}$};
\draw[thick](5.25,-0.75)node{$\Delta^P_{4,0}$};
\draw[thick](-0.75,-0.75)node{$\Delta^P_{0,0}$};
\draw[thick](-0.75,0.75)node{$\Delta^P_{0,1}$};
\draw[thick](-0.75,2.25)node{$\Delta^P_{0,2}$};
\draw[thick](-0.75,3.75)node{$\Delta^P_{0,3}$};
\draw[thick](-0.75,5.25)node{$\Delta^P_{0,4}$};

\draw[thick](0.75,5.25)node{$\Delta^P_{4,1}$};
\draw[thick](2.25,5.25)node{$\Delta^P_{4,2}$};
\draw[thick](3.75,5.25)node{$\Delta^P_{4,3}$};
\draw[thick](5.25,5.25)node{$\Delta^P_{4,4}$};
\draw[thick](2.25,3.75)node{$\Delta^P_{3,2}$};
\draw[thick](3.75,3.75)node{$\Delta^P_{3,3}$};
\draw[thick](5.25,3.75)node{$\Delta^P_{4,3}$};
\draw[thick](3.75,2.25)node{$\Delta^P_{3,2}$};
\draw[thick](5.25,2.25)node{$\Delta^P_{4,2}$};
\draw[thick](5.25,0.75)node{$\Delta^P_{4,1}$};
\draw[ultra thick,fill,orange](6,6)circle(.1);

\end{tikzpicture}
\end{minipage}
\begin{minipage}{0.4\textwidth}
\centering
\begin{tikzpicture}[scale=.7]
\draw[->](0,0)--(7,0); 
\draw[->](0,0)--(0,7); 
\draw[-](0,1.5)--(4.5,1.5); 
\draw[-](0,3)--(3,3); 
\draw[-](0,4.5)--(1.5,4.5); 
\draw[-](1.5,0)--(1.5,4.5); 
\draw[-](3,0)--(3,3); 
\draw[-](4.5,0)--(4.5,1.5); 

\draw[-,dashed](0,0)--(0,-1.5); 
\draw[-,dashed](1.5,0)--(1.5,-1.5); 
\draw[-,dashed](3,0)--(3,-1.5); 
\draw[-,dashed](4.5,0)--(4.5,-1.5); 
\draw[-,dashed](6,0)--(6,-1.5); 
\draw[-,dashed](0,0)--(-1.5,0); 
\draw[-,dashed](0,1.5)--(-1.5,1.5); 
\draw[-,dashed](0,3)--(-1.5,3); 
\draw[-,dashed](0,4.5)--(-1.5,4.5); 
\draw[-,dashed](0,6)--(-1.5,6); 

\draw(7,-.4)node{$j_{\boldsymbol{x}}$};
\draw(-.4,7)node{$j_t$};

\draw[fill](0,0)circle(.07);
\draw[thick](0.2,0.2)node{$+$};
\draw[fill](1.5,0)circle(.07);
\draw[thick](1.3,0.2)node{$-$};
\draw[thick](1.7,0.2)node{$+$};
\draw[fill](3,0)circle(.07);
\draw[thick](2.8,0.2)node{$-$};
\draw[thick](3.2,0.2)node{$+$};
\draw[fill](4.5,0)circle(.07);
\draw[thick](4.3,0.2)node{$-$};
\draw[fill](6,0)circle(.07);

\draw[fill](0,1.5)circle(.07);
\draw[thick](0.2,1.3)node{$-$};
\draw[fill](1.5,1.5)circle(.07);
\draw[thick](1.3,1.3)node{$+$};
\draw[thick](1.7,1.7)node{$+$};
\draw[thick](1.3,1.7)node{$-$};
\draw[thick](1.7,1.3)node{$-$};
\draw[thick](0.2,1.7)node{$+$};
\draw[thick](2.8,1.7)node{$-$};
\draw[fill](3,1.5)circle(.07);
\draw[thick](2.8,1.3)node{$+$};
\draw[thick](3.2,1.3)node{$-$};
\draw[fill](4.5,1.5)circle(.07);
\draw[thick](4.3,1.3)node{$+$};

\draw[fill](0,3)circle(.07);
\draw[thick](0.2,2.8)node{$-$};
\draw[thick](0.2,3.2)node{$+$};
\draw[fill](1.5,3)circle(.07);
\draw[thick](1.3,2.8)node{$+$};
\draw[thick](1.7,2.8)node{$-$};
\draw[thick](2.8,2.8)node{$+$};
\draw[fill](3,3)circle(.07);

\draw[fill](0,4.5)circle(.07);
\draw[thick](1.3,4.3)node{$+$};
\draw[thick](1.3,3.2)node{$-$};
\draw[fill](1.5,4.5)circle(.07);
\draw[thick](0.2,4.3)node{$-$};

\draw[fill](0,6)circle(.07);

\draw[thick](5.8,-0.2)node{$+$};
\draw[thick](4.3,-0.2)node{$+$};
\draw[thick](4.7,-0.2)node{$-$};
\draw[thick](3.2,-0.2)node{$-$};
\draw[thick](2.8,-0.2)node{$+$};
\draw[thick](1.7,-0.2)node{$-$};
\draw[thick](1.3,-0.2)node{$+$};
\draw[thick](0.2,-0.2)node{$-$};
\draw[thick](-0.2,-0.2)node{$+$};
\draw[thick](-0.2,1.3)node{$+$};
\draw[thick](-0.2,0.2)node{$-$};
\draw[thick](-0.2,2.8)node{$+$};
\draw[thick](-0.2,1.7)node{$-$};
\draw[thick](-0.2,4.3)node{$+$};
\draw[thick](-0.2,4.7)node{$-$};
\draw[thick](-0.2,3.2)node{$-$};
\draw[thick](-0.2,5.8)node{$+$};

\draw[ultra thick,fill,orange](0,6)circle(.1);
\draw[ultra thick,fill,orange](1.5,4.5)circle(.1);
\draw[ultra thick,fill,orange](3,3)circle(.1);
\draw[ultra thick,fill,orange](4.5,1.5)circle(.1);
\draw[ultra thick,fill,orange](6,0)circle(.1);

\draw[ultra thick,fill,cyan](0,4.5)circle(.1);
\draw[ultra thick,fill,cyan](1.5,3)circle(.1);
\draw[ultra thick,fill,cyan](3,1.5)circle(.1);
\draw[ultra thick,fill,cyan](4.5,0)circle(.1);

\draw[thick](0.75,0.75)node{$\Delta^P_{1,1}$};
\draw[thick](2.25,0.75)node{$\Delta^P_{2,1}$};
\draw[thick](3.75,0.75)node{$\Delta^P_{3,1}$};
\draw[thick](0.75,2.25)node{$\Delta^P_{2,1}$};
\draw[thick](0.75,3.75)node{$\Delta^P_{3,1}$};
\draw[thick](2.25,2.25)node{$\Delta^P_{2,2}$};
\draw[thick](0.75,-0.75)node{$\Delta^P_{1,0}$};
\draw[thick](2.25,-0.75)node{$\Delta^P_{2,0}$};
\draw[thick](3.75,-0.75)node{$\Delta^P_{3,0}$};
\draw[thick](5.25,-0.75)node{$\Delta^P_{4,0}$};
\draw[thick](-0.75,-0.75)node{$\Delta^P_{0,0}$};
\draw[thick](-0.75,0.75)node{$\Delta^P_{0,1}$};
\draw[thick](-0.75,2.25)node{$\Delta^P_{0,2}$};
\draw[thick](-0.75,3.75)node{$\Delta^P_{0,3}$};
\draw[thick](-0.75,5.25)node{$\Delta^P_{0,4}$};
\end{tikzpicture}
\end{minipage}
\caption{Graphical representations of the full-grid approximation $P_{4,4} u=\sum_{j_{\bx},j_t \le 4} \Delta_{\jx,\jt}^P u$ (left) and the sparse-grid approximation $u_4^{CF} =\sum_{j_{\bx}+j_t \le 4} \Delta_{\jx,\jt}^P u$ (right).}
\label{fig:hierarchical}
\end{figure}
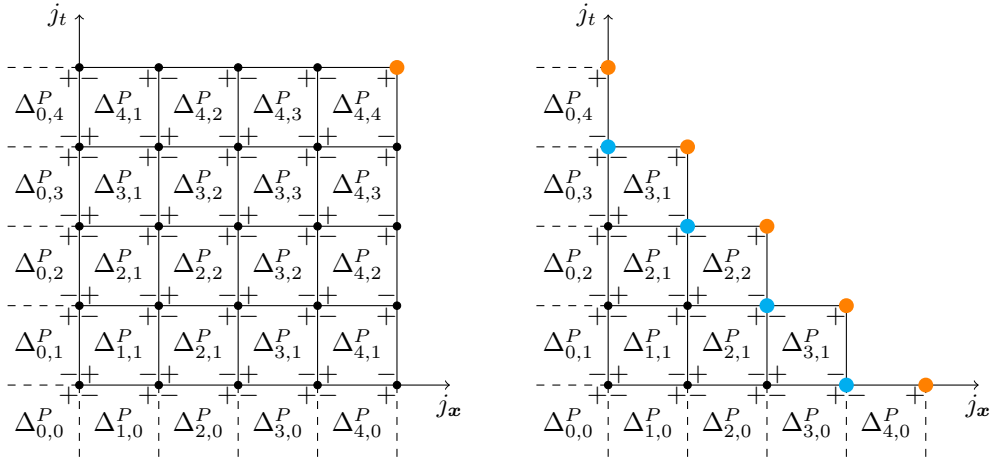

\subsection{Complexity analysis: DoFs count}\label{sec:complexity}

We compare the computational complexity of full- and sparse-grid discretizations in terms of the total number of DoFs. We measure the complexity of the sparse-grid version of the method by summing the dimensions of the~$2J+1$ linear systems in \eqref{eq:19}:
$$
N_{\mathrm{DoFs}}^{\mathrm{sparse}}:=\sum_{j_{\bx}=0}^{J} \mathrm{dim}\big(V_{\jx}^{\bx}(\Omega) \otimes V^t_{J-\jx}(0,T)\big) + \sum_{j_{\bx}=1}^{J} \mathrm{dim}\big(V_{\jx-1}^{\bx}(\Omega) \otimes V^t_{J-\jx}(0,T)\big);
$$
see also Remark~\ref{rem:DoFCount}. We consider separately the cases of space dimension $d=1$ and $d=2,3$.

\paragraph{Case $d=1$.} Let ~$\Omega \subset \mathbb{R}$ be a one-dimensional spatial  domain, and let~$J \in \mathbb{N}$ be the finest refinement level. For $\jx, \jt\le J$, we choose $\Vx \subset H^1_0(\Omega)$ and $\Vt \subset H^2_{0,\bullet}(0,T)$ as the spaces generated by B-splines on uniform meshes with polynomial degrees $p_{\bx}$ and $p_t$ and regularities 
$0\le r_{\bx}\le p_{\bx}-1$ and $1\le r_t\le p_t-1$ in space and time, respectively. This choice ensures conformity of the approximation, namely $C^0$ continuity in space and $C^1$ continuity in time. For a given refinement level $(j_{\bx}, j_t)$, the dimension of the tensor-product space $\Vx \otimes\Vt$ is
\begin{equation*}
    \mathrm{dim}(\Vx \otimes\Vt)=(N_{\jx}(p_{\bx}-r_{\bx})+r_{\bx}-1)(N_{\jt}(p_t-r_t)+r_t),
\end{equation*}
where $N_{\jx}\approx 2^{\jx}$ and $N_{\jt}\approx 2^{\jt}$ are the number of spatial and temporal intervals, respectively. The number of DoFs of the full space--time approximation at the finest level $J$ is
$$
    N_{\mathrm{DoFs}}^{\mathrm{full}}:=\mathrm{dim}(V_{J}^{\bx}(\Omega) \otimes V_J^t(0,T))\approx 2^{2J},
$$
while the total number of DoFs for the sparse-grid space--time approximation satisfies
\begin{align*}
    N_{\mathrm{DoFs}}^{\mathrm{sparse}} & \approx\sum_{j_{\bx}=0}^{J} 2^{j_{\bx}+(J-j_{\bx})}+ \sum_{j_{\bx}=1}^{J} 2^{j_{\bx}-1+(J-j_{\bx})} =2^J (J+1+J 2^{-1}) \approx 2^J J.
\end{align*}
\paragraph{Case $d=2,3$.}
In the case of spatial dimension $d\in\{2,3\}$ we consider spatial discretizations using~$C^0$ finite elements on sequences of shape-regular, simplicial nested meshes obtained by successive mesh halving. Namely, for~$\jx\le J$, $\Vx \subset H^1(\Omega)$ is the space of continuous, piecewise polynomials of degree $p_{\bx}$ defined on the mesh at level~$\jx$. The dimension of~$\Vx$ scales like~$\mathcal{O}(2^{d\jx})$. In time, as above, we consider the spaces generated by B-splines on uniform meshes with degree $p_t$ and regularity $1\le r_t\le p_t-1$,
whose dimension scales like $\mathcal{O}(2^{\jt})$. For the full-grid approximation with fine level $J$, we have
\[
    N_{\mathrm{DoFs}}^{\mathrm{full}} \approx  2^{dJ+J}=2^{(d+1)J},
\]
while for the sparse-grid approximation, we obtain 
\begin{align*}
    N_{\mathrm{DoFs}}^{\mathrm{sparse}} &\approx \sum_{j_{\bx}=0}^{J} 2^{d j_{\bx}+(J-j_{\bx})}+ \sum_{j_{\bx}=1}^{J} 2^{d(j_{\bx}-1)+(J-j_{\bx})} = \frac{(1-2^{(d-1)(J+1)})}{1-2^{d-1}}(2^J+2^{J-d})-2^{J-d} \approx 2^{dJ}.
\end{align*}
Hence, the sparse-grid discretization requires significantly fewer degrees of freedom than the full-grid scheme.
\begin{remark}[DoFs count and solution cost]\label{rem:DoFCount} The full-grid scheme leads to a single algebraic linear system of size $\mathcal{O}(2^{(d+1)J})$, whereas the sparse-grid discretization results in $2J+1$ smaller linear systems of different sizes. The total number of degrees of freedom $N_{\mathrm{DoFs}}^{\mathrm{sparse}}=\mathcal O(2^{dJ})$ computed above corresponds to the sum of their dimensions. Thus, if the cost of each linear system solution grows more than linearly in its dimension, the combination formula \eqref{eq:19} is advantageous against a full-grid scheme even for the same number of DoFs. Additionally, the combination technique naturally allows for parallel implementation: the individual (small) systems can be solved independently, further enhancing the efficiency of the method compared to the full-grid approach. 
\eremk
\end{remark}

\subsection{Convergence}\label{sec:convergenceSG}

We present the main result of this paper, namely convergence estimates for the combination formula in the~$L^2$ norm. The proof is deferred to Section~\ref{sec:proof}, while the auxiliary results on space and time semidiscretizations required for the proof are collected in Section~\ref{sec:semidiscr}.
\begin{theorem} \label{th:1}
Let $p_\bx,p_t\in\mathbb N$.
Assume that the problem domain and coefficient satisfy
\begin{equation} \label{eq:22}
    \partial \Omega \in C^{p_{\bx},1}, \quad c \in H^{p_{\bx}}(\Omega) \cap C^{p_{\bx}-1,1}(\overline{\Omega}).
\end{equation}
Let~$u$ be the solution of problem~\eqref{eq:7} and suppose that it satisfies the mixed regularity assumption
\begin{equation} \label{eq:23}
    u \in \big(H_0^{p_{\bx}+4}(\Omega)\otimes  H^{p_t+3}(0,T)\big)\cap \big(H^{p_{\bx}+1}(\Omega)\otimes  H_{0,\bullet}^{p_t+6}(0,T)\big).
\end{equation}
Assume that the families of approximation spaces $\{\Vx\}_{j_{\bx}}$ and $\{\Vt\}_{j_t}$ satisfy~\eqref{eq:13} and~\eqref{eq:14}, respectively, for approximation degrees $p_{\bx}, p_t \in \mathbb N$.
Assume, in addition, that the following \vanishingconditions at the initial time and on the boundary hold:
\begin{equation} \label{eq:24}
    \partial_t^{\ell_t} ((\Id^{\bx}-\Pi^{\nabla_{\bx}}_{j_{\bx}})\otimes \Id^t) u(\bx,0) =0 \quad \text{for a.e.~} \bx \in \Omega \quad \text{for} \quad \ell_t=1,\dots,p_t+3,
\end{equation}
and 
\begin{equation} \label{eq:25}
    \mathcal{L}^{\ell_{\bx}}(\Id^{\bx} \otimes (\Id^t - \Pi_{j_t}^{\partial_t^2}))u (\bx,t)=0 \quad \text{for a.e.~} (\bx,t) \in \partial \Omega \times (0,T) \quad \text{~for~} \ell_{\bx} \in \mathbb{N} \text{~such that~} 1 \le \ell_{\bx}< (p_{\bx}+4)/2,
\end{equation}
where $\mathcal{L} = - \nabla_{\bx} \cdot (c^2 \nabla_{\bx} \cdot)$, and $\mathcal{L}^{\ell_{\bx}}$ denotes the $\ell_{\bx}$-fold application of $\mathcal{L}$.

\noindent
Then, for~$u_J^{CF}$ obtained with the combination formula \eqref{eq:19}, the following error estimates hold:
\begin{align*}
    \| u - u^{CF}_J \|_{L^2(Q_T)} \lesssim  \big(\|u\|_{H^{p_{\bx}+4}(\Omega)\otimes H^{p_t+3}(0,T)}+\|u\|_{H^{p_{\bx}+1}(\Omega)\otimes H^{p_t+6}(0,T)}\big) \begin{cases}
     2^{-J\min\{p_{\bx}+1,p_t+1\}}, &\text{if } p_{\bx}\neq p_t, \\
    2^{-J(p_t+1)} J,  &\text{if } p_{\bx} =  p_t.
    \end{cases}
\end{align*}
\end{theorem}
\begin{remark}[Mixed regularity assumption] \label{rem:4} The mixed regularity assumption in \eqref{eq:23} is satisfied whenever the data $f, u_0$, and $v_0$ are sufficiently smooth. In fact, for a fixed~$m\in {\mathbb N}$, \cite[\S 7.2.3, Theorem 6, P. 412]{evans1998partial} gives the following result. Assume $\partial \Omega$ smooth and
\begin{equation*}
    \begin{cases}
    u_0\in H^{m+1}(\Omega), \quad v_0\in H^{m}(\Omega),
    \\ f\in H^{m-k}(\Omega) \otimes H^k(0,T) \quad \text{for all~} k=0,\dots,m.
    \end{cases}
\end{equation*}
Suppose also that the following $m^{th}$-order compatibility conditions hold:
\begin{equation*}
    \begin{cases}
    g_0:=u_0\in H^1_0(\Omega), \quad g_1:=v_0\in H^1_0(\Omega),
    \\ g_\ell := \partial_t^{\ell-2}f(\cdot,0)-\mathcal{L} g_{\ell-2} \in H^{1}_0(\Omega) \quad \ell = 2,\ldots,m.
    \end{cases}
\end{equation*}
Then, the solution~$u$ of~\eqref{eq:7} satisfies~$u\in H^{m+1-k}(\Omega)\otimes H^k(0,T)$ for all $k=0,\dots, m+1$. Therefore, if the above conditions are satisfied with~$m=p_{\bx}+p_t+6$, we have $u \in H^{p_{\bx}+4}(\Omega) \otimes H^{p_t+3}(0,T)$ and $u\in H^{p_{\bx}+1}(\Omega) \otimes H^{p_t+6}(0,T)$.

\noindent 
The impact of mixed regularity on the convergence rates of the combination formula for elliptic problems is discussed in~\cite[Table 1]{BeckSangalliTamellini2018}.
\eremk
\end{remark}
\begin{remark}[\Vanishingconditions \eqref{eq:24} and~\eqref{eq:25}]
Table~\ref{tab:vanishing} shows the number of \vanishingconditions in~\eqref{eq:24} and~\eqref{eq:25} for a few representative $(p_{\bx},p_t)$ cases. These conditions require that the projection errors satisfy suitable vanishing constraints at the initial time and on the boundary. This allows the elimination of boundary and initial contributions in the error analysis and ensures optimal convergence rates. Using the commutativity of $\Pi^{\nabla_{\bx}}_{\jx}\otimes \Id^t$ with $\partial_t^{\ell_t}$ and of $\Id^{\bx}\otimes \Pi^{\partial_t^2}_{\jt}$ with $\mathcal{L}^{\ell_x}$, the \vanishingconditions \eqref{eq:24} and~\eqref{eq:25} can be equivalently rewritten as
\begin{align}\label{eq:VanishCondInV}
\begin{aligned}
    \partial_t^{\ell_t} u(\cdot,0) &\in \Vx \quad &&\text{for} \quad \ell_t=1,\dots,p_t+3,\\
    \mathcal{L}^{\ell_{\bx}} u (\bx,\cdot)&\in \Vt \quad \text{for a.e.~} \bx \in \partial \Omega &&\text{for} \quad \ell_{\bx} \in \mathbb{N} \text{~such that~} 1 \le \ell_{\bx}< (p_{\bx}+4)/2,
\end{aligned}
\end{align}
respectively. We now write explicitly the functions $\partial_t^{\ell_t} u$ and $\mathcal{L}^{\ell_{\bx}} u$ to gives a better understanding of these conditions. Since $\mathcal{L}$ and $\partial_t$ commute and $u$ satisfies $\partial_t^2 u+\mathcal{L}u=f$,
higher-order time derivatives of $u$ are obtained by iteratively differentiating the wave equation with respect to time:
\begin{align*}
    \partial_t^{\ell_t} u =
    \begin{cases}
        \partial_t u& \ell_t=1,
        \\ f-\mathcal{L}u & \ell_t=2,
        \\ \partial_t f-\mathcal{L}\partial_t u & \ell_t=3,
        \\ (\partial_t^2 -\mathcal{L})f+\mathcal{L}^2u & \ell_t=4,
        \\ \dots
    \end{cases}
    \hspace{1cm}
    \partial_t^{\ell_t}u(\cdot,0)=
    \begin{cases}
        v_0 &\ell_t=1,
        \\ f(\cdot,0)-\mathcal{L}u_0 & \ell_t=2,
        \\ \partial_t f(\cdot,0)-\mathcal{L}v_0 & \ell_t=3,
        \\ (\partial_t^2 -\mathcal{L})f(\cdot,0)+\mathcal{L}^2 u_0 & \ell_t=4,
        \\ \dots
    \end{cases}
\end{align*}
Applying $\ell_{\bx}$ times the operator $\mathcal{L}$ to the wave equation $\partial_t^2 u+\mathcal{L}u=f$ and using the homogeneous Dirichlet boundary conditions, we obtain, for a.e. $\bx \in \partial \Omega$
\begin{align*}
    \mathcal{L}^{\ell_{\bx}}u =
    \begin{cases}
        f-\partial_t^2 u & \ell_{\bx}=1,
        \\ (\mathcal{L} -\partial_t^2 )f +\partial_t^4 u & \ell_{\bx}=2,
        \\ (\mathcal{L}^2-\partial_t^2 \mathcal{L}+\partial_t^4) f -\partial_t^6 u & \ell_{\bx}=3,
        \\ \dots
    \end{cases}
    \hspace{1cm}
    \mathcal{L}^{\ell_{\bx}}u(\bx,\cdot)=
    \begin{cases}
        f(\bx,\cdot) & \ell_{\bx}=1,
        \\ (\mathcal{L} -\partial_t^2 )f(\bx,\cdot) & \ell_{\bx}=2,
        \\ (\mathcal{L}^2-\partial_t^2 \mathcal{L}+\partial_t^4) f(\bx,\cdot)& \ell_{\bx}=3,
        \\ \dots
    \end{cases}
\end{align*}
Comparing with \eqref{eq:VanishCondInV}, these identities show that the \vanishingconditions \eqref{eq:24}--\eqref{eq:25} can be stated in terms of the initial--boundary value problem data: some combinations of the partial derivatives of $f,u_0,v_0$ at initial time and on the space boundary must belong to the discrete spaces $\Vx$ and $\Vt$. Conditions \eqref{eq:24}--\eqref{eq:25} are implied by the following stronger assumptions:
\begin{align*}
\begin{cases}
   v_0(\bx)=0 & \text{for~} \ell_t=1,
   \\ f(\bx,0)-\mathcal{L} u_0(\bx)= 0 & \text{for~} \ell_t=2,
   \\ \partial^{\ell_t-2}_t f(\bx,0)=0  & \text{for~}  \ell_t=3,\dots,p_t+1
\end{cases}
\end{align*}
for a.e. $\bx \in \Omega$, and
\begin{align*}
    \mathcal{L}^{\ell_{\bx}-1} f (\bx,t)=0 \quad \text{for a.e.~} (\bx,t) \in \partial \Omega \times (0,T) \quad \text{~for~} \ell_{\bx} \in \mathbb{N} \text{~such that~} 1\le \ell_{\bx}< (p_{\bx}+4)/2.
\end{align*}
\Vanishingconditions cannot be completely avoided while preserving optimal convergence rates, as illustrated by the numerical experiment in \S\ref{sec:example4} below.
\begin{table}
\centering 
\begin{tabular}{c|cccc}
\diagbox[width=3em,height=2em]{$p_{\bx}$}{$p_t$}
& 2 & 3 & 4 & 5 \\ \hline
1 & $(5,2)$ & $(6,2)$ & $(7,2)$ & $(8,2)$ \\
2 & $(5,2)$ & $(6,2)$ & $(7,2)$ & $(8,2)$ \\
3 & $(5,3)$ & $(6,3)$ & $(7,3)$ & $(8,3)$ \\
4 & $(5,3)$ & $(6,3)$ & $(7,3)$ & $(8,3)$ \\
5 & $(5,4)$ & $(6,4)$ & $(7,4)$ & $(8,4)$
\end{tabular}
\caption{Number of \vanishingconditions \eqref{eq:24} and~\eqref{eq:25} shown as $(N_t,N_{\bx})$, where~$N_t$ denotes the number of initial-time derivative conditions in~\eqref{eq:24} and~$N_{\bx}$ the number of boundary conditions from powers of~$\mathcal{L}$ in~\eqref{eq:25}, as a function of $p_t$ (columns) and $p_{\bx}$ (rows).} 
\label{tab:vanishing}
\end{table}
\eremk
\end{remark}

\noindent
For $p=p_{\bx}=p_t$, the convergence rates in the $L^2(Q_T)$ norm for the full- and sparse-grid approximations are, respectively,
\begin{equation*}
\begin{aligned}
    \|u -P_{J,J} u \|_{L^2(Q_T)} & \approx 2^{-J(p+1)}, \qquad 
    \|u -u_J^{CF} \|_{L^2(Q_T)} & \approx 2^{-J(p+1)} J;
\end{aligned}
\end{equation*}
see~\eqref{eq:18} and Theorem~\ref{th:1}. In terms of the number of degrees of freedom, these rates are expressed as follows.
\paragraph{Case $d=1$.}
Recalling that~$N_{\mathrm{DoFs}}^{\mathrm{full}} \approx 2^{2J}$ and $N_{\mathrm{DoFs}}^{\mathrm{sparse}} \approx 2^{J} J$, we obtain
\begin{equation} \label{eq:26a}
\begin{aligned}
    \|u -P_{J,J} u \|_{L^2(Q_T)} & \approx (2^{2J})^{-\frac{p+1}{2}} \approx  (N_{\mathrm{DoFs}}^{\mathrm{full}})^{-\frac{p+1}{2}},
    \\ \|u -u_J^{CF} \|_{L^2(Q_T)} & \approx (2^J J)^{-(p+1)}J^{p+2} \approx (N_{\mathrm{DoFs}}^{\mathrm{sparse}} )^{-(p+1)}(\log_2 (N_{\mathrm{DoFs}}^{\mathrm{sparse}} )^{p+2}),
\end{aligned}
\end{equation}
where, in the second estimate, we used that $\log_2(N_{\mathrm{DoFs}}^{\mathrm{sparse}})  \approx J+ \log_2 J \approx J$.

\paragraph{Case $d=2,3$.} Recalling that~$N_{\mathrm{DoFs}}^{\mathrm{full}} \approx 2^{(d+1)J}$ and $N_{\mathrm{DoFs}}^{\mathrm{sparse}} \approx 2^{dJ}$, we obtain
\begin{equation} \label{eq:27a}
\begin{aligned}
    \|u -P_{J,J} u\|_{L^2(Q_T)} & \approx (2^{(d+1)J})^{-\frac{p+1}{d+1}} \approx (N^{\mathrm{full}}_{\mathrm{DoFs}})^{-\frac{p+1}{d+1}},
    \\ \|u -u_J^{CF}\|_{L^2(Q_T)} & \approx \left(2^{dJ} \right)^{-\frac{p+1}{d}}J \approx (N^{\mathrm{sparse}}_{\mathrm{DoFs}})^{-\frac{p+1}{d}}(\log_2 (N^{\mathrm{sparse}}_{\mathrm{DoFs}})),
\end{aligned}
\end{equation}
where, in the second estimate, we used that~$\log_2(N^{\mathrm{sparse}}_{\mathrm{DoFs}}) \approx dJ$.

\medskip
\noindent 
Asymptotically for $N_{\mathrm{DoFs}}^{\mathrm{sparse}}\to \infty$ and $N_{\mathrm{DoFs}}^{\mathrm{full}} \to \infty$, for all $d=1,2,3$, we deduce  
\begin{equation*}
    \| u - P_{J,J} u \|_{L^2(Q_T)} \approx (N^{\mathrm{full}}_{\mathrm{DoFs}})^{-\frac{p+1}{d+1}}, \quad\quad\quad  \| u - u_J^{CF} \|_{L^2(Q_T)} \approx (N^{\mathrm{sparse}}_{\mathrm{DoFs}})^{-\frac{p+1}{d}}.
\end{equation*}
These estimates highlight the advantage in terms of number of DoFs of the use of the combination formula.

\section{Semidiscretization operators and their properties}\label{sec:semidiscr}

To study the convergence of the approximation~$u_J^{CF}$ obtained with the combination formula \eqref{eq:19} to the solution $u$ of \eqref{eq:7}, it is essential to first establish the properties of the semidiscretizations $\Pinfjt$ and $\Pjxinf$. These are defined as the solutions to the following semidiscrete problems: given $u \in \mathcal{V}(Q_T)$, find $\Pinfjt \in H^1_0(\Omega) \otimes \Vt$ such that
\begin{equation} \label{eq:26}
    \mathcal{A}(\Pinfjt, w_{\infty,j_t}) = 
    \mathcal{A}(u, w_{\infty,j_t}) \quad \text{for all } w_{\infty,j_t} \in H^1_0(\Omega) \otimes \Vt,
\end{equation}
and find $\Pjxinf \in \Vx \otimes H_{0,\bullet}^2(0,T)$ such that
\begin{equation} \label{eq:27} 
    \mathcal{A}(\Pjxinf, w_{j_{\bx},\infty}) = \mathcal{A}(u, w_{j_{\bx},\infty}) \quad \text{for all } w_{j_{\bx},\infty} \in V_{j_{\bx}}^{\bx}(\Omega) \otimes H_{0,\bullet}^2(0,T).
\end{equation}
This also defines the corresponding semidiscretization operators~$\Pinfjt: \V(Q_T)\to H^1_0(\Omega) \otimes \Vt$ and~$\Pjxinf:\V(Q_T)\to \Vx \otimes H_{0,\bullet}^2(0,T)$, respectively.

\noindent
In the remainder of this section, we show that~$\Pinfjt$ and~$\Pjxinf$ are well defined (Section~\ref{sec:welldef}), and we establish 
stability results under high regularity assumptions (Section~\ref{sec:stab}). We conclude with convergence results for differences of consecutive Galerkin projections (Section~\ref{sec:gal}).

\subsection{Well-definedness and regularity}\label{sec:welldef}
We first establish the well-definedness and regularity-preserving properties of the semidiscretizations~$\Pinfjt$ and~$\Pjxinf$.
\begin{proposition} \label{prop:1}
For any $u \in \V(Q_T)$, the semidiscretizations $\Pinfjt$ and $\Pjxinf$ are well-defined for all $\jt, \jx \in \mathbb{N}$. Furthermore, assuming $s_{\bx} \ge 1$ and $s_t \ge 2$, and
\begin{equation} \label{eq:28}
    \partial \Omega \in C^{s_{\bx}-1,1}, \quad c \in H^{s_{\bx}-1}(\Omega) \cap C^{s_{\bx}-2,1}(\overline{\Omega})
\end{equation}
the following regularity properties hold:
\begin{itemize}[topsep=4pt,itemsep=0pt]
\item[i)] if $u \in \big(H_0^{s_{\bx}}(\Omega) \otimes H^1(0,T)\big)\cap\big(H^{s_{\bx}-2}(\Omega) \otimes H_{0,\bullet}^{2}(0,T)\big)$, then $\Pinfjt \in H_0^{s_{\bx}}(\Omega) \otimes \Vt$; 
\item[ii)] if $u \in \big(H_0^1(\Omega) \otimes H^{s_t-2}(0,T)\big) \cap \big(L^2(\Omega) \otimes H_{0,\bullet}^{s_t}(0,T)\big)$, then $\Pjxinf \in \Vx \otimes H_{0,\bullet}^{s_t}(0,T)$.
\end{itemize}
\end{proposition}
\begin{proof}
We first prove the statement for $P_{\infty,j_t}u$. Let $\{\psi_{j_t}^i\}_{i=1}^{N_{j_t}^t}$ be a basis of $V_{j_t}^t(0,T)$. Then, any element $v_{\infty,j_t} \in H_0^1(\Omega) \otimes V_{j_t}^t(0,T)$ can be expressed as
\begin{equation} \label{eq:29}
    v_{\infty,j_t}(\bx,t) = \sum_{i=1}^{N_{j_t}^t} v_i(\bx) \psi_{j_t}^i(t) \quad \text{with } v_i \in H^1_0(\Omega).
\end{equation}
Indeed, by the definition of the Hilbert tensor product~\cite[Chapter~II, \S~4]{ReedSimon1980}, we can write
\begin{equation*}
    v_{\infty,j_t}(\bx,t) = \sum_{k=1}^\infty z_k(\bx) v_{j_t}^k(t) \quad \text{with } z_k \in H^1_0(\Omega) \text{ and } v_{j_t}^k \in V_{j_t}^t(0,T).
\end{equation*}
Expanding each $v_{j_t}^k \in V_{j_t}^t(0,T)$ as $v_{j_t}^k(t) = \sum_{i=1}^{N_{j_t}^t} v_{j_t,i}^k \psi_{j_t}^i(t)$ with $v_{j_t,i}^k \in \mathbb{R}$, we obtain~\eqref{eq:29}:
\begin{equation*}
    v_{\infty,j_t}(\bx,t) = \sum_{k=1}^{\infty} z_k(\bx) \sum_{i=1}^{N_{j_t}^t} v_{j_t,i}^k \psi_{j_t}^i(t) = \sum_{i=1}^{N_{j_t}^t} v_i(\bx) \psi^i_{j_t}(t),
\end{equation*}
where $v_i(\bx) := \sum_{k=1}^{\infty} v_{j_t,i}^k z_k(\bx)$. In view of \eqref{eq:29}, finding $P_{\infty,j_t}u$ that solves~\eqref{eq:26} is equivalent to finding its coefficient functions $\{u_j\}_{j=1}^{N_{j_t}^t} \subset H_0^1(\Omega)$ in the expansion $P_{\infty,j_t} u(\bx,t) = \sum_{j=1}^{N_{j_t}^t} u_{j}(\bx) \psi_{j_t}^{j}(t)$. Writing the test functions as $w_{\infty,j_t}(\bx,t) = \sum_{i=1}^{N_{j_t}^t} w_{i}(\bx) \psi_{j_t}^{i}(t)$ with $w_{i} \in H^1_0(\Omega)$, the left-hand side of problem \eqref{eq:26} becomes
\begin{equation*}
    \mathcal{A}(P_{\infty,j_t}u, w_{\infty,j_t}) = \sum_{i=1}^{N_{j_t}^t} \sum_{j=1}^{N_{j_t}^t} \left( \mathbf{A}_{j_t}[i,j] (u_j, w_i)_{L^2(\Omega)} + \mathbf{B}_{j_t}[i,j] (c^2 \nabla_{\bx} u_j, \nabla_{\bx} w_i)_{L^2(\Omega)} \right),
\end{equation*}
where the $(N_{j_t}^t\times N_{j_t}^t)$ matrices $\mathbf{A}_{j_t}$ and $\mathbf{B}_{j_t}$ are defined as
\begin{equation*}
\begin{aligned}
    \mathbf{A}_{j_t}[i,j] := ((\psi_{j_t}^j)'', (\psi_{j_t}^i)')_{L^2_e(0,T)} + (\psi_{j_t}^j)'(0) (\psi_{j_t}^i)'(0), \quad
    \mathbf{B}_{j_t}[i,j] := (\psi_{j_t}^j, (\psi_{j_t}^i)')_{L^2_e(0,T)}.
\end{aligned}
\end{equation*}
We thus consider the following reformulation of~\eqref{eq:26}:
find $\mathbf{u}_{j_t} \in [H^1_0(\Omega)]^{N_{j_t}^t}$ such that
\begin{equation} \label{eq:30}
    (\mathbf{A}_{j_t} \mathbf{u}_{j_t}, \mathbf{w})_{L^2(\Omega)} + (\mathbf{B}_{j_t} c^2 \nabla_{\bx} \mathbf{u}_{j_t}, \nabla_{\bx} \mathbf{w})_{L^2(\Omega)} = (\mathbf{F}_{j_t}, \mathbf{w})_{L^2(\Omega)} \quad \text{for all~} \mathbf{w} \in [H^1_0(\Omega)]^{N_{j_t}^t},
\end{equation}
where the source term $\mathbf{F}_{j_t} \in [L^2(\Omega)]^{N_{j_t}^t}$ is defined componentwise as
\begin{equation*}
    \mathbf{F}_{j_t}[i](\bx) := (\partial_t^2 u(\bx,\cdot), (\psi_{j_t}^{i})')_{L^2_e(0,T)} + \partial_t u(\bx,0) (\psi_{j_t}^{i})'(0) - (\nabla_{\bx} \cdot (c^2 \nabla_{\bx} u(\bx,\cdot)), (\psi_{j_t}^{i})')_{L^2_e(0,T)}.
\end{equation*}
If $\mathbf{u}_{j_t}$ solves~\eqref{eq:30}, then $P_{\infty,j_t} u(\bx,t) := \sum_{j=1}^{N_{j_t}^t} u_j(\bx) \psi_{j_t}^j(t)$, with $u_j := \mathbf{u}_{\jt}[j]$, solves~\eqref{eq:26}.

\noindent 
As an immediate consequence of~\eqref{eq:2}, both matrices~$\mathbf{A}_{j_t}$ and $\mathbf{B}_{j_t}$ have positive definite symmetric parts. Thus, by the Lax--Milgram lemma, there exists a unique solution $\mathbf{u}_{\jt} \in [H_0^1(\Omega)]^{N_{j_t}^t}$ to problem~\eqref{eq:30}. This implies existence of~$P_{\infty,j_t}u$ solution of~\eqref{eq:26}. 
The coercivity of $\mathcal{A}$ in \eqref{eq:8} implies that~$P_{\infty,j_t}u$ is unique.
Moreover, if $u$ satisfies the regularity in $i)$ for $s_{\bx} \ge 1$, and $c \in H^{s_{\bx}-1}(\Omega)$, then $\mathbf{F}_{j_t} \in [H^{s_{\bx}-2}(\Omega)]^{N_{j_t}^t}$. Under the assumptions in \eqref{eq:28}, we deduce $\mathbf{u}_{\jt} \in [H_0^{s_{\bx}}(\Omega)]^{N_{j_t}^t}$ (see \cite[Theorem 4.18]{mclean2000strongly}), and thus $P_{\infty,j_t} u \in H_0^{s_{\bx}}(\Omega) \otimes V^t_{j_t}(0,T)$.

\noindent 
We now prove the statement for $P_{j_{\bx},\infty} u$. Let $\{\psi_{j_{\bx}}^i\}_{i=1}^{N_{j_{\bx}}^{\bx}}$ be a basis of $V_{j_{\bx}}^{\bx}(\Omega)$ that is orthonormal in $L^2(\Omega)$ and orthogonal in $H^1_0(\Omega)$ with respect to the weighted product $(c^2 \nabla_{\bx} \cdot, \nabla_{\bx} \cdot)_{L^2(\Omega)}$. We proceed analogously to the first part. Finding~$P_{j_{\bx},\infty}u$ that solves~\eqref{eq:27} is equivalent to finding its coefficient functions $\{u_j\}_{j=1}^{N_{j_{\bx}}^{\bx}} \subset H^2_{0,\bullet}(0,T)$ in the expansion~$P_{j_{\bx},\infty} u(\bx,t) = \sum_{j=1}^{N_{j_{\bx}}^{\bx}} \psi_{j_{\bx}}^j(\bx) u_j(t)$. Writing the test functions $w_{j_{\bx},\infty} \in V_{j_{\bx}}^{\bx}(\Omega) \otimes H^2_{0,\bullet}(0,T)$ as $w_{j_{\bx},\infty}(\bx,t) = \sum_{i=1}^{N_{j_{\bx}}^{\bx}} \psi_{j_{\bx}}^i(\bx)w_i(t)$, where $w_i \in H^2_{0,\bullet}(0,T)$, and exploiting the orthogonality of the basis, formulation~\eqref{eq:27} can be rewritten as
\begin{equation*}
    \sum_{j=1}^{N_{j_{\bx}}^{\bx}} \left( (u_{j}'', w_{j}')_{L^2_e(0,T)} + u_j'(0) w_j'(0) + \mu_{j_{\bx}}^{j} (u_j, w_j')_{L^2_e(0,T)} \right)
    = \sum_{j=1}^{N_{j_{\bx}}^{\bx}} \left( (f_{j_{\bx}}^{j}, w_j')_{L^2_e(0,T)} + (\partial_t u(\cdot,0), \psi_{j_{\bx}}^{j})_{L^2(\Omega)} w_j'(0) \right),
\end{equation*}
where we have set $\mu_{j_{\bx}}^{j} := (c^2 \nabla_{\bx} \psi^{j}_{j_{\bx}}, \nabla_{\bx} \psi^{j}_{j_{\bx}})_{L^2(\Omega)}$, and 
\begin{equation*}
    f_{j_{\bx}}^{j}(t) := (\partial_t^2 u(\cdot,t), \psi^{j}_{j_{\bx}})_{L^2(\Omega)} + (c^2 \nabla_{\bx} u(\cdot,t), \nabla_{\bx} \psi_{j_{\bx}}^{j})_{L^2(\Omega)}.
\end{equation*}
For each $j=1,\dots,N_{j_{\bx}}^{\bx}$, the initial value problem
\begin{equation*}
\begin{cases}
    u_{j}''(t) + \mu_{j_{\bx}}^{j} u_{j}(t) = f^{j}_{j_{\bx}}(t), & t \in [0,T], \\
    u_{j}(0) = 0, \quad u_{j}'(0) = (\partial_t u(\cdot,0), \psi_{j_{\bx}}^j)_{L^2(\Omega)},
\end{cases}
\end{equation*}
admits a unique solution. This ensures the existence of $P_{j_{\bx},\infty}u$. The uniqueness follows from the coercivity of $\mathcal{A}$ in~\eqref{eq:8}. Furthermore, if $u$ satisfies the regularity in $ii)$ for $s_t \ge 2$, then $f_{j_{\bx}}^j \in H^{s_t - 2}(0,T)$. This implies that $u_j \in H_{0,\bullet}^{s_t}(0,T)$ for all $j=1,\dots,N_{j_{\bx}}^{\bx}$, and consequently $P_{j_{\bx},\infty}u \in V_{j_{\bx}}^{\bx}(\Omega) \otimes H_{0,\bullet}^{s_t}(0,T)$.
\end{proof}

Estimates for the semidiscrete errors~$(P_{\infty,j_t}-\Id)u$ and $(P_{j_{\bx},\infty}-\Id)u$, which are not used in the argument leading to Theorem~\ref{th:1}, are derived in Appendix~\ref{sec:errest}. In the following section, we establish stability estimates for~$P_{\infty,j_t}$ and~$P_{j_{\bx},\infty}$, which are used in the proof of convergence of the combination formula in Section~\ref{sec:convergence} below.

\subsection{Stability estimates in higher-order norms}\label{sec:stab}

Consider the spectral decomposition of the elliptic operator $\mathcal{L} = -\nabla_{\bx} \cdot (c^2 \nabla_{\bx}\cdot ) $ with homogeneous Dirichlet boundary conditions, and let $\{(\lambda_m, \phi_m)\}_{m=1}^\infty$ be the eigenpairs such that $\mathcal{L} \phi_m = \lambda_m \phi_m$, with~$\{\phi_m\}_{m=1}^\infty$ orthonormal in $L^2(\Omega)$ and orthogonal in $H^1_0(\Omega)$ with respect to the weighted product $(c^2 \nabla_{\bx} \cdot, \nabla_{\bx} \cdot)_{L^2(\Omega)}$. For~$s_{\bx}$ a nonnegative integer, let us define~$\mathcal{H}^{s_{\bx}}(\Omega) \subset L^2(\Omega)$ as the space of functions $v$ such that
\begin{equation*}
    \| v \|_{\mathcal{H}^{s_{\bx}}(\Omega)} := \left(\sum_{m=1}^\infty \lambda_m^{s_{\bx}} (v,\phi_m)^2_{L^2(\Omega)} \right)^{1/2} < \infty.
\end{equation*}
Let $\mathcal{B}_{j_t} : V_{j_t}^t(0,T) \to V_{j_t}^t(0,T)$ be the operator defined by
\begin{equation} \label{eq:40}
    ((\mathcal{B}_{j_t} z_{j_t})', v_{j_t}' )_{L_e^2(0,T)} = (z_{j_t}'', v_{j_t}')_{L_e^2(0,T)} + z_{j_t}'(0) v_{j_t}'(0) \quad \text{for all~} v_{j_t} \in V_{j_t}^t(0,T).
\end{equation}
This operator is well-defined by the Riesz representation theorem. We also use~$\mathcal{B}_{j_t}$ to denote its natural extension to space--time functions.

\medskip
\noindent 
The following stability result holds.
\begin{lemma} \label{lem:5}
Let $u \in \mathcal{H}^{s_{\bx}+3}(\Omega) \otimes H_{0,\bullet}^2(0,T)$ with $s_{\bx} \ge 0$ and define $z := (\Id^{\bx} \otimes (\Id^t-\Pi^{\partial_t^2}_{j_t})) u$. Then, the following estimates hold true for all $\jt\ge 0$:
\begin{align*}
    \| P_{\infty,j_t}  z\|_{\mathcal{H}^{s_{\bx}}(\Omega)\otimes L^2(0,T)} & \lesssim \| z \|_{\mathcal{H}^{s_{\bx}+1}(\Omega)\otimes L^2(0,T)}
    \\ \|\partial_t \mathcal{B}_{j_t} P_{\infty,j_t}  z\|_{\mathcal{H}^{s_{\bx}}(\Omega)\otimes L^2(0,T)} & \lesssim \| z \|_{\mathcal{H}^{s_{\bx}+2}(\Omega)\otimes L^2(0,T)} + \| z \|_{\mathcal{H}^{s_{\bx}+3}(\Omega)\otimes L^2(0,T)}
\end{align*}
\end{lemma}
\begin{proof}
We decompose $z$ and $P_{\infty,j_t}z$ as
\begin{equation*}
    z(\bx, t) = \sum_{m=1}^\infty \hat{z}_m(t) \phi_m(\bx), \quad P_{\infty,j_t}z(\bx, t) = \sum_{m=1}^\infty \hat{w}_m^{j_t}(t) \phi_m(\bx),
\end{equation*}
with~$\hat{z}_m \in H_{0,\bullet}^2(0,T)$ and $\hat w_m^{j_t} \in V_{j_t}^t(0,T)$.

\noindent
From definition~\eqref{eq:26}, the condition $\mathcal{A}(P_{\infty,j_t}z, v_{j_t}) = \mathcal{A}(z, v_{j_t})$ must hold for any $v_{j_t} \in H^1_0(\Omega) \otimes V_{j_t}^t(0,T)$. We choose the specific test function $v_{j_t}(\bx,t) = \phi_m(\bx) \eta_{j_t}(t)$, where $\eta_{j_t} \in V_{j_t}^t(0,T)$ is arbitrary. Substituting the above expansions into the bilinear form $\mathcal{A}$, and using the orthogonality and orthonormality properties of the eigenfunctions~$\{\phi_m\}_{m=1}^\infty$, we obtain
\begin{align*}
    \mathcal{A}(P_{\infty,j_t}z, \phi_m \eta_{j_t}) &= (( \hat{w}_m^{j_t})'', \eta_{j_t}')_{L_e^2(0,T)} + (\hat{w}_m^{j_t})'(0)  \eta_{j_t}'(0) + \lambda_m (\hat{w}_m^{j_t}, \eta_{j_t}')_{L_e^2(0,T)}, \\
    \mathcal{A}(z, \phi_m \eta_{j_t}) &=  \lambda_m (\hat{z}_m, \eta_{j_t}')_{L_e^2(0,T)}.
\end{align*}
Then, the discrete variational equation for the $m$-th mode becomes
\begin{equation} \label{eq:41}
    ((\hat{w}_m^{j_t})'', \eta_{j_t}')_{L_e^2(0,T)} + (\hat{w}_m^{j_t})'(0) \eta_{j_t}'(0) + \lambda_m (\hat{w}_m^{j_t}, \eta_{j_t}')_{L_e^2(0,T)} = \lambda_m (\hat{z}_m, \eta_{j_t}')_{L_e^2(0,T)}
\end{equation}
or, equivalently, with $\mathcal{B}_{j_t}$ as in \eqref{eq:40},
\begin{equation} \label{eq:42}
    (\left(\mathcal{B}_{j_t} \hat{w}_m^{j_t}\right)', \eta_{j_t}')_{L_e^2(0,T)}  + \lambda_m (\hat{w}_m^{j_t}, \eta_{j_t}')_{L_e^2(0,T)} = \lambda_m (\hat{z}_m, \eta_{j_t}')_{L_e^2(0,T)}.
\end{equation}
Testing~\eqref{eq:41} with $\eta_{j_t} = \hat{w}_m^{j_t}$ yields
\begin{equation*}
    ((\hat{w}_m^{j_t})'', (\hat{w}_m^{j_t})')_{L_e^2(0,T)} + | (\hat{w}_m^{j_t})'(0)|^2 + \lambda_m (\hat{w}_m^{j_t}, (\hat{w}_m^{j_t})')_{L_e^2(0,T)} = \lambda_m (\hat z_m, (\hat w^{j_t}_m)')_{L^2_e(0,T)}.
\end{equation*}
This, together with~\eqref{eq:3} and the Cauchy--Schwarz inequality, gives 
\begin{align*}
      \frac{1}{2T} \|( \hat{w}_m^{j_t})'\|_{L_e^2(0,T)}^2 + \frac{1}{2e} |( \hat w_m^{j_t})'(T)|^2 +\frac{1}{2} |(\hat{w}_m^{j_t})'(0)|^2 + \frac{\lambda_m}{2T} \|\hat{w}_m^{j_t}\|_{L_e^2(0,T)}^2 & + \frac{\lambda_m}{2e} |\hat w_m^{j_t}(T)|^2
      \\ & \le \lambda_m \| \hat z_m \|_{L^2_e(0,T)} \| (\hat w^{j_t}_m)' \|_{L^2_e(0,T)},
\end{align*}
from which
\begin{equation*}
  \frac{1}{2T} \|(\hat{w}_m^{j_t})'\|_{L_e^2(0,T)}^2 + \frac{\lambda_m}{2T} \|\hat{w}_m^{j_t}\|_{L_e^2(0,T)}^2 \le \lambda_m \|\hat{z}_m\|_{L_e^2(0,T)} \|(\hat{w}_m^{j_t})'\|_{L_e^2(0,T)} \le \frac{T\lambda_m^2}{2}\|\hat{z}_m\|_{L_e^2(0,T)}^2+\frac{1}{2T}\|(\hat{w}_m^{j_t})'\|_{L_e^2(0,T)}^2,
\end{equation*}
where, in the second inequality, we have used the Young inequality. This implies
\begin{equation} \label{eq:43}
    \|\hat{w}_m^{j_t}\|_{L_e^2(0,T)}^2 \le T^2 \lambda_m \|\hat{z}_m\|_{L_e^2(0,T)}^2,
\end{equation}
which gives 
\begin{align*}
    \| P_{\infty,j_t}z\|_{\mathcal{H}^{s_{\bx}}(\Omega) \otimes L^2(0,T)}^2 = \sum_{m=1}^\infty \lambda_m^{s_{\bx}} \| \hat{w}_m^{j_t}\|_{L_e^2(0,T)}^2 
    & \le T^2 \sum_{m=1}^\infty \lambda_m^{s_{\bx}+1} \|\hat{z}_m\|_{L_e^2(0,T)}^2  = T^2\|z\|_{\mathcal{H}^{s_{\bx}+1}(\Omega) \otimes L^2(0,T)}^2.
\end{align*}
This proves the first estimate in the statement and, in particular, implies $P_{\infty,j_t} z \in \mathcal{H}^{s_{\bx}}(\Omega) \otimes L^2(0,T)$.

\medskip
\noindent 
For the second estimate, testing~\eqref{eq:42} with~$ \eta_{j_t} = \mathcal{B}_{j_t} \hat{w}_m^{j_t}$, we deduce
\begin{equation*}
\begin{aligned}
    \|(\mathcal{B}_{j_t} \hat{w}_m^{j_t} )'\|^2_{L^2_e(0,T)} & =\lambda_m(\hat{z}_m-\hat{w}_m^{j_t},(\mathcal{B}_{j_t} \hat{w}_m^{j_t} )')_{L^2_e(0,T)}
     \lesssim (\lambda_m^{3/2}  +  \lambda_m) \| \hat z_m \|_{L^2_e(0,T)} \|  (\mathcal{B}_{j_t} \hat w_{m}^{j_t})'\|_{L^2_e(0,T)},
\end{aligned}
\end{equation*}
where the last step follows from the Cauchy--Schwarz inequality and~\eqref{eq:43}. Then, we conclude
\begin{align*}
    \|\partial_t \mathcal{B}_{j_t} P_{\infty,j_t}z\|_{\mathcal{H}^{s_{\bx}}(\Omega) \otimes L^2(0,T)}^2 = \sum_{m=1}^\infty \lambda_m^{s_{\bx}} \|(\mathcal{B}_{j_t} \hat{w}_m^{j_t})'\|_{L_e^2(0,T)}^2 & \lesssim \sum_{m=1}^\infty \lambda_m^{s_{\bx}} \left( (\lambda_m^3+\lambda_m^2) \|\hat{z}_m\|_{L_e^2(0,T)}^2 \right)
    \\ & = \|z\|_{\mathcal{H}^{s_{\bx}+3}(\Omega) \otimes L^2(0,T)}^2 + \|z\|_{\mathcal{H}^{s_{\bx}+2}(\Omega) \otimes L^2(0,T)}^2,
\end{align*}
which completes the proof.
\end{proof}
\noindent
One could expect to obtain the first bound in Lemma~\ref{lem:5} with the same norm 
$\|\cdot\|_{\mathcal{H}^{s_{\bx}}(\Omega)\otimes L^2(0,T)}$ on the left- and on the right-hand side. However, the stronger norm on the right-hand side is due to the presence of the eigenvalue $\lambda_m$ in the bound~\eqref{eq:43}, which naturally arises from the discrete variational ODE \eqref{eq:41}.

\medskip
\noindent
For~$c=1$, the following characterization of the space~$\mathcal{H}^{s_{\bx}}(\Omega)$ was proved, e.g., in~\cite[Lemma 3.1]{thomee2006galerkin}:
\begin{equation*}
    \mathcal{H}^{s_{\bx}}(\Omega) = \{ v \in H^{s_{\bx}}(\Omega) \mid \mathcal{L}^{\ell_{\bx}} v = 0 \quad \text{on~} \partial \Omega, \quad \text{for~} \ell_{\bx} < s_{\bx}/2,  \ell_{\bx}\in \mathbb{N}\},
\end{equation*}
together with the equivalence between the norm~$\| \cdot \|_{\mathcal{H}^{s_{\bx}}(\Omega)}$ and the standard Sobolev norm~$ \| \cdot \|_{H^{s_{\bx}}(\Omega)}$, where boundary conditions are intended in the sense of traces in $L^2(\partial \Omega)$. This result readily extends to any coefficient~$c \in W^{\infty,s_{\bx}-1}(\Omega)$. Therefore, an immediate consequence of Lemma~\ref{lem:5} is the following result.
\begin{corollary}\label{cor:1}
Let $u \in H^{s_{\bx}+3}(\Omega) \otimes H_{0,\bullet}^2(0,T)$ with $ s_{\bx} \ge 0$, and define $z := (\Id^{\bx} \otimes (\Id^t - \Pi^{\partial_t^2}_{j_t})) u$. Suppose also that~\eqref{eq:28} holds and that~$c \in W^{\infty,s_{\bx}+2}$. Then, the following estimates hold true for all $j_t \ge 0$:
\begin{align} \nonumber 
    \| P_{\infty,j_t} z \|_{H^{s_{\bx}}(\Omega) \otimes L^2(0,T)} & \lesssim \| z \|_{H^{s_{\bx}+1}(\Omega) \otimes L^2(0,T)}, \\ \nonumber \| \partial_t \mathcal{B}_{j_t} P_{\infty,j_t} z \|_{H^{s_{\bx}}(\Omega) \otimes L^2(0,T)} & \lesssim \| z \|_{H^{s_{\bx}+3}(\Omega) \otimes L^2(0,T)},
\end{align}
provided that $\mathcal{L}^{\ell_{\bx}} z(\bx,t)=0$ for all $(\bx,t) \in \partial \Omega \times (0,T)$ and integers $\ell_{\bx}$ such that $1 \le \ell_{\bx}< (s_{\bx}+3)/2$, where~$\mathcal{L} = - \nabla_{\bx} \cdot (c^2 \nabla_{\bx} \cdot)$.
\end{corollary}
\noindent
Let $\mathcal{L}_{j_{\bx}} : V_{j_{\bx}}^{\bx}(\Omega) \to V_{j_{\bx}}^{\bx}(\Omega)$ be the discrete negative diffusion operator with coefficient~$c$, which is defined by
\begin{equation} \label{eq:44}
    (\mathcal{L}_{j_{\bx}} z_{j_{\bx}}, v_{j_{\bx}} )_{L^2(\Omega)} = (c^2 \nabla_{\bx} z_{j_{\bx}}, \nabla_{\bx} v_{j_{\bx}})_{L^2(\Omega)}\quad \text{for all~} v_{j_{\bx}} \in V_{j_{\bx}}^{\bx}(\Omega).
\end{equation}
This operator is well defined by the Riesz representation theorem. Furthermore, it can be naturally extended to space--time functions.

\medskip
\noindent 
The following stability results hold.
\begin{lemma}\label{lem:4}
Let $u \in H^1_0(\Omega) \otimes H^{s_t+3}(0,T)$ with $ s_t \ge 1$, and define $w := ((\Id^{\bx}-\Pi^{\nabla_{\bx}}_{j_{\bx}})\otimes \Id^t) u$. Then, the following estimates hold for all $j_{\bx} \ge 0$: 
\begin{align} \label{eq:45}
& \begin{aligned}
    \| \partial_t^{s_t} P_{j_{\bx},\infty} w \|_{L^2(Q_T)} & \lesssim \|\partial_t^{s_t+1} w \|_{L^2(Q_T)} 
    \\ & \quad + \|\partial_t^{s_t} P_{j_{\bx},\infty} w(\cdot, 0)\|_{L^2(\Omega)} + \|c \nabla_{\bx} \partial_t^{s_t-1} P_{j_{\bx},\infty} w (\cdot, 0)\|_{L^2(\Omega)},
\end{aligned}
\\ \label{eq:46}
& \begin{aligned}
    \|\partial_t^{s_t} \mathcal{L}_{j_{\bx}} P_{j_{\bx},\infty}w \|_{L^2(Q_T)} & \lesssim \|\partial_t^{s_t+2} w \|_{L^2(Q_T)} + \|\partial_t^{s_t+3} w \|_{L^2(Q_T)} 
    \\ & \quad + \|\partial_t^{s_t+2} P_{j_{\bx},\infty} w(\cdot, 0)\|_{L^2(\Omega)} + \|c \nabla_{\bx} \partial_t^{s_t+1} P_{j_{\bx},\infty} w (\cdot, 0)\|_{L^2(\Omega)}.
\end{aligned}
\end{align}
Moreover, the following identity holds for all $\ell_t \ge 1$:
\begin{equation} \label{eq:47}
    \partial_t^{\ell_t} P_{j_{\bx},\infty}w(\bx,0) = \sum_{i=0}^{\lfloor \frac{\ell_t-1}{2} \rfloor} (-1)^i \mathcal{L}_{j_{\bx}}^i \Pi_{j_{\bx}}^{L^2_{\bx}} \partial_t^{\ell_t-2i} w(\bx,0).
\end{equation}
Here, $\Pi_{j_{\bx}}^{L^2_{\bx}} : L^2(\Omega) \to\Vx$ is the $L^2(\Omega)$ projection, which is defined by
\begin{equation} \label{eq:48}
    ( (\Pi_{j_{\bx}}^{L^2_{\bx}}- \Id^{\bx}) w, w_{j_{\bx}})_{L^2(\Omega)} = 0 \qquad \text{for all~} w_{\jx} \in \Vx.
\end{equation}
\end{lemma}
\begin{proof}
Recalling~\eqref{eq:27}, we have that $P_{j_{\bx},\infty} w \in V_{j_{\bx}}^{\bx}(\Omega) \otimes H^2_{0,\bullet}(0,T)$ satisfies
\begin{equation} \label{eq:49}
\begin{aligned}
    & (\partial_t^2 P_{j_{\bx},\infty} w, \partial_t v_{j_{\bx}})_{L_e^2(Q_T)}  + (\partial_t P_{j_{\bx},\infty} w(\cdot,0), \partial_t v_{j_{\bx}}(\cdot,0))_{L^2(\Omega)} 
    + (c^2\nabla_{\bx} P_{j_{\bx},\infty} w, \nabla_{\bx} \partial_t v_{j_{\bx}})_{L_e^2(Q_T)}
    \\ & \quad = (\partial_t^2 w, \partial_t v_{j_{\bx}})_{L_e^2(Q_T)} + (\partial_t w(\cdot,0), \partial_t v_{j_{\bx}}(\cdot,0))_{L^2(\Omega)}
\end{aligned}
\end{equation}
for all $v_{j_{\bx}} \in V_{j_{\bx}}^{\bx}(\Omega) \otimes H^2_{0,\bullet}(0,T)$, since $(c^2\nabla_{\bx} w, \nabla_{\bx} \partial_t v_{j_{\bx}} )_{L^2_e(Q_T)} = 0$ due to \eqref{eq:11}.

\noindent
Consider testing \eqref{eq:49} with $v_{j_{\bx}}(\bx,t) = \phi_{j_{\bx}}(\bx)\eta(t)$, where $\phi_{j_{\bx}}\in\Vx$ and $\eta \in H^2_{0,\bullet}(0,T)$ satisfies $\eta'(0)=0$. This yields
\begin{equation*}
\begin{aligned}
     \int_0^T e^{-t/T} \eta'(t) \Big((\partial_t^2 P_{j_{\bx},\infty} w(\cdot,t), \phi_{j_{\bx}})_{L^2(\Omega)}   + (c^2\nabla_{\bx} P_{j_{\bx},\infty} w(\cdot,t), \nabla_{\bx} \phi_{j_{\bx}})_{L^2(\Omega)} - (\partial_t^2 w(\cdot,t), \phi_{j_{\bx}})_{L^2(\Omega)}\Big) \dd t = 0.
\end{aligned}
\end{equation*}
Since $H^1_{0,\bullet}(0,T)$ is dense in $L^2(0,T)$, we deduce that, for all $\phi_{j_{\bx}} \in V_{j_{\bx}}^{\bx}(\Omega)$, it holds, for any~$t\in(0,T)$,
\begin{equation} \label{eq:50}
\begin{aligned}
     (\partial_t^2 P_{j_{\bx},\infty} w(\cdot,t), \phi_{j_{\bx}})_{L^2(\Omega)} + (c^2\nabla_{\bx} P_{j_{\bx},\infty} w(\cdot,t), \nabla_{\bx} \phi_{j_{\bx}})_{L^2(\Omega)} = (\partial_t^2 w(\cdot,t),  \phi_{j_{\bx}})_{L^2(\Omega)}.
\end{aligned}
\end{equation}
Now, testing~\eqref{eq:49} with $v_{j_{\bx}}(\bx,t) = \phi_{j_{\bx}}(\bx)\eta(t)$, where $\phi_{j_{\bx}}\in\Vx$ and $\eta \in H^2_{0,\bullet}(0,T)$ satisfies $\eta'(0)=1$, and subtracting \eqref{eq:50}, we find
\begin{equation*}
     (\partial_t P_{j_{\bx},\infty} w(\cdot,0), \phi_{j_{\bx}})_{L^2(\Omega)} = (\partial_t w(\cdot,0), \phi_{j_{\bx}})_{L^2(\Omega)} \quad \text{for all~} \phi_{j_{\bx}} \in V_{j_{\bx}}^{\bx}(\Omega),
\end{equation*}
which is equivalent to
\begin{equation} \label{eq:51}
    \partial_t P_{j_{\bx},\infty} w(\cdot,0) = \Pi_{j_{\bx}}^{L^2_{\bx}} \partial_t w(\cdot,0).
\end{equation}
Differentiating \eqref{eq:50} $s_t-1$ times with respect to time, we obtain, for any~$t\in(0,T)$,
\begin{equation} \label{eq:52}
    (\partial_t^{s_t+1} P_{j_{\bx},\infty} w(\cdot,t), \phi_{j_{\bx}})_{L^2(\Omega)} + (c^2 \nabla_{\bx} \partial_t^{s_t-1} P_{j_{\bx},\infty} w(\cdot,t), \nabla_{\bx} \phi_{j_{\bx}})_{L^2(\Omega)} = (\partial_t^{s_t+1} w(\cdot,t), \phi_{j_{\bx}})_{L^2(\Omega)}
\end{equation}
for all $\phi_{j_{\bx}} \in V_{j_{\bx}}^{\bx}(\Omega)$. Note that part \textit{ii)} of Proposition \ref{prop:1}  implies that $P_{j_{\bx},\infty} w \in V_{j_{\bx}}^{\bx}(\Omega) \otimes H_{0,\bullet}^{s_t+3}(0,T)$. 

\medskip
\noindent 
In this part of the proof, we denote the time variable by~$\tau \in (0,T)$, since below we integrate with respect to time up to the upper limit~$t$. We define the energy of~$\partial_{\tau}^{s_t-1}P_{j_{\bx},\infty} w$ by
\begin{equation*}
    E_{s_t}(\tau) :=\frac12\|\partial_{\tau}^{s_t} P_{j_{\bx},\infty} w(\cdot,\tau)\|_{L^2(\Omega)}^2 + \frac12\|c \nabla_{\bx} \partial_{\tau}^{s_t-1} P_{j_{\bx},\infty} w (\cdot, \tau)\|_{L^2(\Omega)}^2.
\end{equation*}
Choosing the test function $\phi_{j_{\bx}} = \partial_{\tau}^{s_t} P_{j_{\bx},\infty} w(\cdot, \tau)$ in \eqref{eq:52} and applying the Cauchy--Schwarz inequality yields
\begin{equation*}
    E_{s_t}'(\tau) = (\partial_{\tau}^{s_t+1} w(\cdot,\tau), \partial_{\tau}^{s_t} P_{j_{\bx},\infty} w(\cdot,\tau))_{L^2(\Omega)} \le \|\partial_{\tau}^{s_t+1} w(\cdot, \tau)\|_{L^2(\Omega)} \sqrt{2 E_{s_t}(\tau)}.
\end{equation*}
Observing that $E_{s_t}'(\tau) = 2\sqrt{E_{s_t}(\tau)} (\sqrt{E_{s_t}(\tau)})'$, and considering the nontrivial case $E_{s_t}(\tau) > 0$, we conclude
\begin{equation*}
  (\sqrt{E_{s_t}(\tau)})' \le \frac1{\sqrt2}\|\partial_{\tau}^{s_t+1} w(\cdot,\tau)\|_{L^2(\Omega)}.
\end{equation*}
Integrating over $(0,t)$ for any~$t\in (0,T)$ gives
\begin{equation*}
    \frac1{\sqrt2} \|\partial_t^{s_t} P_{j_{\bx},\infty} w(\cdot, t)\|_{L^2(\Omega)} \le \sqrt{E_{s_t}(t)} \le \sqrt{E_{s_t}(0)} + \frac1{\sqrt2}\int_0^t \|\partial_{\tau}^{s_t+1} w(\cdot,\tau)\|_{L^2(\Omega)} \dd \tau,
\end{equation*}
which implies~\eqref{eq:45}. 

\medskip
\noindent
In order to prove~\eqref{eq:46}, differentiating~\eqref{eq:50} $s_t$ times with respect to time and using~\eqref{eq:44} give
\begin{equation*}
    (\partial_t^{s_t+2} P_{j_{\bx},\infty} w(\cdot,t), \phi_{j_{\bx}})_{L^2(\Omega)} + (\mathcal{L}_{j_{\bx}} \partial_t^{s_t} P_{j_{\bx},\infty} w(\cdot,t), \phi_{j_{\bx}})_{L^2(\Omega)} = (\partial_t^{s_t+2} w(\cdot,t), \phi_{j_{\bx}})_{L^2(\Omega)}
\end{equation*}
for all $\phi_{j_{\bx}} \in V_{j_{\bx}}^{\bx}(\Omega)$. From this, we obtain that, for all $\bx \in \Omega$ and~$t\in(0,T)$,
\begin{equation} \label{eq:53}
\begin{aligned}
     \partial_t^{s_t+2} P_{j_{\bx},\infty} w(\bx,t) + \mathcal{L}_{j_{\bx}} \partial_t^{s_t} P_{j_{\bx},\infty} w(\bx,t) = \Pi_{j_{\bx}}^{L^2_{\bx}} \partial_t^{s_t+2} w(\bx,t),
\end{aligned}
\end{equation}
with $\Pi_{j_{\bx}}^{L^2_{\bx}}$ as in \eqref{eq:48}. Using the stability of $\Pi_{j_{\bx}}^{L^2_{\bx}}$ in $L^2(\Omega)$ and \eqref{eq:45} (with $s_t \to s_t+2$), we conclude
\begin{equation*}
\begin{aligned}
    \| \mathcal{L}_{j_{\bx}} \partial_t^{s_t} P_{j_{\bx},\infty} w\|_{L^2(Q_T)} & \le \| \partial_t^{s_t+2} P_{j_{\bx},\infty} w \|_{L^2(Q_T)} + \| \partial_t^{s_t+2} w\|_{L^2(Q_T)}
    \\ & \lesssim \sqrt{E_{s_t+2}(0)} + \| \partial_t^{s_t+3} w \|_{L^2(Q_T)} + \| \partial_t^{s_t+2} w \|_{L^2(Q_T)},
\end{aligned}
\end{equation*}
which leads to \eqref{eq:46}.

\medskip 
\noindent 
Finally, using~\eqref{eq:51} and~\eqref{eq:53}, which is also valid with~$s_t=0$, and a recursive argument, we obtain, for all $\bx \in \Omega$, 
\begin{align*}
\begin{cases}
    P_{j_{\bx},\infty}w(\bx,0) = 0, & \vspace{0.1cm}
    \\ \partial_t P_{j_{\bx},\infty}w(\bx,0) = \Pi_{j_{\bx}}^{L^2_{\bx}} \partial_t w(\bx,0), & \vspace{0.1cm}
    \\ \partial_t^2 P_{j_{\bx},\infty}w(\bx,0) = \Pi_{j_{\bx}}^{L^2_{\bx}} \partial_t^2 w(\bx,0), & \vspace{0.1cm}
    \\ \partial_t^3 P_{j_{\bx},\infty}w(\bx,0) = \Pi_{j_{\bx}}^{L^2_{\bx}} \partial_t^3 w(\bx,0) - \mathcal{L}_{j_{\bx}} \partial_t P_{j_{\bx},\infty} w(\bx,0) = \Pi_{j_{\bx}}^{L^2_{\bx}} \partial_t^3 w(\bx,0) - \mathcal{L}_{j_{\bx}} \Pi_{j_{\bx}}^{L^2_{\bx}} \partial_t w(\bx,0), &  \vspace{0.1cm}
    \\ \partial_t^4 P_{j_{\bx},\infty}w(\bx,0) = \Pi_{j_{\bx}}^{L^2_{\bx}} \partial_t^4 w(\bx,0) - \mathcal{L}_{j_{\bx}} \partial_t^2 P_{j_{\bx},\infty} w(\bx,0) = \Pi_{j_{\bx}}^{L^2_{\bx}} \partial_t^4 w(\bx,0) - \mathcal{L}_{j_{\bx}} \Pi_{j_{\bx}}^{L^2_{\bx}} \partial_t^2 w(\bx,0), & \vspace{0.1cm}
    \\ \cdots & \vspace{0.1cm}
    \\ \partial_t^{\ell_t} P_{j_{\bx},\infty}w(\bx,0) = \Pi_{j_{\bx}}^{L^2_{\bx}} \partial_t^{\ell_t} w(\bx,0) - \mathcal{L}_{j_{\bx}} \partial_t^{\ell_t-2} P_{j_{\bx},\infty} w(\bx,0), &
    \end{cases}
\end{align*}
which, by induction, yields \eqref{eq:47}.
\end{proof}
\noindent
An immediate consequence of Lemma \ref{lem:4} is the following result.
\begin{corollary}\label{cor:2}
Let $u \in H^1_0(\Omega) \otimes H^{s_t+3}(0,T)$ with $ s_t \ge 1$, and define $w := ((\Id^{\bx}-\Pi^{\nabla_{\bx}}_{j_{\bx}})\otimes \Id^t) u$. Then, the following estimates hold true for all $\jx\ge 0$:
\begin{align} \label{eq:54}
    \| P_{j_{\bx},\infty} w \|_{L^2(\Omega) \otimes H^{s_t}(0,T)} & \lesssim \| w \|_{L^2(\Omega) \otimes H^{s_t+1}(0,T)}, \\ \label{eq:55} \| \mathcal{L}_{j_{\bx}} P_{j_{\bx},\infty} w \|_{L^2(\Omega) \otimes H^{s_t}(0,T)} & \lesssim \| w \|_{L^2(\Omega) \otimes H^{s_t+3}(0,T)},
\end{align}
provided that $\partial_t^{\ell_t} w(\bx,0)=0$ for all $\bx \in \Omega$ and $\ell_t=1,\ldots,s_t$.
\end{corollary}

\subsection{Convergence of Galerkin increments} \label{sec:gal}

Preliminarily, we estimate the error between the space--time discrete Galerkin projection $P_{j_\bx,j_t}$ and the extensions to space--time functions of the projectors $\Pi_{j_\bx}^{\nabla_\bx}$ and $\Pi_{j_t}^{\partial_t^2}$ defined in \eqref{eq:11} and \eqref{eq:12}, applied to the semidiscrete approximations $P_{\infty,j_t}u$ and $P_{j_\bx,\infty}u$ defined in~\eqref{eq:26} and~\eqref{eq:27}. Then, in Proposition~\ref{prop:6}, we estimate the difference of the Galerkin projections at two consecutive refinement levels in space and in time.
\begin{lemma} 
For all $z \in H^{p_{\bx}+1}(\Omega) \otimes H^2_{0,\bullet}(0,T)$, under assumption~\eqref{eq:22}, it holds
\begin{equation} \label{eq:56}
    \| (P_{j_{\bx},j_t} -\Pi_{j_{\bx}}^{\nabla_{\bx}} \otimes \Id^t) P_{\infty,j_t}z\|_{\mathcal{V}(Q_T)}\lesssim 2^{-\jx (p_{\bx}+1)} \| \partial_t \mathcal{B}_{j_t} P_{\infty,j_t}z \|_{H^{p_{\bx}+1}(\Omega)\otimes L^2(0,T)}.
\end{equation}
For all $w \in H^1_0(\Omega) \otimes H^{p_t+3}(0,T)$ it holds 
\begin{equation} \label{eq:57}
    \| (P_{j_{\bx},j_t} - \Id^{\bx} \otimes \Pi_{j_t}^{\partial_t^2}) P_{j_{\bx},\infty}w\|_{\mathcal{V}(Q_T)} \lesssim 2^{-j_t (p_t+1)} \| \mathcal{L}_{j_{\bx}} P_{j_{\bx},\infty}w \|_{L^2(\Omega)\otimes H^{p_t+3}(0,T)}.
\end{equation}
\end{lemma}
\begin{proof}
Using the coercivity in \eqref{eq:8} and consistency, we have
\begin{equation} \label{eq:58}
\begin{aligned}
    \| (P_{j_{\bx},j_t}-\Pi_{j_{\bx}}^{\nabla_{\bx}} \otimes \Id^t) P_{\infty,j_t}z \|^2_{\mathcal{V}(Q_T)} & \lesssim \A((P_{j_{\bx},j_t}-\Pi_{j_{\bx}}^{\nabla_{\bx}}  \otimes \Id^t) P_{\infty,j_t}z ,(P_{j_{\bx},j_t}-\Pi_{j_{\bx}}^{\nabla_{\bx}} \otimes \Id^t) P_{\infty,j_t}z )
   \\ &  = \A((\Id-\Pi_{j_{\bx}}^{\nabla_{\bx}}  \otimes \Id^t) P_{\infty,j_t}z, (P_{j_{\bx},j_t}-\Pi_{j_{\bx}}^{\nabla_{\bx}}  \otimes \Id^t) P_{\infty,j_t}z).
\end{aligned}
\end{equation}
Let us set $z_{j_{\bx},j_t} := (P_{j_{\bx},j_t}-\Pi_{j_{\bx}}^{\nabla_{\bx}} \otimes \Id^t) P_{\infty,j_t}z$, for convenience. Using the definition of $\Pi_{j_{\bx}}^{\nabla_{\bx}}$ in \eqref{eq:11}, we further obtain
\begin{equation} \label{eq:59}
\begin{aligned}
    \A((\Id-\Pi_{j_{\bx}}^{\nabla_{\bx}}  \otimes \Id^t) P_{\infty,j_t}z, z_{j_{\bx},j_t}) &  = ((\Id-\Pi_{j_{\bx}}^{\nabla_{\bx}}\otimes \Id^t)\partial_t^2 P_{\infty,j_t}z, \partial_t z_{j_{\bx},j_t} )_{L_e^2(Q_T)}
    \\ & \quad + ((\Id-\Pi_{j_{\bx}}^{\nabla_{\bx}}\otimes \Id^t) \partial_t P_{\infty,j_t}z(\cdot,0), \partial_t z_{j_{\bx},j_t}(\cdot,0))_{L^2(\Omega)}.
\end{aligned}
\end{equation}
Combining~\eqref{eq:58} and~\eqref{eq:59}, using the definition of the operator $\mathcal{B}_{j_t}$ in~\eqref{eq:40}, and applying the Cauchy--Schwarz inequality along with the approximation property~\eqref{eq:13} of $\Pi^{\nabla_{\bx}}_{j_{\bx}}$, we deduce
\begin{align*}
    \| z_{j_{\bx},j_t}\|^2_{\mathcal{V}(Q_T)} & \lesssim ((\Id-\Pi_{j_{\bx}}^{\nabla_{\bx}}\otimes \Id^t)\partial_t^2 P_{\infty,j_t}z, \partial_t z_{j_{\bx},j_t} )_{L_e^2(Q_T)}
    \\ & \quad + ((\Id-\Pi_{j_{\bx}}^{\nabla_{\bx}}\otimes \Id^t) \partial_t P_{\infty,j_t}z(\cdot,0), \partial_t z_{j_{\bx},j_t}(\cdot,0))_{L^2(\Omega)} 
    \\ & = (\partial_t \mathcal{B}_{j_t} (\Id-\Pi_{j_{\bx}}^{\nabla_{\bx}}\otimes \Id^t) P_{\infty,j_t}z, \partial_t z_{j_{\bx},j_t})_{L^2_e(Q_T)}
    \\ & \lesssim 
    \| (\Id-\Pi_{j_{\bx}}^{\nabla_{\bx}}\otimes \Id^t) \partial_t \mathcal{B}_{j_t}P_{\infty,j_t}z \|_{L^2_e(Q_T)}  \| z_{j_{\bx},j_t} \|_{\mathcal{V}(Q_T)}
    \\ & \lesssim 2^{-j_{\bx} (p_{\bx}+1)}\| \partial_t \mathcal{B}_{j_t} P_{\infty,j_t}z \|_{H^{p_{\bx}+1}(\Omega)\otimes L^2(0,T)} \| z_{j_{\bx},j_t} \|_{\mathcal{V}(Q_T)},
\end{align*}
which leads to \eqref{eq:56}.

\medskip
\noindent
To obtain \eqref{eq:57}, we proceed similarly. Using the coercivity in \eqref{eq:8} and consistency, we obtain
\begin{equation} \label{eq:60}
\begin{aligned}
    \| (P_{j_{\bx},j_t}-\Id^{\bx} \otimes \Pi_{j_t}^{\partial_t^2}) P_{j_{\bx},\infty}w \|^2_{\mathcal{V}(Q_T)} & \lesssim \A((P_{j_{\bx},j_t}- \Id^{\bx} \otimes \Pi_{j_t}^{\partial_t^2}) P_{j_{\bx},\infty}w ,(P_{j_{\bx},j_t}-\Id^{\bx} \otimes \Pi_{j_t}^{\partial_t^2}) P_{j_{\bx},\infty}w )
   \\ &  = \A((\Id-\Id^{\bx} \otimes \Pi_{j_t}^{\partial_t^2}) P_{j_{\bx},\infty}w, (P_{j_{\bx},j_t}-\Id^{\bx} \otimes \Pi_{j_t}^{\partial_t^2}) P_{j_{\bx},\infty}w).
\end{aligned}
\end{equation}
Let us set $w_{j_{\bx},j_t} := (P_{j_{\bx},j_t}-\Id^{\bx} \otimes \Pi_{j_t}^{\partial_t^2}) P_{j_{\bx},\infty}w$. Using the definition of $\Pi_{j_t}^{\partial_t^2}$ in~\eqref{eq:12} 
we further deduce
\begin{equation} \label{eq:61}
    \A((\Id-\Id^{\bx} \otimes \Pi_{j_t}^{\partial_t^2}) P_{j_{\bx},\infty}w, w_{j_{\bx},j_t}) = 
    (c^2 \nabla_{\bx} (\Id - \Id^{\bx} \otimes \Pi_{j_t}^{\partial_t^2}) P_{j_{\bx},\infty}w, \partial_t \nabla_{\bx}w_{j_{\bx},j_t})_{L^2_e(Q_T)}.
\end{equation}
Combine~\eqref{eq:60} and~\eqref{eq:61}, use the definition of the discrete operator~$\mathcal{L}_{j_{\bx}}$ in~\eqref{eq:44}, and apply the  Cauchy--Schwarz inequality together with the approximation property~\eqref{eq:14} of $\Pi_{j_t}^{\partial_t^2}$, to obtain
\begin{align*}
    \| w_{j_{\bx},j_t}\|^2_{\mathcal{V}(Q_T)} & \lesssim (c^2 \nabla_{\bx} (\Id - \Id^{\bx} \otimes \Pi_{j_t}^{\partial_t^2}) P_{j_{\bx},\infty}w, \partial_t \nabla_{\bx}w_{j_{\bx},j_t})_{L^2_e(Q_T)}
    \\ & = (\mathcal{L}_{j_{\bx}} (\Id - \Id^{\bx} \otimes \Pi_{j_t}^{\partial_t^2}) P_{j_{\bx},\infty}w, \partial_t w_{j_{\bx},j_t})_{L^2_e(Q_T)}
    \\ & \lesssim \| (\Id - \Id^{\bx} \otimes \Pi_{j_t}^{\partial_t^2}) \mathcal{L}_{j_{\bx}} P_{j_{\bx},\infty}w \|_{L^2_e(Q_T)}  \| w_{j_{\bx},j_t} \|_{\mathcal{V}(Q_T)} 
    \\ & \lesssim 2^{-j_t (p_t+1)} \| \mathcal{L}_{j_{\bx}} P_{j_{\bx},\infty}w \|_{L^2(\Omega)\otimes H^{p_t+3}(0,T)} \| w_{j_{\bx},j_t} \|_{\mathcal{V}(Q_T)},
\end{align*}
which concludes the proof of~\eqref{eq:57}.
\end{proof}
\begin{proposition}\label{prop:6}
For all $z \in H^{p_{\bx}+1}(\Omega) \otimes H^2_{0,\bullet}(0,T)$, under assumption~\eqref{eq:22}, it holds
\begin{align} \label{eq:62}
    \| (P_{j_{\bx},j_t} - P_{j_{\bx}-1,j_t})z \|_{L^2(Q_T)} & \lesssim 2^{-j_{\bx} (p_{\bx}+1)} \big(\|P_{\infty,j_t}z\|_{H^{p_{\bx}+1}(\Omega)\otimes L^2(0,T)}+ \| \partial_t \mathcal{B}_{j_t} P_{\infty,j_t}z \|_{H^{p_{\bx}+1}(\Omega)\otimes L^2(0,T)}\big).
\end{align}
For all $w \in H^1_0(\Omega) \otimes H^{p_t+3}(0,T)$, it holds 
\begin{equation} \label{eq:63}
\begin{aligned} 
     \| (P_{j_{\bx},j_t} - P_{j_{\bx},j_t-1})w\|_{L^2(Q_T)} \lesssim 2^{-j_t (p_t+1)} \big(& \|P_{j_{\bx},\infty}w\|_{L^2(\Omega)\otimes H^{p_t+3}(0,T)}
     \\ & \quad +\| \mathcal{L}_{j_{\bx}} P_{j_{\bx},\infty}w \|_{L^2(\Omega)\otimes H^{p_t+3}(0,T)}\big).
\end{aligned}
\end{equation}
\end{proposition}
\begin{proof}
We first prove \eqref{eq:62}. By adding and subtracting the semidiscrete Galerkin projection $P_{\infty,j_t}$ defined in~\eqref{eq:26}, and using the triangle inequality, we write
\begin{equation}\label{eq:64}
    \| (P_{j_{\bx},j_t} - P_{j_{\bx}-1,j_t})z\|_{L^2(Q_T)} \le \| (P_{j_{\bx},j_t} - P_{\infty,j_t})z\|_{L^2(Q_T)} + \| (P_{\infty,j_t} - P_{j_{\bx}-1,j_t})z\|_{L^2(Q_T)}.
\end{equation}
We have~$P_{j_{\bx},j_t} = P_{j_{\bx},j_t} P_{\infty,j_t}$. Indeed, from \eqref{eq:10} and \eqref{eq:26}, we get
\begin{equation*}
    \mathcal{A}(P_{j_{\bx},j_t} z, z_{j_{\bx},j_t}) = \mathcal{A}(z, z_{j_{\bx},j_t}) = \mathcal{A}(P_{\infty,j_t} z, z_{j_{\bx},j_t})  = \mathcal{A}(P_{j_{\bx},j_t} P_{\infty,j_t} z, z_{j_{\bx},j_t})
\end{equation*}
for all $z_{j_{\bx},j_t} \in V_{j_{\bx}}^{\bx}(\Omega) \otimes V_{j_t}^t(0,T)$. Thus, adding and subtracting the projection~$\Pi^{\nabla_{\bx}}_{j_{\bx}}$ and applying the triangle inequality, we obtain
\begin{equation*}
\begin{aligned} 
     \| (P_{j_{\bx},j_t} - P_{\infty,j_t})z\|_{L^2(Q_T)} & = \| (P_{j_{\bx},j_t} - \Id) P_{\infty,j_t}z\|_{L^2(Q_T)}
    \\ & \le \| (P_{j_{\bx},j_t} -\Pi_{j_{\bx}}^{\nabla_{\bx}} \otimes \Id^t) P_{\infty,j_t}z\|_{L^2(Q_T)}+ \| (\Id - \Pi_{j_{\bx}}^{\nabla_{\bx}} \otimes \Id^t) P_{\infty,j_t}z\|_{L^2(Q_T)}.
\end{aligned}
\end{equation*}
For the first term on the right-hand side, using the Poincar\'e inequality and \eqref{eq:56}, we derive
\begin{align*}
    \| (P_{j_{\bx},j_t} -\Pi_{j_{\bx}}^{\nabla_{\bx}} \otimes \Id^t) P_{\infty,j_t}z\|_{L^2(Q_T)} & \lesssim 2^{-\jx (p_{\bx}+1)} \| \partial_t \mathcal{B}_{j_t} P_{\infty,j_t}z \|_{H^{p_{\bx}+1}(\Omega)\otimes L^2(0,T)}.
\end{align*}
For the second term, using the approximation properties~\eqref{eq:13} of $\Pi_{j_{\bx}}^{\nabla_{\bx}}$, we deduce
\begin{align*}
    \| (\Id - \Pi_{j_{\bx}}^{\nabla_{\bx}} \otimes \Id^t) P_{\infty,j_t}z\|_{L^2(Q_T)}\lesssim 2^{-\jx(p_{\bx}+1)} \|  P_{\infty,j_t}z\|_{H^{p_{\bx}+1}(\Omega)\otimes L^2(0,T)}.
\end{align*}
This completes the estimate of the first term on the right--hand side of~\eqref{eq:64}. Repeating exactly the same argument for the second term, we deduce~\eqref{eq:62}.

\noindent
Now we prove \eqref{eq:63}. By adding and subtracting the semidiscrete Galerkin projection $P_{j_{\bx},\infty}$ defined in \eqref{eq:27}, and using the triangle inequality, we write
\begin{equation} \label{eq:65}
    \| (P_{j_{\bx},j_t} - P_{j_{\bx},j_t-1})w\|_{L^2(Q_T)} \le \| (P_{j_{\bx},j_t} - P_{j_{\bx},\infty})w\|_{L^2(Q_T)} + \| (P_{j_{\bx},\infty} - P_{j_{\bx},j_t-1})w\|_{L^2(Q_T)}.
\end{equation}
Similarly as before, noting that $P_{j_{\bx},j_t} = P_{j_{\bx},j_t} P_{j_{\bx},\infty}$ , we obtain
\begin{equation*}
\begin{aligned}
    \| (P_{j_{\bx},j_t} - P_{j_{\bx},\infty})w\|_{L^2(Q_T)} & = \| (P_{j_{\bx},j_t} - \Id) P_{j_{\bx},\infty}w\|_{L^2(Q_T)}
    \\ & \le \| (P_{j_{\bx},j_t} - \Id^{\bx} \otimes \Pi_{j_t}^{\partial_t^2}) P_{j_{\bx},\infty}w\|_{L^2(Q_T)} + \| (\Id - \Id^{\bx} \otimes \Pi_{j_t}^{\partial_t^2}) P_{j_{\bx},\infty}w\|_{L^2(Q_T)}.
\end{aligned}
\end{equation*}
For the first term on the right-hand side, using Poincar\'e's inequality and \eqref{eq:57},
\begin{align*}
    \| (P_{j_{\bx},j_t} - \Id^{\bx} \otimes \Pi_{j_t}^{\partial_t^2}) P_{j_{\bx},\infty}w\|_{L^2(Q_T)} \lesssim 2^{-j_t (p_t+1)} \| \mathcal{L}_{j_{\bx}} P_{j_{\bx},\infty}w \|_{L^2(\Omega)\otimes H^{p_t+3}(0,T)}.
\end{align*}
For the second term, using the approximation property~\eqref{eq:14} of $\Pi^{\partial_t^2}_{j_t}$, we obtain
\begin{align*}
    \| (\Id - \Id^{\bx} \otimes \Pi_{j_t}^{\partial_t^2}) P_{j_{\bx},\infty}w\|_{L^2(Q_T)}\lesssim 2^{-j_t(p_t+1)} \|  P_{j_{\bx},\infty}w \|_{L^2(\Omega)\otimes H^{p_t+3}(0,T)},
\end{align*}
and the estimate of the first term on the right--hand side of~\eqref{eq:65} is complete. Applying the same argument to the second term, we deduce~\eqref{eq:63}.
\end{proof}

\section{Proof of convergence of the combination formula}\label{sec:convergence}

Before presenting the proof of Theorem~\ref{th:1} in Section~\ref{sec:proof}, we establish an auxiliary result in Section~\ref{sec:aux}.

\subsection{Auxiliary result}\label{sec:aux}

Analogously to~$\Delta^P_{j_{\bx},j_t}$ defined in~\eqref{eq:20}, for $j_{\bx}, j_t \ge 0$, we define the \emph{detail projection} operator $\Delta_{j_{\bx},j_t}^Q : \mathcal{V}(Q_T) \to V_{j_{\bx}}^{\bx}(\Omega) \otimes V_{j_t}^t(0,T)$ as
\begin{equation}\label{eq:66}
\begin{aligned}
    \Delta_{j_{\bx},j_t}^Qu & := (\Pi^{\nabla_{\bx}}_{j_{\bx}} \otimes \Pi^{\partial_t^2}_{j_t} - \Pi^{\nabla_{\bx}}_{j_{\bx}-1} \otimes \Pi^{\partial_t^2}_{j_t})u - (\Pi^{\nabla_{\bx}}_{j_{\bx}} \otimes \Pi^{\partial_t^2}_{j_t-1} - \Pi^{\nabla_{\bx}}_{j_{\bx}-1} \otimes \Pi^{\partial_t^2}_{j_t-1})u
    \\ & = \big((\Pi_{j_{\bx}}^{\nabla_{\bx}}- \Pi_{j_{\bx}-1}^{\nabla_{\bx}}) \otimes (\Pi_{j_t}^{\partial_t^2} - \Pi_{j_t-1}^{\partial_t^2})\big)u.
\end{aligned}
\end{equation}
where~$\Pi^{\nabla_{\bx}}_{j_{\bx}}$ and~$\Pi^{\partial_t^2}_{j_t}$ are the projections defined in~\eqref{eq:11} and~\eqref{eq:12}, respectively, and with the convention that $(\Id^{\bx} \otimes \Pi_{-1}^{\partial_t^2}) u = (\Pi_{-1}^{\nabla_{\bx}} \otimes \Id^t) u = 0$.

\noindent
It is immediate to deduce from estimates~\eqref{eq:13} and~\eqref{eq:14} that, if~$u$ satisfies the regularity assumption
\begin{equation*}
    u \in H_0^{p_{\bx}+1}(\Omega) \otimes  H_{0,\bullet}^{p_t+3}(0,T),
\end{equation*}
then
\begin{equation} \label{eq:67}
\begin{aligned}
    \| \Delta^Q_{j_{\bx},j_t} u \|_{L^2(Q_T)} & = \big\| \big((\Pi_{j_{\bx}}^{\nabla_{\bx}}- \Pi_{j_{\bx}-1}^{\nabla_{\bx}}) \otimes (\Pi_{j_t}^{\partial_t^2} - \Pi_{j_t-1}^{\partial_t^2})\big) u \big\|_{L^2(Q_T)} 
    \\ & \lesssim 2^{-j_t(p_t+1)} \big\| \big((\Pi_{j_{\bx}}^{\nabla_{\bx}}- \Pi_{j_{\bx}-1}^{\nabla_{\bx}}) \otimes \Id^t\big) u \big\|_{L^2(\Omega) \otimes H^{p_t+3}(0,T)} 
    \\ & \lesssim 2^{-j_{\bx}(p_{\bx}+1)-\jt(p_t+1)} 
    \| u \|_{H^{p_{\bx}+1}(\Omega) \otimes H^{p_t+3}(0,T)}.
\end{aligned}
\end{equation}

\noindent
In the following proposition, we estimate the gap between~$\Delta^P_{j_{\bx},j_t}u$ and $\Delta^Q_{j_{\bx},j_t}u$.
\begin{proposition} \label{prop:7}
Suppose \eqref{eq:22}, and that the solution~$u$ of problem~\eqref{eq:7} satisfies the mixed regularity assumption~\eqref{eq:23}. Then, provided that the \vanishingconditions \eqref{eq:24} and~\eqref{eq:25} are satisfied, we have
\begin{equation*}
    \big\| \big( \Delta_{j_{\bx},j_t}^P - \Delta_{j_{\bx},j_t}^Q \big) u \big\|_{L^2(Q_T)} \lesssim  2^{-j_{\bx} (p_{\bx} +1)-j_t (p_t+1)} \big(\|u\|_{H^{p_{\bx}+4}(\Omega)\otimes H^{p_t+3}(0,T)}
    + \| u \|_{H^{p_{\bx}+1}(\Omega)\otimes H^{p_t+6}(0,T)}\big),
\end{equation*}
where $\Delta^P_{j_{\bx},j_t}$ and $\Delta^Q_{j_{\bx},j_t}$ are defined in \eqref{eq:20} and \eqref{eq:66}, respectively. 
\end{proposition}
\begin{proof}
Using 
\begin{equation*}
    P_{j_{\bx},j_t}(\Pi_{j_{\bx}}^{\nabla_{\bx}} \otimes \Pi_{j_t}^{\partial_t^2}) = \Pi_{j_{\bx}}^{\nabla_{\bx}} \otimes \Pi_{j_t}^{\partial_t^2},
\end{equation*}
we compute
\begin{equation} \label{eq:71}
\begin{aligned}
    \Delta_{j_{\bx},j_t}^P - \Delta_{j_{\bx},j_t}^Q & =  P_{j_{\bx},j_t}(\Id -\Pi_{j_{\bx}}^{\nabla_{\bx}} \otimes \Pi_{j_t}^{\partial_t^2}) - P_{j_{\bx}-1,j_t}(\Id -\Pi_{j_{\bx}-1}^{\nabla_{\bx}} \otimes \Pi_{j_t}^{\partial_t^2})
    \\ & \quad - P_{j_{\bx},j_t-1}(\Id -\Pi_{j_{\bx}}^{\nabla_{\bx}} \otimes \Pi_{j_t-1}^{\partial_t^2})+ P_{j_{\bx}-1,j_t-1}(\Id -\Pi_{j_{\bx}-1}^{\nabla_{\bx}} \otimes \Pi_{j_t-1}^{\partial_t^2}).
\end{aligned}
\end{equation}
Then, inserting the identity
\begin{align*}
    \Id-\Pi_{j_{\bx}}^{\nabla_{\bx}} \otimes \Pi_{j_t}^{\partial_t^2}=\Id^{\bx}\otimes(\Id^t-\Pi_{j_t}^{\partial_t^2})+(\Id^{\bx}-\Pi_{j_{\bx}}^{\nabla_{\bx}} )\otimes \Id^t -(\Id^{\bx}-\Pi_{j_{\bx}}^{\nabla_{\bx}} )\otimes (\Id^t-\Pi_{j_t}^{\partial_t^2})
\end{align*}
into~\eqref{eq:71} and reordering the terms yield
\begin{equation}\label{eq:72}
\begin{split}
    \Delta_{j_{\bx},j_t}^P - \Delta_{j_{\bx},j_t}^Q & =
   (P_{j_{\bx},j_t}-P_{j_{\bx}-1,j_t})(\Id^{\bx}\otimes (\Id^t-\Pi_{j_t}^{\partial_t^2}))
   \\ & \quad - (P_{j_{\bx},j_t-1}-P_{j_{\bx}-1,j_t-1})(\Id^{\bx}\otimes (\Id^t-\Pi_{j_t-1}^{\partial_t^2}))
   \\ & \quad + (P_{j_{\bx},j_t}-P_{j_{\bx},j_t-1})((\Id^{\bx}-\Pi^{\nabla_{\bx}}_{j_{\bx}})\otimes \Id^t)
   \\ & \quad - (P_{j_{\bx}-1,j_t}-P_{j_{\bx}-1,j_t-1})((\Id^{\bx}-\Pi^{\nabla_{\bx}}_{j_{\bx}-1})\otimes \Id^t)
   \\ & \quad - P_{j_{\bx},j_t}((\Id^{\bx}-\Pi^{\nabla_{\bx}}_{j_{\bx}})\otimes (\Id^t-\Pi_{j_t}^{\partial_t^2}))
   \\ & \quad + P_{j_{\bx}-1,j_t}((\Id^{\bx}-\Pi^{\nabla_{\bx}}_{j_{\bx}-1})\otimes (\Id^t-\Pi_{j_t}^{\partial_t^2}))
   \\ & \quad + P_{j_{\bx},j_t-1}((\Id^{\bx}-\Pi^{\nabla_{\bx}}_{j_{\bx}})\otimes (\Id^t-\Pi_{j_t-1}^{\partial_t^2}))
   \\ & \quad - P_{j_{\bx}-1,j_t-1}((\Id^{\bx}-\Pi^{\nabla_{\bx}}_{j_{\bx}-1})\otimes (\Id^t-\Pi_{j_t-1}^{\partial_t^2})).
   \end{split}
\end{equation}
Due to the coercivity in~\eqref{eq:8}, consistency, and the definitions of $\Pi^{\nabla_{\bx}}_{j_{\bx}}$ and $\Pi_{j_t}^{\partial_t^2}$, we compute
\begin{align*}
    & \| P_{j_{\bx},j_t}((\Id^{\bx} -\Pi^{\nabla_{\bx}}_{j_{\bx}})\otimes (\Id^t-\Pi_{j_t}^{\partial_t^2}))u\|^2_{\mathcal{V}(Q_T)}
    \\ & \quad \lesssim \mathcal{A}( P_{j_{\bx},j_t}((\Id^{\bx}-\Pi^{\nabla_{\bx}}_{j_{\bx}})\otimes (\Id^t-\Pi_{j_t}^{\partial_t^2}))u, P_{j_{\bx},j_t}((\Id^{\bx}-\Pi^{\nabla_{\bx}}_{j_{\bx}})\otimes (\Id^t-\Pi_{j_t}^{\partial_t^2}))u)
    \\ & \quad = \mathcal{A}((\Id^{\bx}-\Pi^{\nabla_{\bx}}_{j_{\bx}})\otimes (\Id^t-\Pi_{j_t}^{\partial_t^2})u, P_{j_{\bx},j_t}((\Id^{\bx}-\Pi^{\nabla_{\bx}}_{j_{\bx}})\otimes (\Id^t-\Pi_{j_t}^{\partial_t^2}))u)
    \\ & \quad = 0.
\end{align*}
In a similar way, one can show that also the contributions of last three terms in~\eqref{eq:72} to~$(\Delta_{j_{\bx},j_t}^P - \Delta_{j_{\bx},j_t}^Q)u$ are zero. The four remaining terms are estimated combining the results in Proposition~\ref{prop:6},  Corollary~\ref{cor:1}, which is valid under the \vanishingconditions \eqref{eq:25}, and Corollary~\ref{cor:2}, which is valid under the \vanishingconditions \eqref{eq:24}, together with the approximation properties~\eqref{eq:13} and~\eqref{eq:14}. Indeed, setting~$z:=(\Id^{\bx}\otimes (\Id^t-\Pi_{j_t}^{\partial_t^2}))u $ and~$w := ((\Id^{\bx}-\Pi^{\nabla_{\bx}}_{j_{\bx}})\otimes \Id^t) u$, we get
\begin{align*}
    \|(P_{j_{\bx},j_t}-P_{j_{\bx}-1,j_t})z \|_{L^2(Q_T)} 
    & \lesssim 2^{-j_{\bx} (p_{\bx}+1)}\left(\|P_{\infty,j_t}z\|_{H^{p_{\bx}+1}(\Omega)\otimes L^2(0,T)}+ \| \partial_t \mathcal{B}_{j_t} P_{\infty,j_t}z\|_{H^{p_{\bx}+1}(\Omega)\otimes L^2(0,T)}\right)
    \\ & \lesssim 2^{-j_{\bx} (p_{\bx}+1)}\left(\|z\|_{H^{p_{\bx}+2}(\Omega)\otimes L^2(0,T)}+ \| z \|_{H^{p_{\bx}+4}(\Omega)\otimes L^2(0,T)}\right)
    \\ & \lesssim  2^{-j_{\bx} (p_{\bx} +1)} \| z\|_{H^{p_{\bx}+4}(\Omega)\otimes L^2(0,T)}
    \\ & \lesssim  2^{-j_{\bx} (p_{\bx} +1)-j_t (p_t+1)} \|u\|_{H^{p_{\bx}+4}(\Omega)\otimes H^{p_t+3}(0,T)}
\end{align*}
and also
\begin{align*}
    \| (P_{j_{\bx},j_t} - P_{j_{\bx},j_t-1})w \|_{L^2(Q_T)} & \lesssim 2^{-j_t (p_t+1)}\left(\|P_{j_{\bx},\infty}w\|_{L^2(\Omega)\otimes H^{p_t+3}(0,T)}+\| \mathcal{L}_{j_{\bx}} P_{j_{\bx},\infty}w\|_{L^2(\Omega)\otimes H^{p_t+3}(0,T)}\right)
    \\ & \lesssim
    2^{-j_t (p_t+1)} (\|w\|_{L^2(\Omega)\otimes H^{p_t+4}(0,T)}+ \|w\|_{L^2(\Omega)\otimes  H^{p_t+6}(0,T)})
    \\ & \lesssim
    2^{-j_t (p_t+1)} \|w\|_{L^2(\Omega)\otimes H^{p_t+6}(0,T)}
    \\ & \lesssim
    2^{-j_t (p_t+1)-\jx (p_{\bx}+1)}\|u\|_{H^{p_{\bx}+1}(\Omega)\otimes H^{p_t+6}(0,T)},
\end{align*}
and the proof is complete.
\end{proof}
\begin{remark}[\Vanishingconditions and loss of convergence rates] If the \vanishingconditions \eqref{eq:24} and~\eqref{eq:25} required in Corollaries~\ref{cor:1} and~\ref{cor:2} are not satisfied, we expect a deterioration in the convergence rate stated in Proposition~\ref{prop:7}.
In particular, estimates of the initial-time terms in~\eqref{eq:45} and~\eqref{eq:46}, together with~\eqref{eq:47}, could give a precise quantification of this loss when some of the initial \vanishingconditions are not satisfied.
Quantifying the reduction in the convergence rates when the \vanishingconditions on the boundary are not satisfied would require an explicit characterization of the difference between the~$\mathcal{H}^{s_{\bx}}(\Omega)$ and~$H^{s_{\bx}}(\Omega)$ norms. We recall that these two norms are equivalent when the \vanishingconditions on the boundary~\eqref{eq:25} are satisfied~\cite[Lemma 3.1]{thomee2006galerkin}.
\end{remark}

\begin{remark}
[Are the \vanishingconditions of Proposition~\ref{prop:7} necessary for optimal convergence?]\label{rem:Necessity} To show that the convergence rate stated in Proposition~\ref{prop:7} may deteriorate when not all \vanishingconditions are satisfied, we numerically study the decay of the details~$\Delta^P_{j_{\bx},\jt}u$, $\Delta^Q_{j_{\bx},\jt}u$, and of their difference~$(\Delta^P_{\jx,\jt}-\Delta^Q_{j_{\bx},\jt}) u$, as the mesh size~$h= h_{j_\bx}= h_{j_t}$ tends to zero. We consider the case $j_{\bx}=j_t$ and use maximal-regularity splines of degree $p_{x}=1$ in space and $p_t=2$ in time. To this end, we consider the functions
\[
u_1(x,t) = (e^{xt}-1)\sin(\pi x), \quad
u_2(x,t) = (e^{xt}-1-xt)\sin(\pi x), \quad
u_3(x,t) = \big(e^{xt}-1-xt-\tfrac12 x^2 t^2\big)\sin(\pi x)
\]
on $Q_T=(0,1)\times (0,1)$, with $c=1$.
None of these functions satisfies the boundary \vanishingconditions \eqref{eq:25}. Concerning the initial conditions~\eqref{eq:24}, $u_1$ satisfies none, $u_2$ satisfies~\eqref{eq:24} for~$\ell_t=1$, and~$u_3$ satisfies~\eqref{eq:24} for~$\ell_t=1,2$. Figure~\ref{fig:DeltaP_DeltaQ} shows the decay of the $L^2(Q_T)$ norms of the details and of their difference, as functions of the mesh size~$h=2^{-\jt}$, $j_t=1,\dots,6$. In all cases, $\|\Delta^Q_{j_{\bx},\jt}u\|_{L^2(Q_T)}$ decays at the rate $\mathcal{O}(h^{p_{\bx}+p_t+2}) = \mathcal{O}(h^5)$, in agreement with~\eqref{eq:67}. In contrast, $\|\Delta^P_{j_{\bx},\jt}u\|_{L^2(Q_T)}$ and $\|(\Delta^P_{j_{\bx},\jt}-\Delta^Q_{j_{\bx},\jt}) u\|_{L^2(Q_T)}$ exhibit a suboptimal rate. The rate of Proposition~\ref{prop:7} is recovered as these conditions are progressively enforced, as observed for $u_3$. In this case, the Galerkin detail converges with optimal rate even if three of the five initial \vanishingconditions \eqref{eq:24} and all the boundary \vanishingconditions \eqref{eq:25} are not satisfied.
\begin{figure}[ht]
    \includegraphics[width=\linewidth]{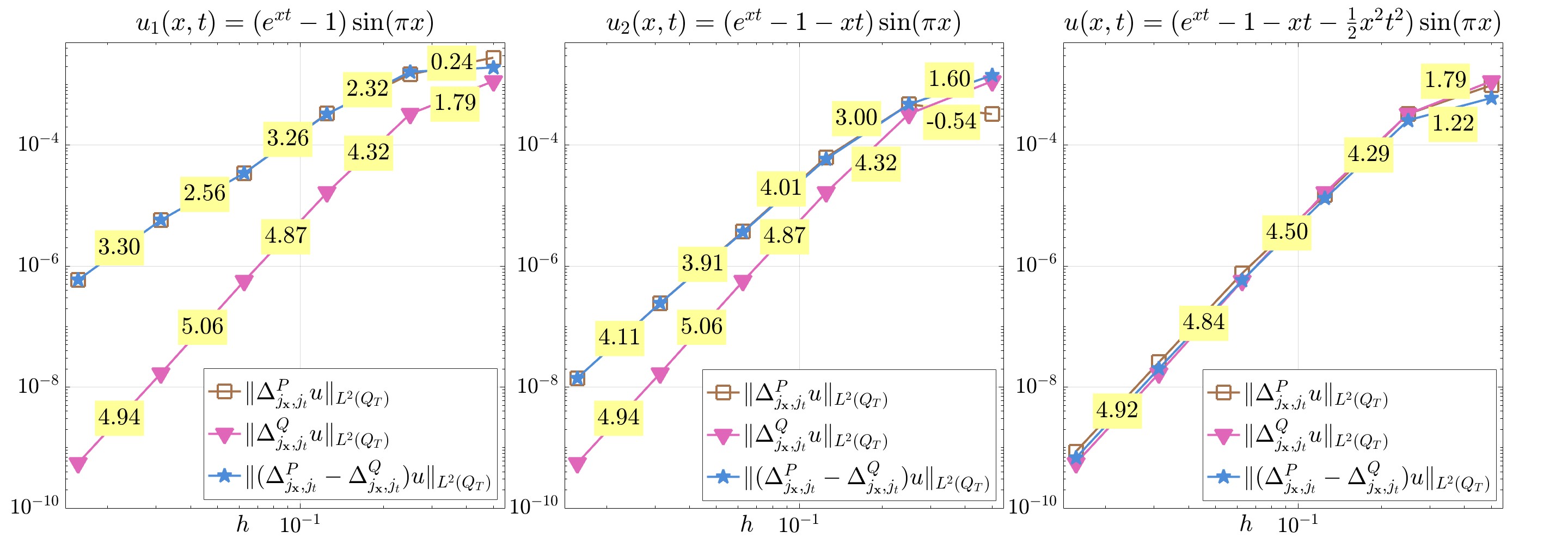}
    \caption{$\|\Delta^P_{j_{\bx},\jt}u\|_{L^2(Q_T)}$, $\|\Delta^Q_{j_{\bx},\jt}u\|_{L^2(Q_T)}$, and~$\|(\Delta^P_{j_{\bx},\jt}-\Delta^Q_{j_{\bx},\jt}) u\|_{L^2(Q_T)}$ for the three functions in Remark~\ref{rem:Necessity} and maximal-regularity splines of degree $p_{\bx}=1$ in space and $p_t=2$ in time.}
    \label{fig:DeltaP_DeltaQ}
\end{figure}
\end{remark}

\subsection{Proof of Theorem~\ref{th:1}}\label{sec:proof}

We can now prove our main theorem.
\begin{proof}[Proof of Theorem \ref{th:1}]
Using the triangle inequality with~$P_{J,J} u$, we write
\begin{equation} \label{eq:triangle}
\begin{aligned}
    \| u - u^{CF}_J \|_{L^2(Q_T)} \le \| u - P_{J,J} u \|_{L^2(Q_T)} + \| P_{J,J} u - u^{CF}_J \|_{L^2(Q_T)}.
\end{aligned}
\end{equation}
The term~$\| u - P_{J,J} u \|_{L^2(Q_T)}$ is estimated with~\eqref{eq:18}. For~$\| P_{J,J} u - u^{CF}_J \|_{L^2(Q_T)}$, we use the second identity in~\eqref{eq:21bis} and the triangle inequality, and obtain
\begin{equation*}
\begin{aligned}
    \| P_{J,J} u - u^{CF}_J \|_{L^2(Q_T)}
    & \le \sum_{\substack{j_{\bx},j_t \le J  \\ j_{\bx}+j_t > J}}  \| \Delta_{j_{\bx},j_t}^P u \|_{L^2(Q_T)}  
    \\ & \le \sum_{\substack{j_{\bx},j_t \le J  \\ j_{\bx}+j_t > J}} \| (\Delta_{j_{\bx},j_t}^P - \Delta_{j_{\bx},j_t}^Q )u \|_{L^2(Q_T)} + \sum_{\substack{j_{\bx},j_t \le J  \\ j_{\bx}+j_t > J}} \| \Delta_{j_{\bx},j_t}^Q u \|_{L^2(Q_T)}
    \\ & \le \sum_{\jx +\jt > J} \| (\Delta_{j_{\bx},j_t}^P - \Delta_{j_{\bx},j_t}^Q )u \|_{L^2(Q_T)} + \sum_{\jx+\jt > J} \| \Delta_{j_{\bx},j_t}^Q u \|_{L^2(Q_T)},
\end{aligned}
\end{equation*}
with~$\Delta_{j_{\bx},j_t}^Q$ defined as in~\eqref{eq:66}. For the first term, using Proposition~\ref{prop:7}, we obtain
\begin{align}
    & \sum_{\jx +\jt> J}  \big\|\big( \Delta_{j_{\bx},j_t}^P - \Delta_{j_{\bx},j_t}^Q \big)u \big\|_{L^2(Q_T)}
    \\ & \label{eq:NormSumSumNorm}
    \qquad \lesssim \big(\|u\|_{H^{p_{\bx}+4}(\Omega)\otimes H^{p_t+3}(0,T)} +  \|u\|_{H^{p_{\bx}+1}(\Omega)\otimes H^{p_t+6}(0,T)} \big) \sum_{\jx +\jt > J} 2^{-j_{\bx} (p_{\bx} +1)-j_t (p_t+1)} .
    \nonumber
\end{align}
The steps above are justified as the series on the right-hand side is convergent.
Analogously, for the second term, we have
\begin{align*}
   \sum_{\jx +\jt > J}  \big\|\Delta_{j_{\bx},j_t}^Q u \big\|_{L^2(Q_T)} \lesssim \|u\|_{H^{p_{\bx}+1}(\Omega)\otimes H^{p_t+3}(0,T)} \sum_{\jx +\jt > J} 2^{-j_{\bx}(p_{\bx}+1)-j_t (p_t+1)},
\end{align*}
where we have used~\eqref{eq:67}. Then, we arrive at 
\begin{align*}
   \| P_{J,J} u - u^{CF}_J \|_{L^2(Q_T)} & \lesssim \big(\|u\|_{H^{p_{\bx}+4}(\Omega)\otimes H^{p_t+3}(0,T)}+\|u\|_{H^{p_{\bx}+1}(\Omega)\otimes H^{p_t+6}(0,T)}\big) \sum_{\jx+\jt > J} 2^{-j_{\bx}(p_{\bx}+1) -j_t (p_t+1)}.
\end{align*}
To complete the proof, it only remains to show that
\begin{equation} \label{eq:73}
    \sum_{\jx+\jt > J} 2^{-j_{\bx}(p_{\bx}+1) -j_t (p_t+1)} \lesssim \begin{cases}
     2^{-J\min\{p_{\bx}+1,p_t+1\}}, &\text{if } p_{\bx}\neq p_t,
    \\ 2^{-J(p_t+1)} J, & \text{if } p_{\bx} =  p_t.
    \end{cases}
\end{equation}
To this end, we split the index set $\{(\jx,\jt)\in \mathbb{N}^2 \mid \jx+\jt>J\}$ into two disjoint sets (see Figure \ref{fig:2})
\begin{equation} \label{eq:74}
\begin{aligned}
   I_1 := \{ (\jx,\jt)\in \mathbb{N}^2 \mid 0 \le \jx \le J, J -\jx< \jt\}, \quad \quad 
   I_2 := \{(\jx,\jt)\in \mathbb{N}^2 \mid J< \jx, 0 \le \jt\}. 
\end{aligned}
\end{equation}
\begin{figure}[htb!]\centering
\begin{tikzpicture}[scale=.6]
\draw[yellow,fill=yellow!50!white,draw=none]
(4,0)--(4,7)--(7,7)--(7,0)--(4,0);
\draw[cyan,fill=cyan!50!white,opacity=0.5,draw=none]
(4,0)--(0,4)--(0,7)--(4,7)--(4,0);
\draw[->](0,0)--(7,0); 
\draw[->](0,0)--(0,7); 
\draw[thick](0,4)--(4,0);
\draw[thick](4,0)--(4,6);
\draw[thick, dashed](4,6)--(4,7);
\draw(-.5,0)node{$0$};
\draw(-.5,4)node{$J$};
\draw(4,-0.5)node{$J$};
\draw(7,-.4)node{$\jx$};
\draw(-.4,7)node{$\jt$};
\draw(2,4)node{$I_1$};
\draw(5.5,4)node{$I_2$};
\end{tikzpicture}
\caption{Index sets associated with the splitting in \eqref{eq:74}.}
\label{fig:2}
\end{figure}
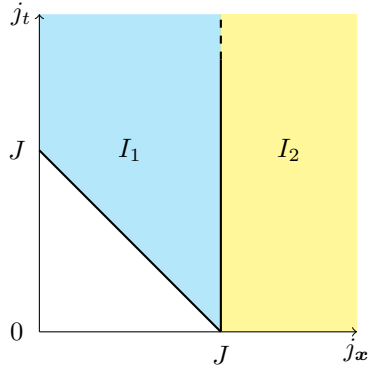

\noindent
Then, we derive
\begin{align*}
   \sum_{\jx+\jt > J} 2^{-j_{\bx}(p_{\bx}+1) -j_t (p_t+1)} & = \sum_{j_{\bx}=0}^{J} \sum_{j_t = J- j_{\bx}+1}^{\infty} 2^{-j_{\bx}(p_{\bx}+1) -j_t(p_t+1)} + \sum_{j_{\bx} = J+ 1}^\infty \sum_{j_t=0}^\infty 2^{-j_{\bx}(p_{\bx}+1) - j_t(p_t+1)}
    \\ &  \lesssim \sum_{j_{\bx}=0}^{J} 2^{-j_{\bx}(p_{\bx}+1) -(J- j_{\bx} +1)(p_t+1)} +\sum_{j_{\bx} = J+ 1}^\infty  2^{-j_{\bx}(p_{\bx}+1)}  
    \\ & \lesssim \sum_{j_{\bx}=0}^{J} 2^{-j_{\bx}(p_{\bx} - p_t) - (J+1)(p_t+1)} + 2^{-(J+1)(p_{\bx}+1)} 
    \\ & = 2^{-(J+1)(p_t+1)} \Big(\sum_{j_{\bx}=0}^{J} 2^{-j_{\bx} (p_{\bx}- p_t)} + 2^{-(J+1)(p_{\bx} - p_t)}\Big)
    \\ &=2^{-(J+1)(p_t+1)} \sum_{j_{\bx}=0}^{J+1} 2^{-j_{\bx} (p_{\bx}- p_t)} =:\mathcal S_J
\end{align*}
We distinguish three cases.
\begin{itemize}
\item[\emph{i)}] If $p_{\bx}< p_t$, we have
\[
\begin{split}
    {\mathcal S}_J & = 2^{-(J+1)(p_t+1)}\Big(\frac{2^{(J+2)({p_t}-p_{\bx})}-1}{2^{({p_t}-p_{\bx})}-1}\Big) \le 2^{-(J+1)(p_t+1)}\Big(\frac{2^{2({p_t}-p_{\bx})}}{2^{({p_t}-p_{\bx})}-1}\Big)2^{J({p_t}-p_{\bx})}
    \\ & 
    \lesssim 2^{-J(p_{\bx}+1)}=2^{-J\min\{p_{\bx}+1,p_t+1\}}.
\end{split}
\]
\item[\emph{ii)}] If $p_{\bx} > p_t$, we have~${\mathcal S}_J \lesssim 2^{-(J+1)(p_t+1)}\lesssim 2^{-J(p_t+1)}=2^{-J\min\{p_{\bx}+1,p_t+1\}}$.
\item[\emph{iii)}] If $p_{\bx} = p_t$, we have~${\mathcal S}_J = 2^{-(J+1)(p_t+1)}(J+2)\lesssim 2^{-J(p_t+1)}J$.
\end{itemize}
This proves~\eqref{eq:73} and the proof of the theorem is complete.
\end{proof}

\section{Numerical experiments} \label{sec:numericalexp}

In this section, we present numerical experiments in $(1+1)$- and $(2+1)$-dimensions to validate the theoretical estimates in Theorem \ref{th:1} and further illustrate the effectiveness of the proposed method. In particular, we confirm the saving in the number of degrees of freedom achieved by the sparse-grid discretization compared to the full-grid scheme, as  discussed in Section \ref{sec:complexity}. Moreover, we investigate whether the \vanishingconditions at~$t=0$ and on~$\partial\Omega$ required by Theorem~\ref{th:1} are strictly necessary for optimal convergence or if they can be relaxed, and we study the effect of the choice of the mesh sizes $h_0^{\bx}$ and $h_0^t$ at level~$0$ in relation to the oscillatory nature of the solutions (see the example in \S\ref{sec:example3}).

\medskip
\noindent 
The method has been implemented using Matlab\footnote{Replication data are available at \url{https://github.com/ChiaraPerinati/Sparse-grid_XT_Wave}.}. All linear systems have been solved with Matlab's backslash direct solver. Differently from Figure~\ref{fig:DeltaP_DeltaQ}, in each experiment of this section the polynomial degrees in space and time coincide $p=p_\bx=p_t$, and similarly do the initial mesh sizes $h_0 = h_0^{\bx} = h_0^t$. Recalling that $h_{j_\bx} = h_0 2^{-j_\bx}$ and $h_{j_t} = h_0 2^{-j_t}$ for $0 \le j_\bx, j_t \le J$ (see Section~\ref{sec:numericalscheme}), we consider the same finest mesh parameter $h_J = h_0 2^{-J}$ in space and in time, and study convergence by increasing~$J$, which simultaneously halves the mesh parameter $h_J$ in both directions. The full-grid solution~$P_{J,J}u$ is computed on a tensor-product mesh with mesh parameter~$h_J$ in both space and time. The sparse-grid solution~$u_J^{CF}$, obtained via the combination formula~\eqref{eq:19}, has coarsest and finest mesh parameters $h_0$ and $h_J$, respectively, in both space and time. We plot the errors against~$h_J$ and the value of~$h_0$ is specified in each example.

\subsection{\texorpdfstring{$(1+1)$}{1+1}-dimensional example satisfying \vanishingconditions}\label{sec:example1}

We consider problem~\eqref{eq:1} on the spatial domain $\Omega = (-1,1)$ with final time $T=1$, constant wave velocity $c = 1$, and homogeneous initial data ($u_0 = v_0 = 0$).The source term $f$ is chosen such that the exact solution is
\begin{equation}\label{eq:exp1}
 	u(x,t) =t^6 \sin\left(\pi x\right) \sin^2\left(\frac54\pi t\right).
\end{equation}
This function is smooth and satisfies
\begin{align*} 
    \partial_t^{\ell_t} u(\cdot,0)&  =0 \quad \text{in } \Omega \quad \text{for } \ell_t=1,\dots,7,\\
    \partial_x^{2\ell_{\bx}} u (\cdot,t)& =0 \quad \text{on } \partial \Omega \times (0,T) \quad \text{for } \ell_{\bx}\ge 1.
\end{align*}
Hence, the \vanishingconditions \eqref{eq:24}--\eqref{eq:25} required for optimal convergence rates hold for all polynomial degrees $p = p_{\bx} = p_t \le 4$. We fix $h_0=2^{-1}$, and vary the finest level~$J$ from 1 to 6.

\noindent 
We first choose $V^{t}_{\jt}(0,T)$ as the space of $C^1$-regular splines in time and $V^{\bx}_{\jx}(\Omega)$ as the space of maximal-regularity splines in space, both with even polynomial degrees $p = 2, 4$. With this choice of spaces, assumptions~\eqref{eq:13}--\eqref{eq:14} hold (see Remark~\ref{rem:1}), so the full-grid approximation satisfies the optimal convergence rates in estimate~\eqref{eq:16}. Moreover, since the exact solution~\eqref{eq:exp1} satisfies the \vanishingconditions at the initial time and on the boundary, Theorem~\ref{th:1} applies, and the discrete solution obtained by the sparse-grid method achieves optimal convergence rates as well.

\medskip
\noindent
The left panel of Figure~\ref{fig:exp1C1} shows the $L^2(Q_T)$ relative errors plotted against $h_J$ for both the full- and sparse-grid methods, and $J=1,\ldots,6$, confirming the expected rate $\mathcal{O}(h_J^{p+1})$.
We also plot the difference $\|P_{J,J}u - u_J^{CF}\|_{L^2(Q_T)}$ between the full- and the sparse-grid approximations, which decays at the same optimal rate. 
We recall that the finest mesh sizes in space and time used for a given sparse-grid solution $u_J^{CF}$ coincide with the space and time mesh sizes of the full-grid solution $P_{J,J}u$ plotted at the same abscissa. Some of the meshes used are shown in Figure~\ref{fig:MeshesEx1}.

\medskip
\noindent 
The right panel of Figure~\ref{fig:exp1C1} compares the errors against the total number of degrees of freedom $N_{\mathrm{DoFs}}$. The sparse-grid approximation achieves better accuracy for a comparable number of degrees of freedom.
In particular, the error decay observed is consistent with the rate $\mathcal{O}( (N_{\mathrm{DoFs}})^{-(p+1)}(\log_2 (N_{\mathrm{DoFs}}))^{p+2})$ predicted by \eqref{eq:26a} for the sparse-grid approximations, whereas the full-grid method exhibits only a rate of order $\mathcal{O}((N_{\mathrm{DoFs}})^{-\frac{p+1}{2}})$, in agreement with the complexity estimates~\eqref{eq:26a}.

\begin{figure}[htb!]
    \includegraphics[width=0.49\linewidth]{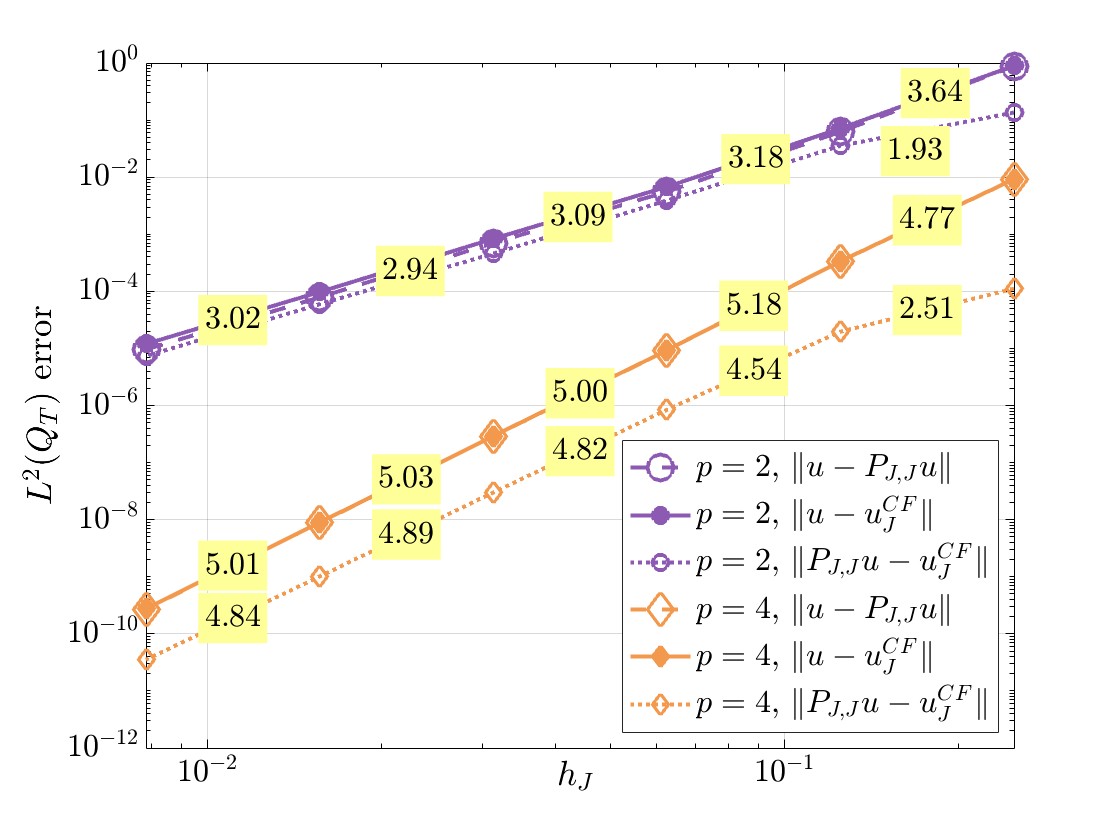}
    \includegraphics[width=0.49\linewidth]{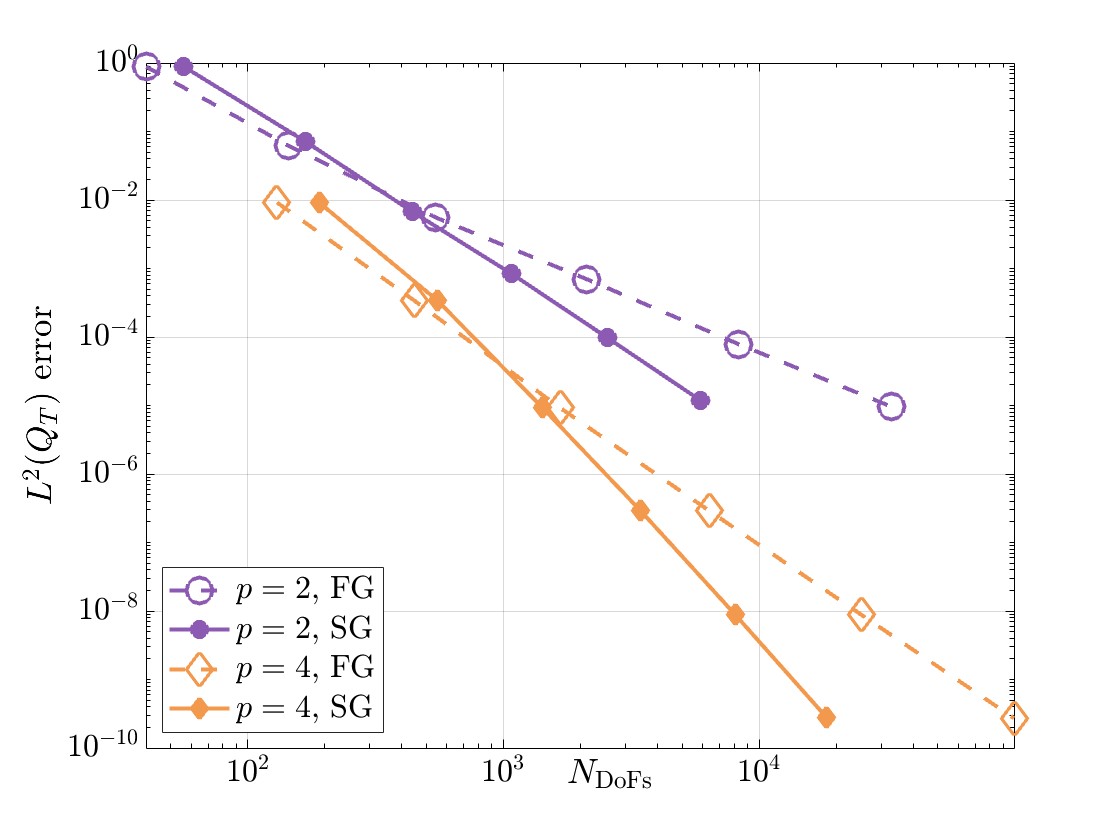}
    \caption{Example in \S\ref{sec:example1}. $L^2(Q_T)$ relative errors with $C^1$-regular splines in time and maximal-regularity splines in space for the exact solution as in~\eqref{eq:exp1}, plotted with respect to $h_J$ (left panel) and $N_{\mathrm{DoFs}}$ (right panel).
    Comparison between full-grid (FG) and sparse-grid (SG) approximations.}
    \label{fig:exp1C1}
\end{figure}
\noindent

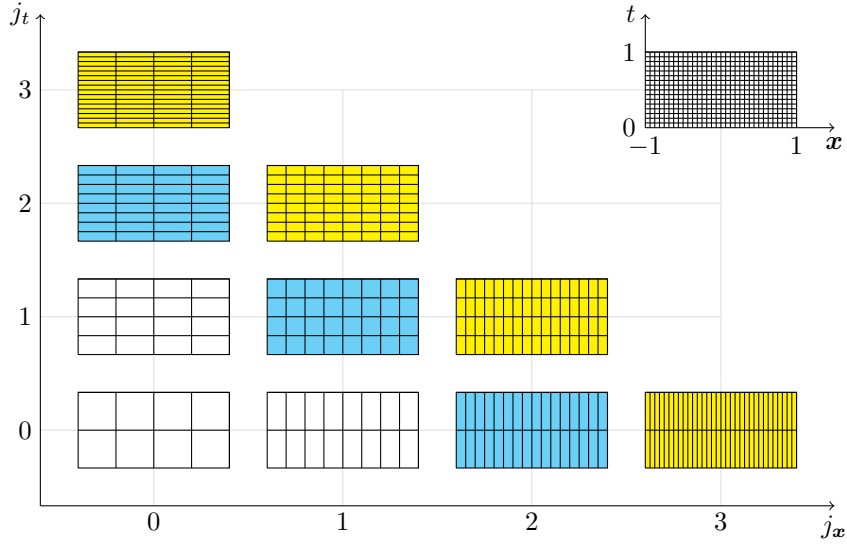
\begin{figure}[ht]
\newcommand{\SGrec}[1]{\draw[#1] (\SGx,\SGy)--(\SGx+2,\SGy)--(\SGx+2,\SGy+1)--(\SGx,\SGy+1)--(\SGx,\SGy);}
\newcommand{\SGvgrid}[1]{ \foreach \k in {1,...,#1}{\draw(\SGx+\k*2/#1,\SGy)--(\SGx+\k*2/#1,\SGy+1);} }
\newcommand{\SGhgrid}[1]{ \foreach \k in {1,...,#1}{\draw(\SGx,\SGy+\k/#1)--(\SGx+2,\SGy+\k/#1);} }
\centering
\begin{tikzpicture}[scale=1] 
\draw[lightgray!50!white](-.5,.5)--(8.5,.5);
\draw[lightgray!50!white](-.5,2)--(8.5,2);
\draw[lightgray!50!white](-.5,3.5)--(8.5,3.5);
\draw[lightgray!50!white](-.5,5)--(8.5,5);
\draw[lightgray!50!white](1,-.5)--(1,5);
\draw[lightgray!50!white](3.5,-.5)--(3.5,5);
\draw[lightgray!50!white](6,-.5)--(6,5);
\draw[lightgray!50!white](8.5,-.5)--(8.5,5);
\draw[->](-.5,-.5)--(10,-.5)node[anchor=north]{$j_\bx$};
\draw[->](-.5,-.5)--(-.5,6)node[anchor=east]{$j_t$};
\def\SGx{0};\def\SGy{0};\SGrec{}\SGvgrid{4}\SGhgrid{2}
\def\SGx{2.5};\def\SGy{0};\SGrec{}\SGvgrid{8}\SGhgrid{2}
\def\SGx{5};\def\SGy{0};\SGrec{fill=cyan!50!white}\SGvgrid{16}\SGhgrid{2}
\def\SGx{7.5};\def\SGy{0};\SGrec{fill=yellow}\SGvgrid{32}\SGhgrid{2}
\def\SGx{0};\def\SGy{1.5};\SGrec{}\SGvgrid{4}\SGhgrid{4}
\def\SGx{2.5};\def\SGy{1.5};\SGrec{fill=cyan!50!white}\SGvgrid{8}\SGhgrid{4}
\def\SGx{5};\def\SGy{1.5};\SGrec{fill=yellow}\SGvgrid{16}\SGhgrid{4}
\def\SGx{0};\def\SGy{3};\SGrec{fill=cyan!50!white}\SGvgrid{4}\SGhgrid{8}
\def\SGx{2.5};\def\SGy{3};\SGrec{fill=yellow}\SGvgrid{8}\SGhgrid{8}
\def\SGx{0};\def\SGy{4.5};\SGrec{fill=yellow}\SGvgrid{4}\SGhgrid{16}
\def\SGx{7.5};\def\SGy{4.5};\SGrec{}\SGvgrid{32}\SGhgrid{16}
\draw[->](9.5,4.5)--(10,4.5)node[anchor=north]{$\bx$};
\draw[->](7.5,5.5)--(7.5,6)node[anchor=east]{$t$};
\draw(1,-.7)node{0};
\draw(3.5,-.7)node{1};
\draw(6,-.7)node{2};
\draw(8.5,-.7)node{3};
\draw(-.7,.5)node{0};
\draw(-.7,2)node{1};
\draw(-.7,3.5)node{2};
\draw(-.7,5)node{3};
\draw(7.5,4.5) node[anchor=north]{$-1$};
\draw(9.5,4.5) node[anchor=north]{$1$};
\draw(7.5,4.5) node[anchor=east]{$0$};
\draw(7.5,5.5) node[anchor=east]{$1$};
\end{tikzpicture}
\caption{Some of the meshes used in the example in \S\ref{sec:example1}.
Each rectangle depicts the space--time cylinder $\Omega_T=(-1,1)\times(0,1)$.
For~$J=3$, the sparse-grid solution $u_3^{CF}$ is the sum of the four Galerkin solutions computed on the yellow meshes, minus the sum of the three Galerkin solutions on the blue meshes. The full-grid solution $P_{3,3}u$ is computed on the fine mesh in the upper-right corner. The errors committed by these approximations are comparable and correspond to the third column of markers from the right ($J=3$) in the left panel of Figure~\ref{fig:exp1C1}.
}
\label{fig:MeshesEx1}
\end{figure}

\noindent 
We repeat the experiment using maximal-regularity splines in both space and time for $p = 2, 3, 4$. In this case, assumption~\eqref{eq:14} is only conjectured (see Remark~\ref{rem:1}). Nevertheless, as shown in Figure~\ref{fig:exp1max}, left panel, the numerical results still exhibit optimal convergence rates $\mathcal{O}(h^{p+1}_J)$ in the $L^2(Q_T)$ norm for both the full- and the sparse-grid methods. 
In the right panel of Figure~\ref{fig:exp1max}, we plot the errors against $N_{\mathrm{DoFs}}$, showing the same reduction in computational cost as in the previous case.
\begin{figure}[htb!]
    \includegraphics[width=0.49\linewidth]{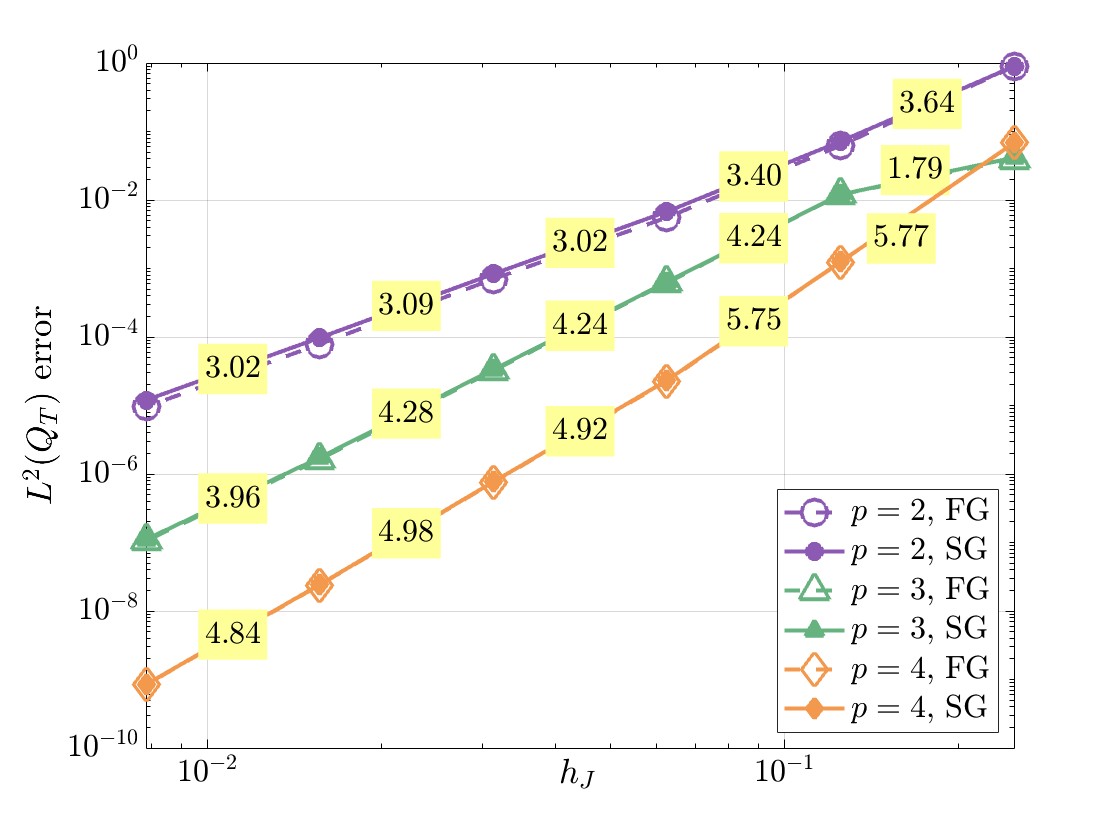}
    \includegraphics[width=0.49\linewidth]{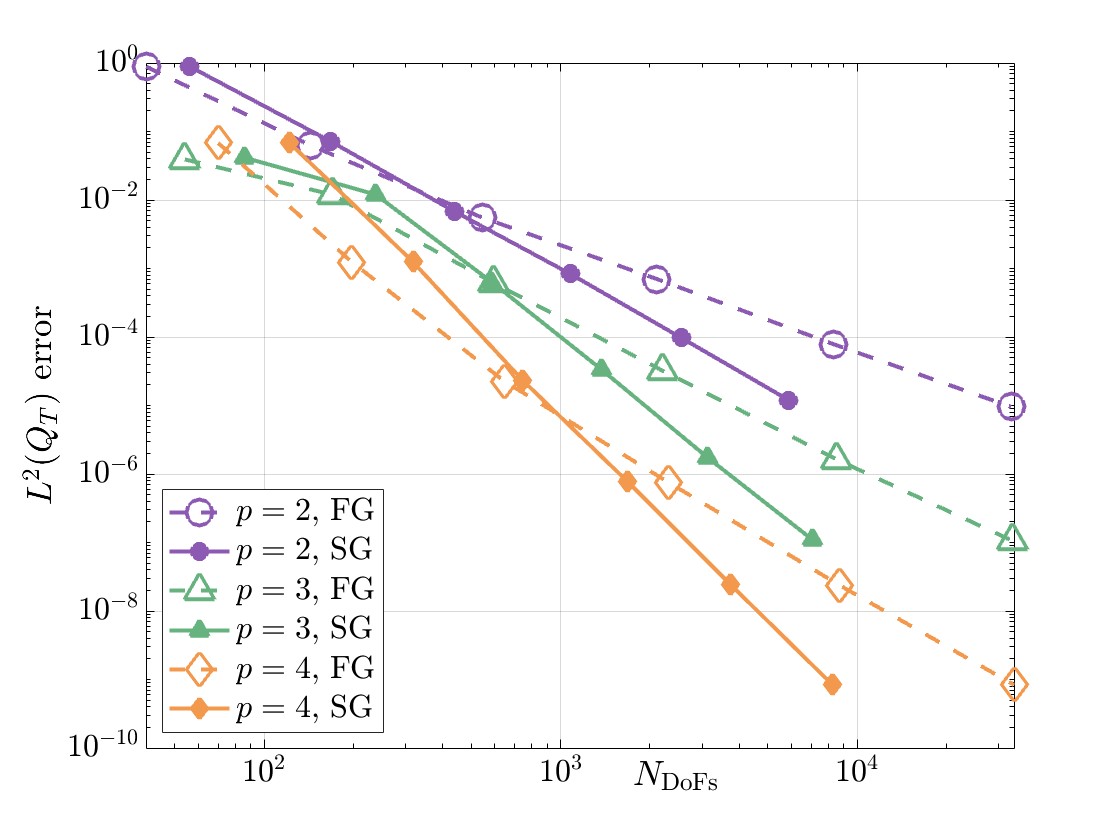}
    \caption{Example in \S\ref{sec:example1}. $L^2(Q_T)$ relative errors with maximal-regularity splines both in space and in time for the exact solution as in~\eqref{eq:exp1}, plotted with respect to $h_J$ (left panel) and $N_{\mathrm{DoFs}}$ (right panel). Comparison between full-grid (FG) and sparse-grid (SG) approximations.} 
    \label{fig:exp1max}
\end{figure}

\subsection{\texorpdfstring{$(1+1)$}{1+1}-dimensional example not satisfying \vanishingconditions}\label{sec:example2} 

We consider the model problem \eqref{eq:1} with the spatial domain $\Omega = (0,1)$, the final time $T=1$, and the wave velocity $c^2(x) = 1+x$. The initial data $u_0, v_0$ and the source term $f$ are chosen such that the exact solution is
\begin{align}
	u(x,t) & = e^{xt}x(1-x),\label{eq:exp52a}
\end{align}
The solution is smooth but does not satisfy any of the \vanishingconditions \eqref{eq:24}--\eqref{eq:25}.

\medskip
\noindent 
We discretize with maximal-regularity splines both in space and time with polynomial degrees $p=2,3,4$, and consider uniform refinements $h_J = h_0 \, 2^{-J}$ with $h_0 = 2^{-1}$ for $J = 1,\dots,5$.

\medskip
\noindent
Even though the \vanishingconditions are not satisfied, the numerical results still exhibit optimal convergence rates $\mathcal{O}(h_J^{p+1})$ in the $L^2(Q_T)$ norm, as shown in Figure~\ref{fig:exp2x(1-x)}, for both the full- and the sparse-grid methods. The right panel confirms that, as the mesh is refined and the number of degrees of freedom grows, the sparse-grid approximation achieves asymptotically the same accuracy as the full-grid solution at a lower computational cost. 

\noindent 
Although conditions~\eqref{eq:24}--\eqref{eq:25} do not appear necessary here to obtain the convergence rates predicted by Theorem~\ref{th:1}, the example in \S\ref{sec:example4} below shows that their absence can lead to deteriorated convergence rates.

\begin{figure}[htb!]
    \centering
    \includegraphics[width=0.49\linewidth]{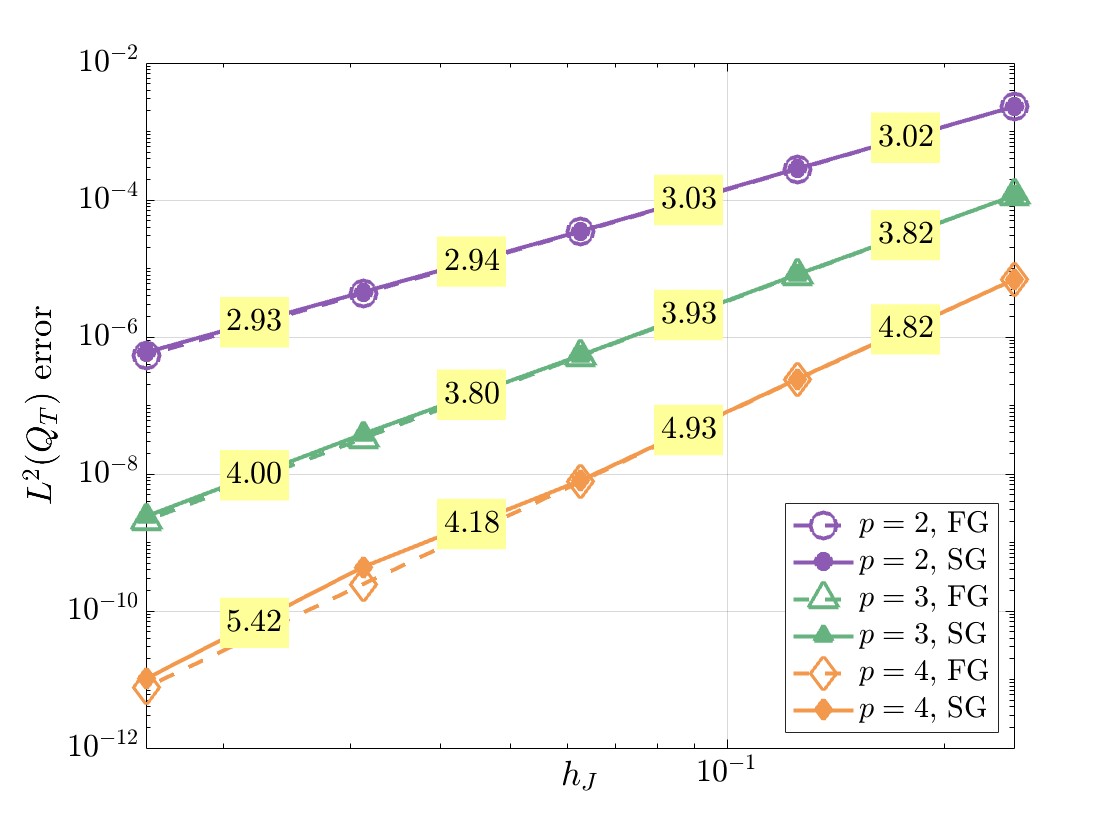}
    \includegraphics[width=0.49\linewidth]{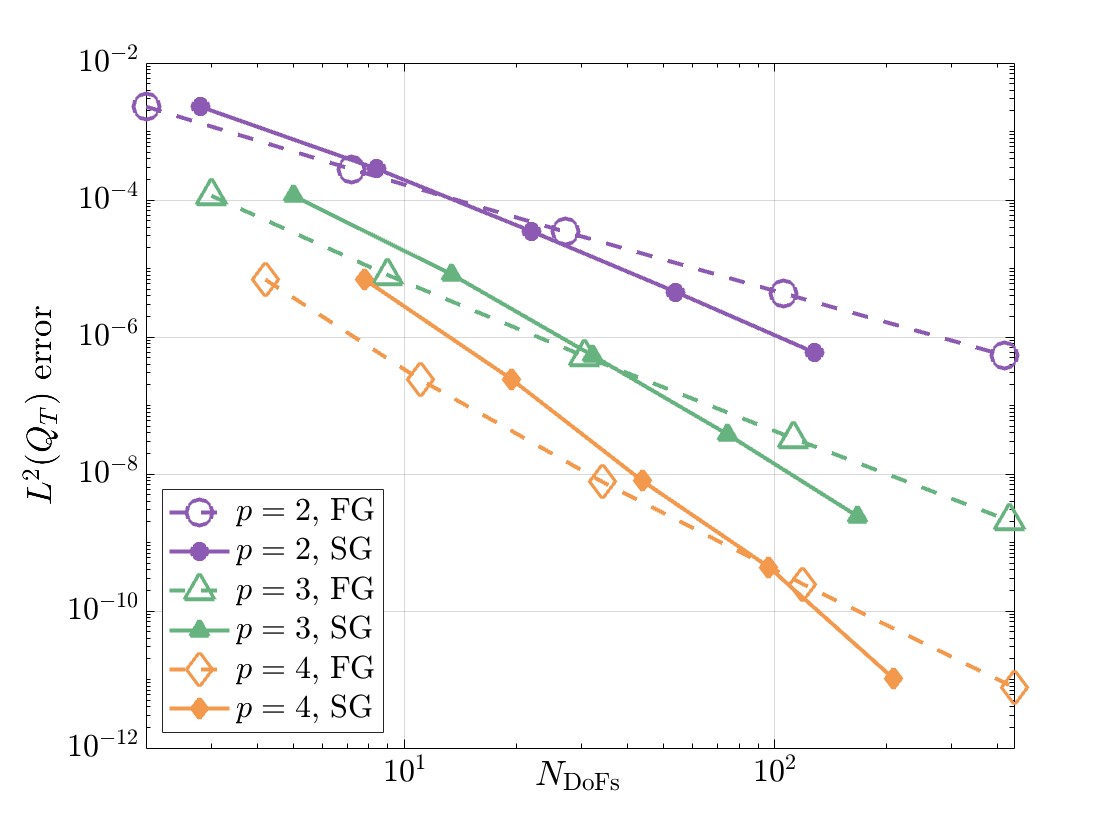}
    \caption{Example in \S\ref{sec:example2}. $L^2(Q_T)$ relative errors with maximal-regularity splines in both space and in time for the exact solution as in~\eqref{eq:exp52a}, plotted with respect to $h_J$ (left panel) and $N_{\mathrm{DoFs}}$ (right panel). Comparison between full-grid (FG) and sparse-grid (SG) approximations.}
    \label{fig:exp2x(1-x)}
\end{figure}

\subsection{\texorpdfstring{$(1+1)$}{1+1}-dimensional traveling wave solution}\label{sec:example3} 

We consider the model problem \eqref{eq:1} with spatial domain $\Omega = (-1,1)$, final time $T=1$, wave velocity $c=\frac45$, and homogeneous source term $f=0$. The initial data $u_0$ and $v_0$ are chosen such that the exact solution is
\begin{equation}\label{eq:exp3}
	u(x,t) = \omega(x-ct)-\omega(2-x-ct),
\end{equation}
where $\omega(s):=e^{-16(s-0.6)^2}$. Homogeneous Dirichlet conditions hold exactly at $x=1$ by construction, while at $x=-1$ they hold in practice since $|u(-1,t)| \leq 10^{-17}$ for all $t \in (0,T)$. This wave field, plotted in Figure~\ref{fig:Energy}, top row, represents a smooth wave traveling and reflecting off the right Dirichlet boundary.

\medskip
\noindent
We discretize with maximal-regularity splines in both space and time, with polynomial degrees $p=2,3,4$.

\medskip
\noindent
We first test how well the sparse-grid method preserves the energy of the solution. The total energy of a function $v\in\mathcal{V}(Q_T)$ at time $s\in [0,T]$ is defined by
\begin{equation*}
    \mathcal{E}(s;v):=\frac12 \Big(\|\partial_t v(\cdot,s)\|_{L^2(\Omega)}^2 +\|c(\cdot)\nabla_{\bx} v(\cdot,s)\|^2_{L^2(\Omega)} \Big).
\end{equation*} 
Figure~\ref{fig:Energy}, bottom row, shows the relative energy errors $|\mathcal{E}(t;u) - \mathcal{E}(t;P_{J,J}u)| / \mathcal{E}(t;u)$ and
$|\mathcal{E}(t;u) - \mathcal{E}(t;u_J^{CF})| / \mathcal{E}(t;u)$ for the full-grid approximation $P_{J,J}u$ (left) and the sparse-grid approximation
$u_J^{CF}$ (right), respectively, for $p=2,3,4$ and $h_{6}= h_0\,2^{-6}$, $h_0=2^{-1}$, at $800$ equispaced time instants in $[0,T]$. The energy error does not increase in time and is uniformly bounded for both methods.
 \begin{figure}[htb!]
    \centering
     \includegraphics[width=0.6\linewidth]{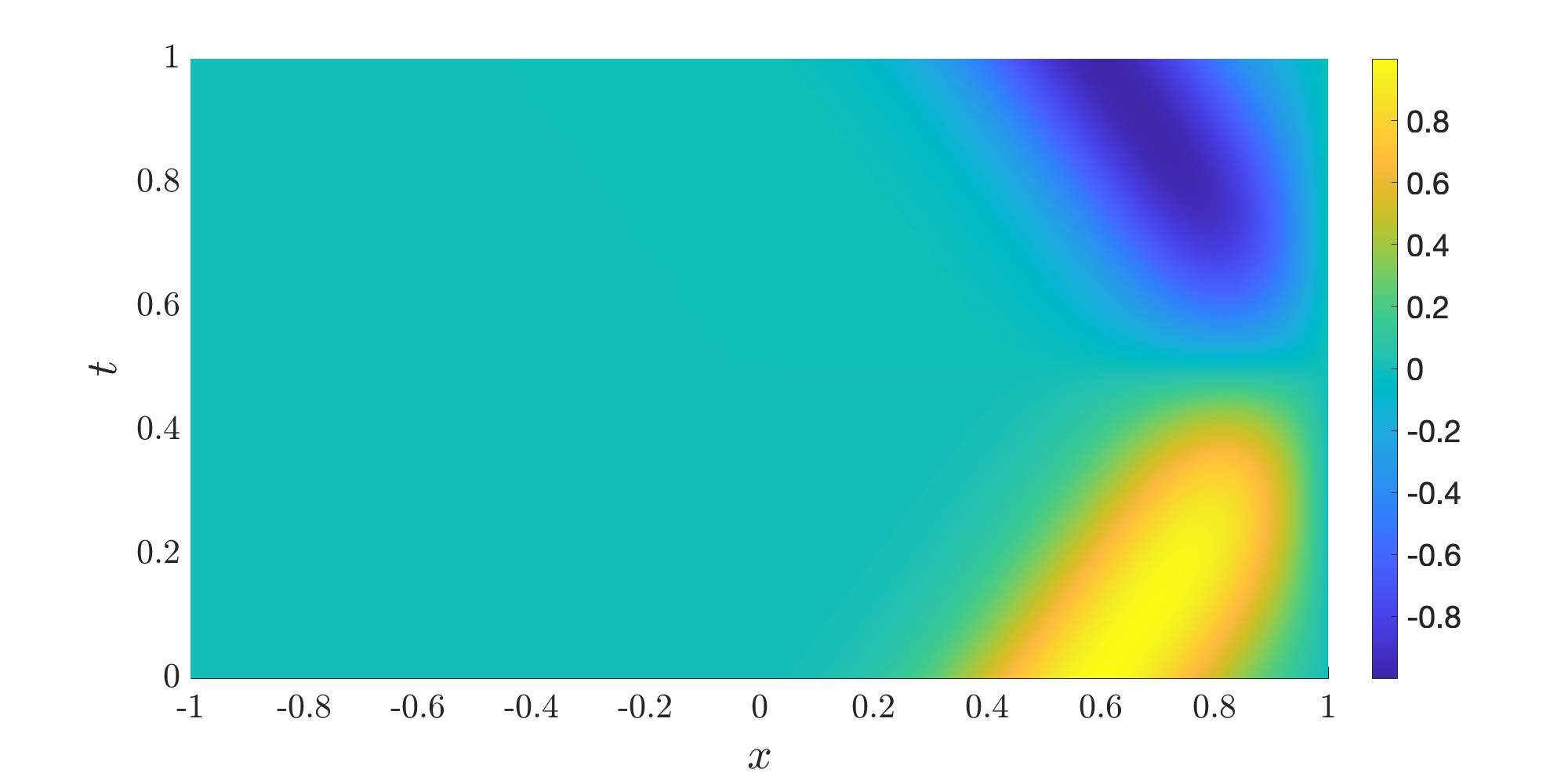}
     \includegraphics[width=0.49\linewidth]{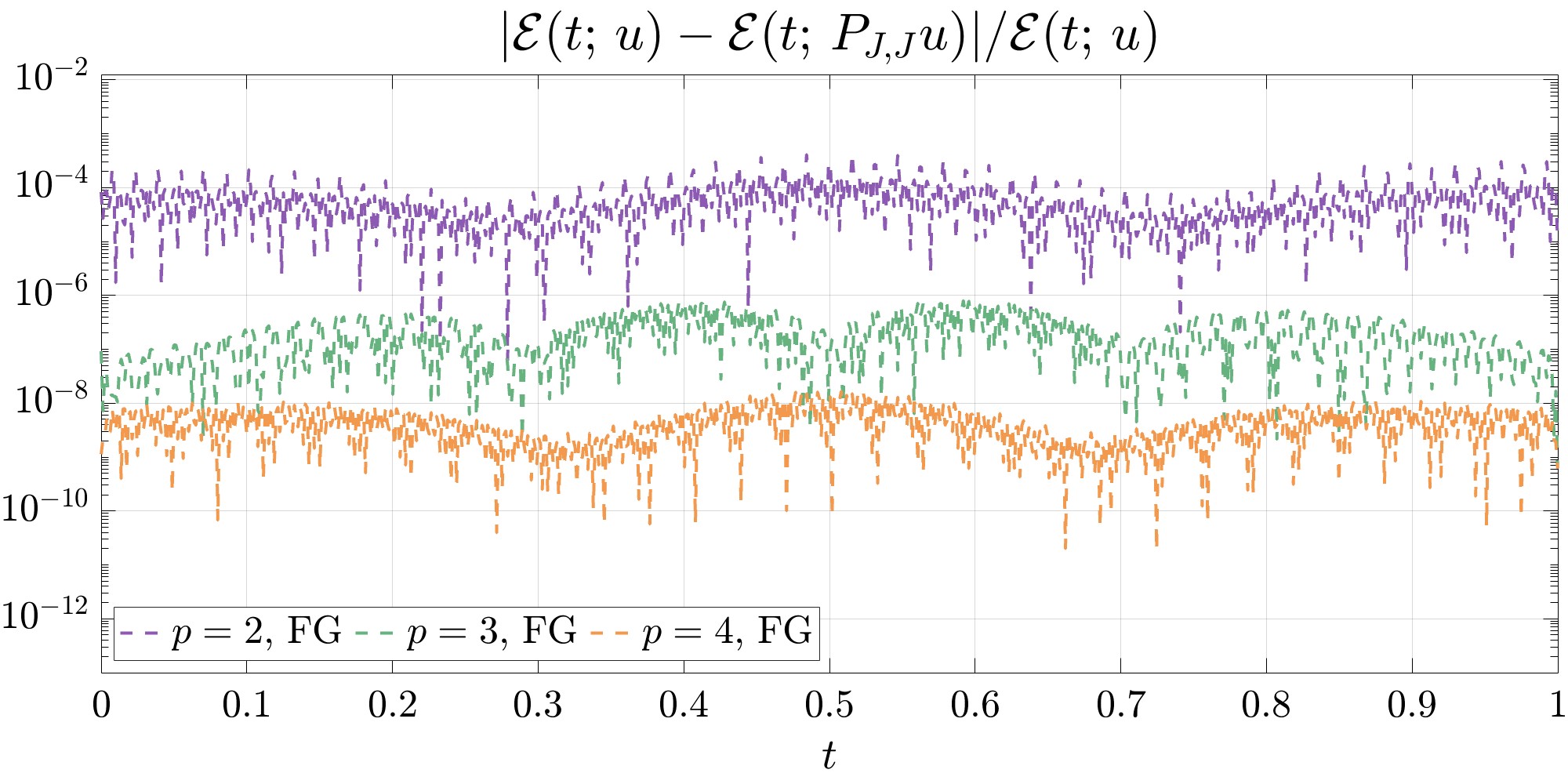}
     \includegraphics[width=0.49\linewidth]{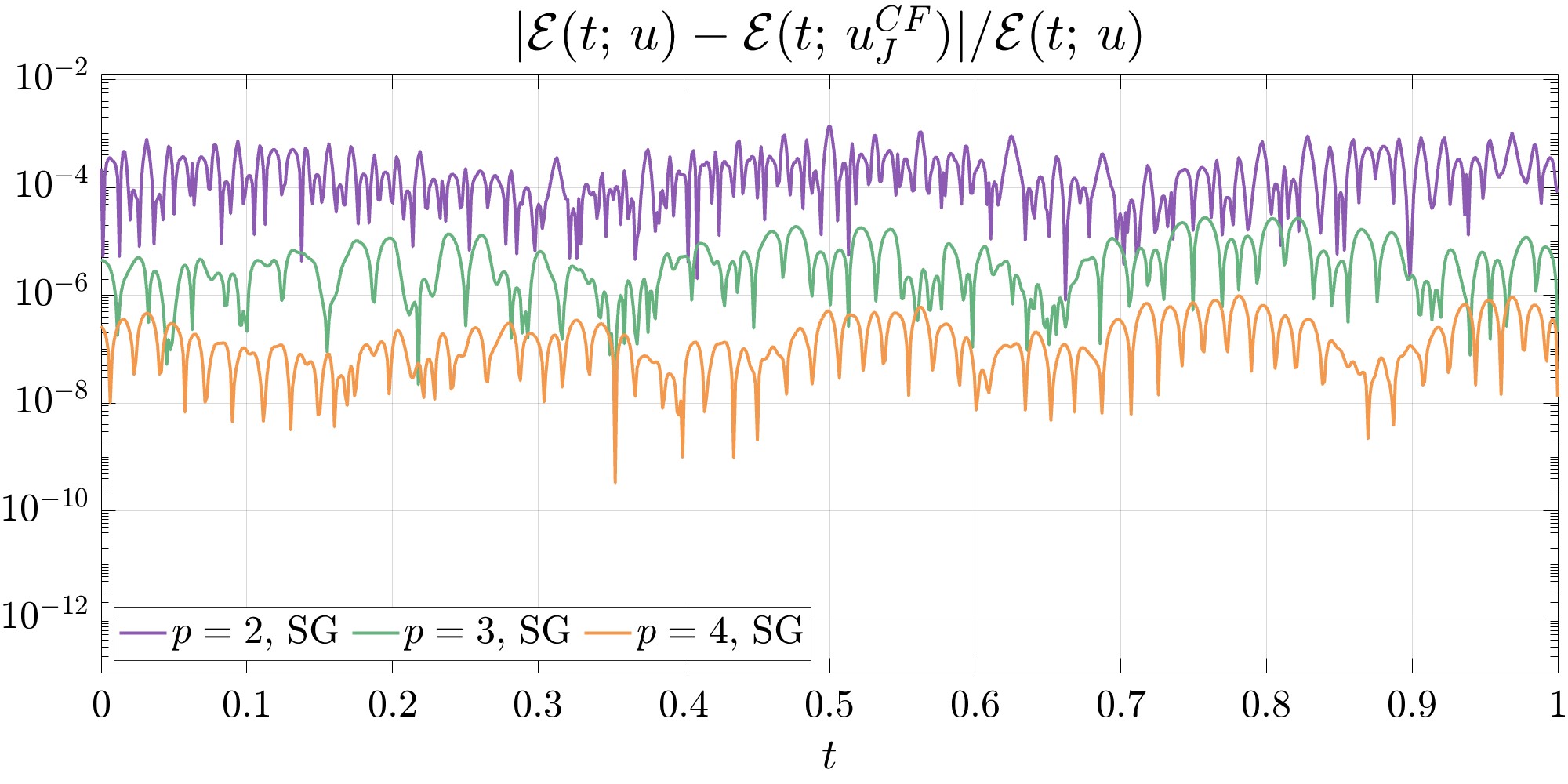}
     \caption{Example in \S\ref{sec:example3}. Top row: exact solution~\eqref{eq:exp3}. Bottom row: relative energy errors $|\mathcal{E}(t;\,u)-\mathcal{E}(t;\,P_{J,J}u)|/\mathcal{E}(t;\,u)$ (full-grid) and $|\mathcal{E}(t;\,u)-\mathcal{E}(t;\,u_J^{CF})|/\mathcal{E}(t;\,u)$ (sparse-grid) for $p=2,3,4$ and $h_6=h_0 2^{-6}$ with $h_0= 2^{-1}$ at 800 equispaced instants in $[0,T]$.}
     \label{fig:Energy}
 \end{figure}

\medskip
\noindent
For the same test case, we also test the effect of the choice of the coarsest mesh 
size~$h_0$ on the sparse-grid approximation. We consider three sequences 
of uniform refinements, each parametrized by~$J$, with the same set of finest mesh sizes $\{h_J = h_0 2^{-J}\}$ but different values of~$h_0$: 
$h_0 = 2^{-1}$ and $J \in \{2,\dots,5\}$ (left plot), 
$h_0 = 2^{-2}$ and $J \in \{1,\dots,4\}$ (central plot), 
$h_0 = 2^{-3}$ and $J \in \{0,\dots,3\}$ (right plot).
Since the sets of finest mesh sizes coincide, the full-grid 
solutions~$P_{J,J}u$ are identical across the three cases and serve as 
a common reference; only the coarser meshes entering the combination 
formula~\eqref{eq:19} differ, as illustrated in the bottom row of 
Figure~\ref{fig:3}. The relative $L^2(Q_T)$ errors are plotted against the mesh size $h_J=h_0 2^{-J}$ in the first row.

\medskip
\noindent
For sparse grids with coarsest
mesh size $h_0 = 2^{-1}$ (left plot), the sparse-grid method does not exhibit the expected convergence rate $\mathcal{O}(h_{J}^{p+1})$ for any of the considered values $p=2,3,4$. A refinement of the initial mesh from $h_0=2^{-1}$ to $h_0=2^{-2}$ (central plot) improves the observed convergence, and further refinement to $h_0=2^{-3}$ (right plot) yields significantly better results. These results suggest that, as expected, refining $h_0$ mitigates the pre-asymptotic behavior associated with the coarse meshes, which is inevitably present in the combination formula \eqref{eq:19}. An interesting open question is to identify criteria for the choice of the coarsest mesh sizes~$h_0^{\bx}$ and~$h_0^t$ ensuring the onset of the asymptotic convergence regime.

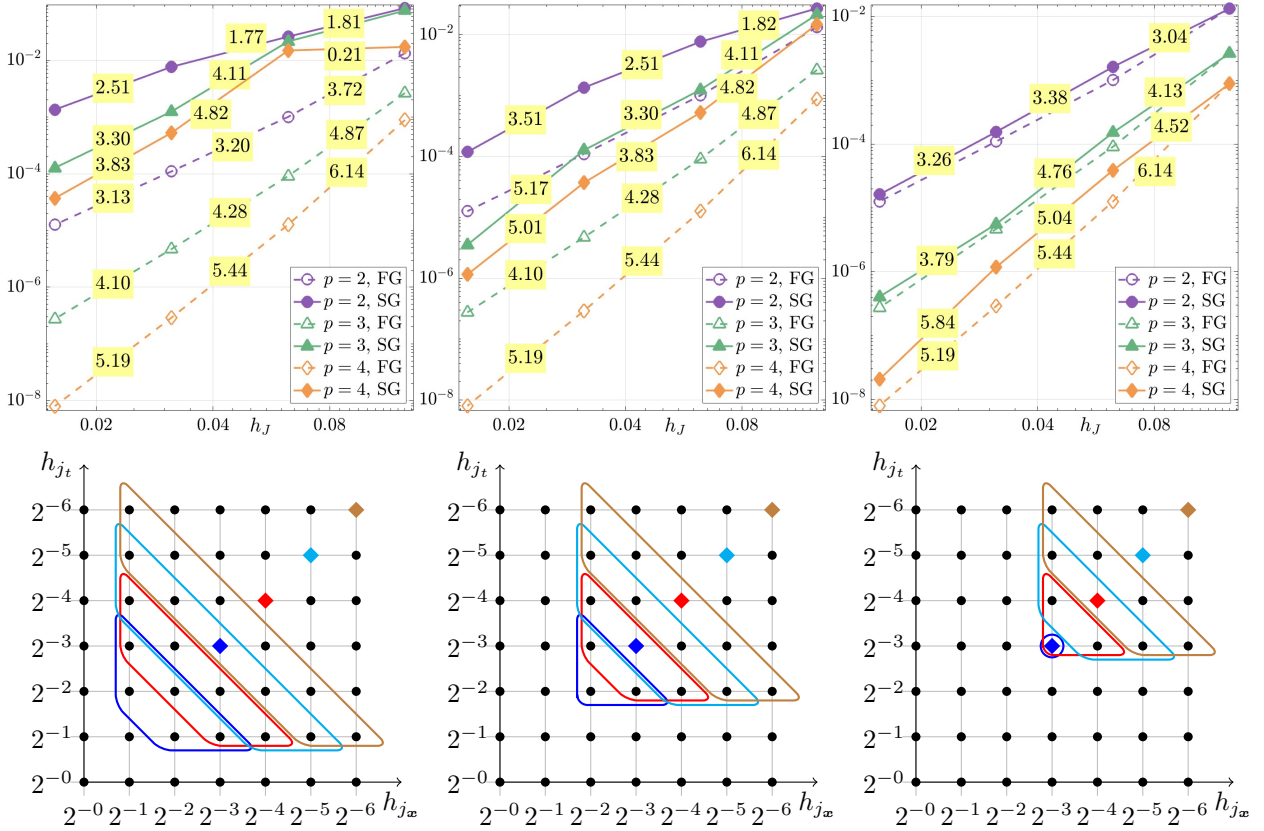
\begin{figure}[!htbp]
    \includegraphics[width=\linewidth]{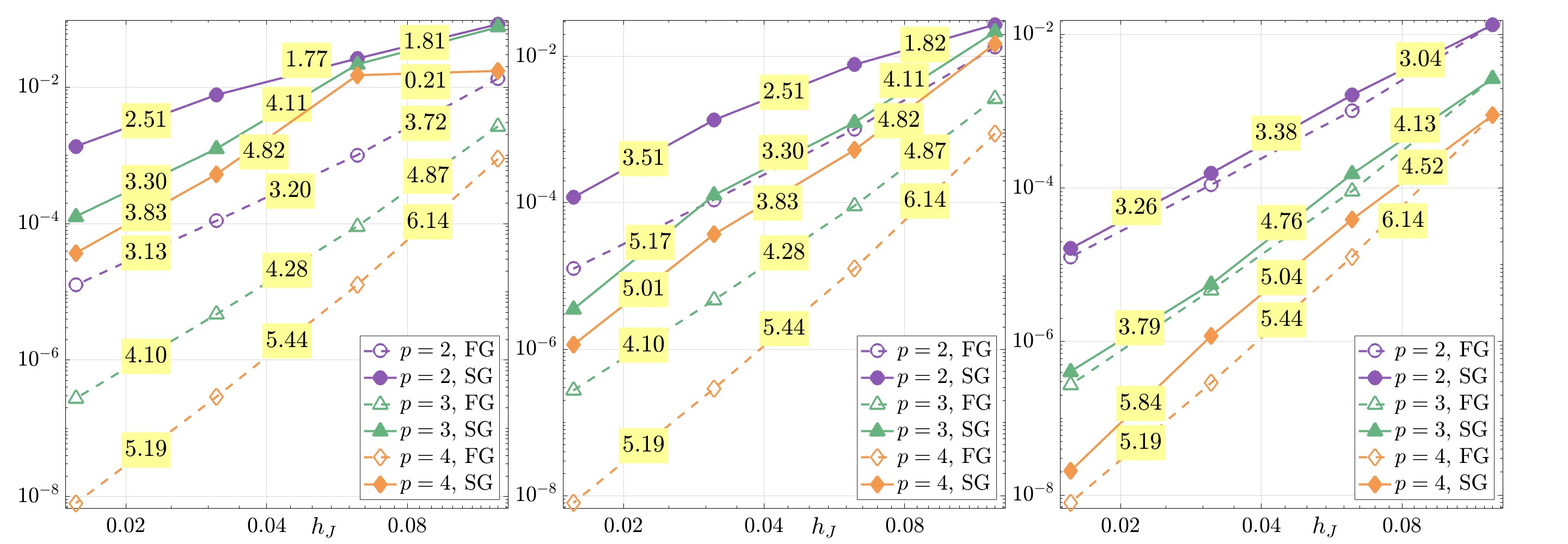}
\newcommand{\SGdiamond}[3]{%
  \fill[#3] (#1, #2-0.18) -- (#1+0.18, #2) -- (#1, #2+0.18) -- (#1-0.18, #2) -- cycle;
}
\newcommand{\SGtrapez}[3]{\draw[thick,rounded corners,#3]
(#1-\SGoffset+0.5,#2-\SGoffset)--(1.5+#1+\SGoffset,#2-\SGoffset)--
(#2-\SGoffset,1.5+#1+\SGoffset)--(#2-\SGoffset,#1-\SGoffset)--
(#1-\SGoffset,#2-\SGoffset)--(#1-\SGoffset+0.5,#2-\SGoffset);}
\centering
\begin{tikzpicture}[scale=.6] 
\foreach \j in {0,...,6}{
\draw[thin,lightgray](-.3,\j)--(6,\j);
\draw[thin,lightgray](\j,-.3)--(\j,6);
\draw(\j,-0.7)node{$2^{-\j}$};
\draw(-0.7,\j)node{$2^{-\j}$};}
\foreach \j in {0,...,6}{\foreach \k in {0,...,6}{\fill(\j,\k) circle(.1);}}
\draw[->](0,0)--(7,0)node[anchor=north]{$h_{\jx}$};
\draw[->](0,0)--(0,7)node[anchor=east]{$h_{\jt}$};
\def\SGoffset{.3}\SGtrapez21{blue}
\def\SGoffset{.2}\SGtrapez31{red}
\def\SGoffset{.3}\SGtrapez41{cyan}
\def\SGoffset{.2}\SGtrapez51{brown}
\SGdiamond{3}{3}{blue}
\SGdiamond{4}{4}{red}
\SGdiamond{5}{5}{cyan}
\SGdiamond{6}{6}{brown}
\end{tikzpicture}
\begin{tikzpicture}[scale=.6] 
\foreach \j in {0,...,6}{
\draw[thin,lightgray](-.3,\j)--(6,\j);
\draw[thin,lightgray](\j,-.3)--(\j,6);
\draw(\j,-0.7)node{$2^{-\j}$};
\draw(-0.7,\j)node{$2^{-\j}$};}
\foreach \j in {0,...,6}{\foreach \k in {0,...,6}{\fill(\j,\k) circle(.1);}}
\draw[->](0,0)--(7,0)node[anchor=north]{$h_{\jx}$};
\draw[->](0,0)--(0,7)node[anchor=east]{$h_{\jt}$};
\def\SGoffset{.3}\SGtrapez22{blue}
\def\SGoffset{.2}\SGtrapez32{red}
\def\SGoffset{.3}\SGtrapez42{cyan}
\def\SGoffset{.2}\SGtrapez52{brown}
\SGdiamond{3}{3}{blue}
\SGdiamond{4}{4}{red}
\SGdiamond{5}{5}{cyan}
\SGdiamond{6}{6}{brown}
\end{tikzpicture}
\centering
\begin{tikzpicture}[scale=.6] 
\foreach \j in {0,...,6}{
\draw[thin,lightgray](-.3,\j)--(6,\j);
\draw[thin,lightgray](\j,-.3)--(\j,6);
\draw(\j,-0.7)node{$2^{-\j}$};
\draw(-0.7,\j)node{$2^{-\j}$};}
\foreach \j in {0,...,6}{\foreach \k in {0,...,6}{\fill(\j,\k) circle(.1);}}
\draw[->](0,0)--(7,0)node[anchor=north]{$h_{\jx}$};
\draw[->](0,0)--(0,7)node[anchor=east]{$h_{\jt}$};
\draw[thick, blue](3,3) circle(.25);
\def\SGoffset{.2}\SGtrapez33{red}
\def\SGoffset{.3}\SGtrapez43{cyan}
\def\SGoffset{.2}\SGtrapez53{brown}
\SGdiamond{3}{3}{blue}
\SGdiamond{4}{4}{red}
\SGdiamond{5}{5}{cyan}
\SGdiamond{6}{6}{brown}
\end{tikzpicture}
\caption{Example in \S\ref{sec:example3}. Top row: $L^2(Q_T)$ relative errors with maximal-regularity splines both in space and in time for the  solution in~\eqref{eq:exp3} and Figure~\ref{fig:Energy}, plotted with respect to $h_J$ with initial level $h_0=2^{-1}$ (left), $h_0=2^{-2}$ (middle), $h_0=2^{-3}$ (right). Comparison between full-grid (FG) and sparse-grid (SG) approximations. Bottom row: meshes used in the corresponding upper plots. Each diamond denotes a full-grid mesh with size $h_J$ both in space and time, while the trapezoid of the same color collects the meshes used in the combination formula for the corresponding refinement level.
Blue diamonds and trapezoids correspond to the rightmost markers in the error plots, red ones correspond to the second-rightmost etc.}
\label{fig:3}
\end{figure}

\subsection{\texorpdfstring{$(2+1)$}{2+1}-dimensional examples}\label{sec:example4} 

We consider here two $(2+1)$-dimensional problems, both with spatial domain $\Omega = (0,1)^2$, final time $T=1$, and wave velocity $c = 1$, with data such that the exact solutions are
\begin{align}
    u_a(x,y,t) &= t^6\sin(\pi x)\sin(\pi y)\sin(txy), \label{eq:exp4_van} \\
    u_b(x,y,t) &= \sin(\pi x)\sin(\pi y) \sin^2(txy),
    \label{eq:exp4_nonvan}
\end{align}
respectively. Both solutions are smooth. The wave field $u_a$ satisfies the \vanishingconditions \eqref{eq:24} at the initial time for all $p \le 3$, while $u_b$ fails to satisfy them for all  $p \ge 1$. We set $h_0=2^{-1}$ and consider a sequence of uniform refinements such that $h_J = h_0 2^{-J}$ for $J=1,\dots,4$. We choose $V^{t}_{j_t}(0,T)$ as the space of maximal-regularity splines in time and $V^{\bx}_{\jx}(\Omega)$ as the space of $C^0$ finite elements on sequences of shape-regular, simplicial nested meshes obtained by successive mesh halving. Polynomial degrees $p=2,3$ are considered.

\noindent 
Figures~\ref{fig:2Dexample_van} and \ref{fig:2Dexample_nonvan} show relative $L^2(Q_T)$ errors versus $h_J$ and $N_{\mathrm{DoFs}}$ for the full- and sparse-grid schemes for $u_a$ and $u_b$, respectively. For~$u_a$, the left panel of Figure~\ref{fig:2Dexample_van} shows that both the full- and sparse-grid methods converge with the optimal rate $\mathcal{O}(h_J^{p+1})$. The sparse-grid scheme achieves better accuracy for a comparable number of degrees of freedom, as shown in the right panel, in agreement with the complexity estimate~\eqref{eq:27a}.
Specifically, the sparse-grid method yields error decay rates of order $\mathcal{O}((N_{\mathrm{DoFs}})^{-\frac{p+1}{d}}(\log_2 (N_{\mathrm{DoFs}})))$, whereas the full-grid method exhibits rates of order $\mathcal{O}((N_{\mathrm{DoFs}})^{-\frac{p+1}{d+1}})$.

\begin{figure}[htb!]
\includegraphics[width=0.49\linewidth]{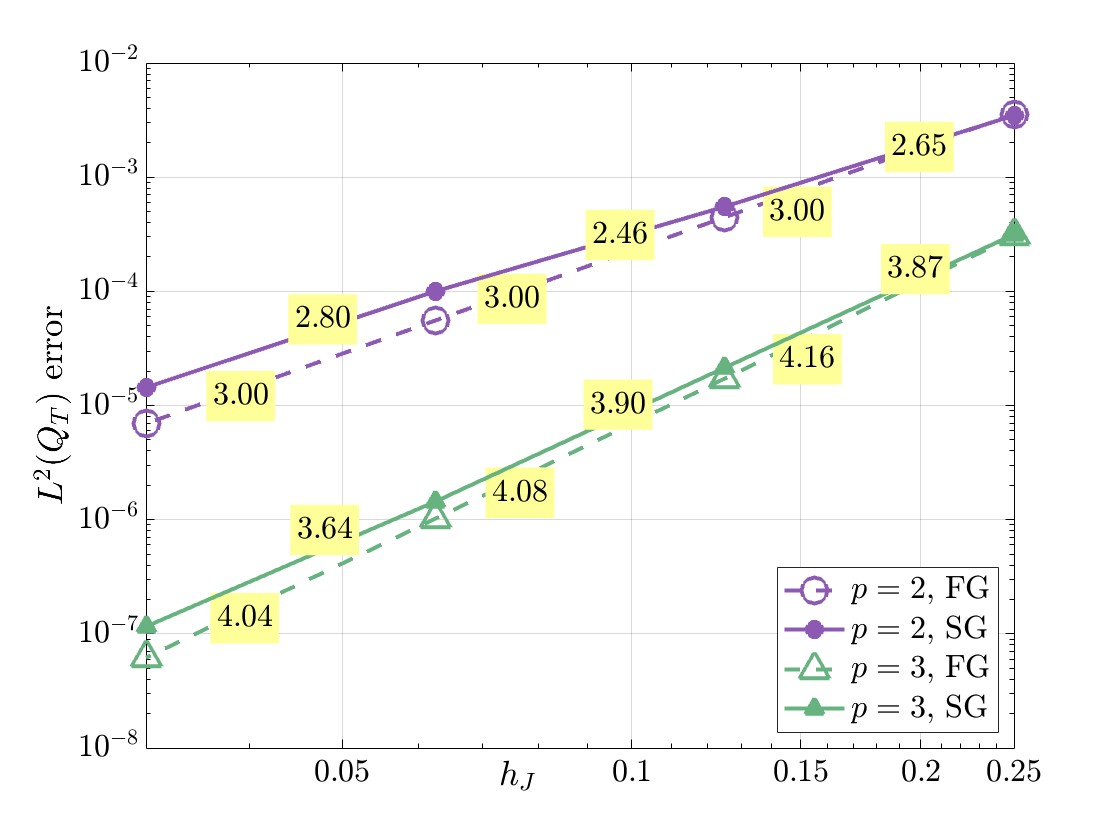}
\includegraphics[width=0.49\linewidth]{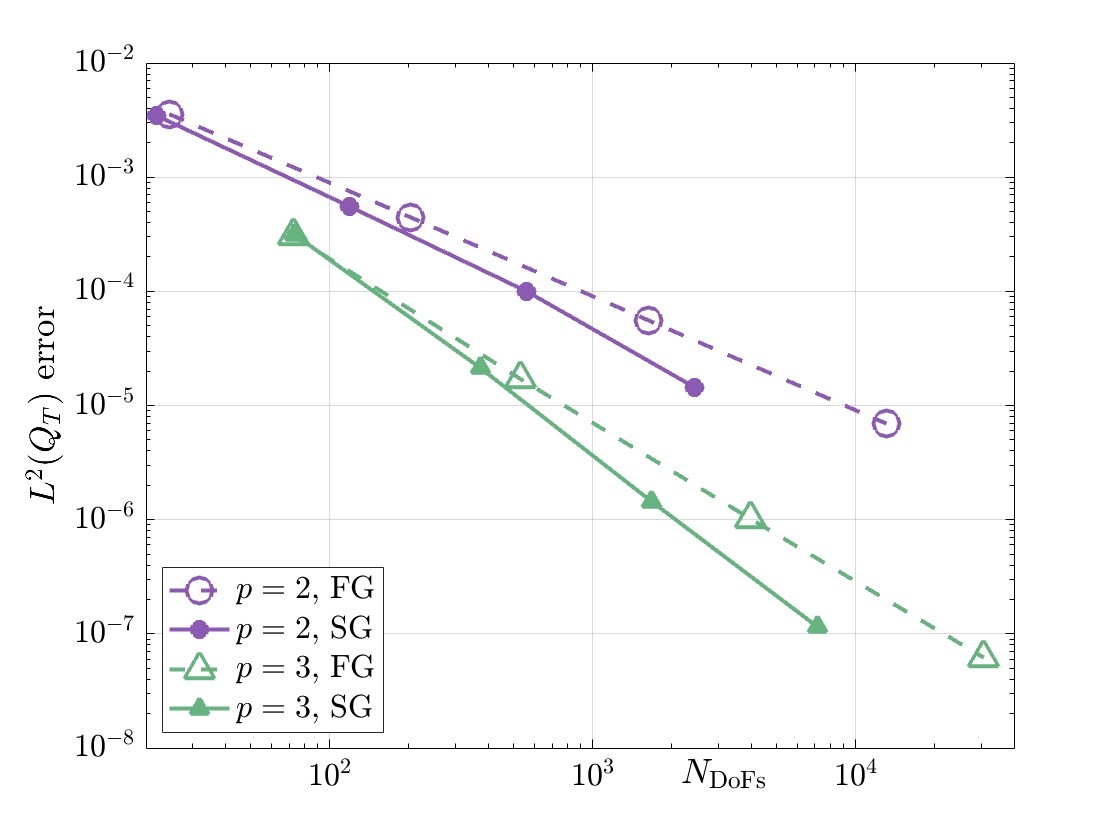}
\caption{Example in \S\ref{sec:example4}. $L^2(Q_T)$ relative errors for the exact solution $u_a$ as in~\eqref{eq:exp4_van}, plotted with respect to $h_J$ (left panel) and $N_{\mathrm{DoFs}}$ (right panel). Comparison between full-grid (FG) and sparse-grid (SG) approximations.}
\label{fig:2Dexample_van}
\end{figure}

\noindent
For the solution $u_b$, the left panel of Figure~\ref{fig:2Dexample_nonvan} shows that both methods converge at the optimal rate for $p=2$, whereas for $p=3$, the convergence rate of the sparse-grid scheme degrades. This example illustrates that certain initial \vanishingconditions are essential for achieving optimal rates. In particular, the numerical results indicate that a result comparable to Theorem~\ref{th:1} may fail if all initial \vanishingconditions are omitted. It therefore remains of interest to identify minimal sufficient conditions ensuring optimal convergence of the combination formula.               

\begin{figure}[htb!]
\includegraphics[width=0.49\linewidth]{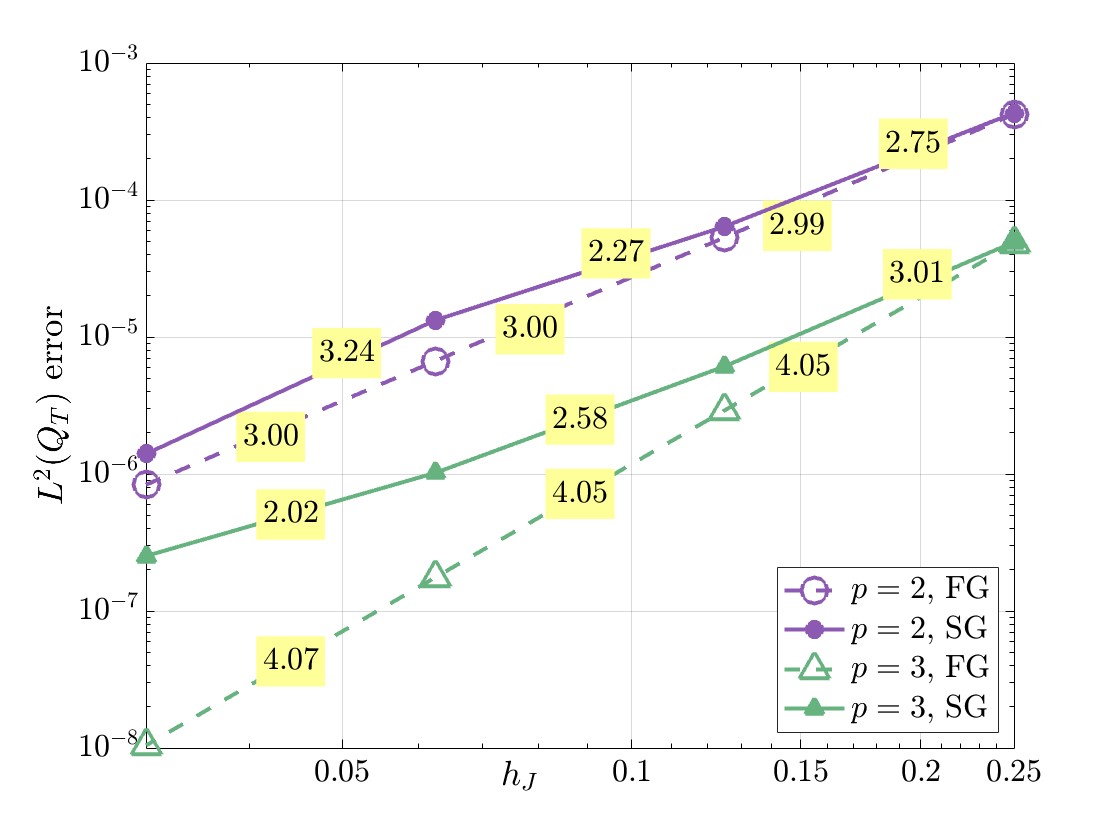}
\includegraphics[width=0.49\linewidth]{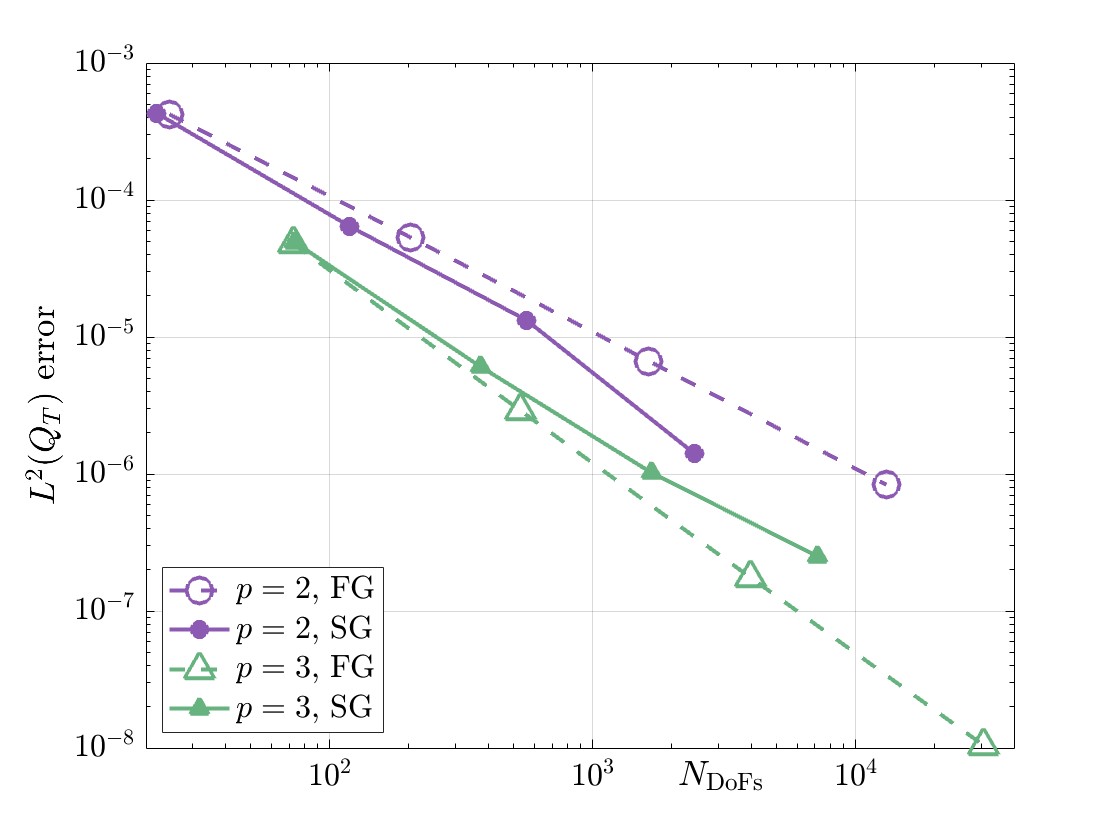}
\caption{Example in \S\ref{sec:example4}. $L^2(Q_T)$ relative errors for the exact solution $u_b$ as in~\eqref{eq:exp4_nonvan}, plotted with respect to $h_J$ (left panel) and $N_{\mathrm{DoFs}}$ (right panel). Comparison between full-grid (FG) and sparse-grid (SG) approximations.}
\label{fig:2Dexample_nonvan}
\end{figure}

\section{Conclusion}\label{sec:conclusion}
We have proposed and analyzed a fast space--time numerical scheme for the linear wave equation based on the combination technique on sparse grids. Based on an unconditionally stable second-order-in-time conforming Galerkin formulation, the method reduces the number of degrees of freedom compared to full tensor-product discretizations while preserving optimal convergence rates. Furthermore, the combination formula naturally enables efficient parallelization, which represents a key advantage over standard space--time discretizations.

\smallskip
\noindent 
We have established theoretical estimates in the $L^2(Q_T)$ norm, showing that optimal accuracy is achieved under suitable mixed-regularity assumptions and \vanishingconditions for the initial data and the source term. Theoretical findings have been validated through numerical experiments in both $(1+1)$- and $(2+1)$-dimensions. The numerical results confirm the predicted order of convergence and highlight the computational efficiency of the sparse-grid approach. Furthermore, our tests indicate that while some of the theoretical \vanishingconditions might be partially relaxed in practice without loss of optimality, they cannot be completely ignored without incurring an order reduction.
\medskip

\appendix

\begin{table}[ht!] 
\centering
	\begin{tblr}{
			colspec = {|c |c |c|},
			row{1} = {font=\bfseries},
			width = \linewidth}
		\hline
		\textbf{Symbol} & \textbf{Meaning} & \textbf{Definition}\\
		\hline
        $\Omega, T, Q_T = \Omega \times (0,T)$ & Spatial domain, final time, space--time cylinder & \S\ref{sec:notation}\\
$c, f, u_0, v_0$ & Wave velocity, source term, initial data & \eqref{eq:1}\\
        $H_0^1(\Omega), H^{s_{\bx}}_{0}(\Omega)$ &  Spatial spaces with zero boundary trace &  \S\ref{sec:notation}\\
        $H^2_{0,\bullet}(0,T), H^{s_t}_{0,\bullet}(0,T)$ & Temporal spaces with zero initial condition & \S\ref{sec:notation}\\
        $(u,v)_{L^2_e(0,T)},\| u \|_{L^2_e(0,T)}$ & Exponentially weighted $L^2(0,T)$ scalar product and norm  & \S\ref{sec:notation}\\
        $(u,v)_{L^2_e(Q_T)},\| u \|_{L^2_e(Q_T)}$ & Exponentially weighted $L^2(Q_T)$ scalar product and norm & \S\ref{sec:notation}\\
        $ \V(Q_T), \|\cdot\|_{\V(Q_T)} $ &Trial/test space and energy norm & \eqref{eq:4}, \eqref{eq:5}\\
        $\A$ & Bilinear form &\eqref{eq:6}\\
        $\jx, \jt$ &  Space and time refinement level indices &\S\ref{sec:numericalscheme}\\
        $\Vx, \Vt$ & Discrete spatial and temporal spaces &\S\ref{sec:numericalscheme}\\
        $h_{\jx}, h_{\jt}, h_0^{\bx}, h_0^t$ & Mesh parameters and initial mesh sizes &\S\ref{sec:numericalscheme}\\
        $P_{j_{\bx},j_t} u $ & Galerkin projection onto $\Vx \otimes \Vt$ & \eqref{eq:10}\\
        $\Pg_{j_{\bx}}, \Ppt_{j_t}$ & Projection operators &\eqref{eq:11}, \eqref{eq:12}\\
        $p_{\bx}, p_t$ & Polynomial approximation degrees & \S\ref{sec:numericalscheme}\\
		$u^{CF}_J, J$ & Combination formula, fine level index & \eqref{eq:19}\\
		$\Delta_{j_{\bx},j_t}^P u$ & Space--time detail operator &\eqref{eq:20}\\
        $N_{\mathrm{DoFs}}^{\mathrm{full}}, N_{\mathrm{DoFs}}^{\mathrm{sparse}}$ & Total number of degrees of freedom & \S\ref{sec:complexity}\\
        $\mathcal L$ & Differential operator in space & Thm.~\ref{th:1} \\
		$\Pinfjt,\Pjxinf$ & Semidiscrete projections & \eqref{eq:26}, \eqref{eq:27} \\
        $\{(\lambda_m,\phi_m)\}_{m=1}^\infty$ & Eigenpairs of $\mathcal{L}$ with Dirichlet boundary conditions & \S\ref{sec:stab}\\
        $\mathcal{H}^{s_{\bx}}(\Omega),\, \|\cdot\|_{\mathcal{H}^{s_{\bx}}(\Omega)}$ & Spectral Sobolev space associated with $\mathcal{L}$ and its norm & \S\ref{sec:stab}\\
		$\mathcal{B}_{j_t}$ & Discrete temporal operator & \eqref{eq:40}\\
		$\mathcal{L}_{j_{\bx}}$ & Discrete spatial operator &\eqref{eq:44}\\
		$\Delta_{j_{\bx},j_t}^Qu $ & Detail projection operator & \eqref{eq:66}\\
		\hline
 	\end{tblr}
	\caption{Summary of the main notation and symbols.
    }
	\label{tab:3}
\end{table}

\section{Semidiscrete error estimates} \label{sec:errest}
In this section, we derive estimates for $(P_{\infty,j_t}-\Id)u$ and $(P_{j_{\bx},\infty}-\Id)u$. 
In order to do so, we first prove two estimates analogous to \eqref{eq:17}. The proof follows a similar argument to that of \cite[Lemma 4.1]{FerrariPerugia2026}; we report it here for the sake of completeness.
\begin{proposition}\label{lem:1}
Let~$u \in \mathcal{V}(Q_T)$, and let~$\Pinfjt$ and $\Pjxinf$ be the unique solutions to the discrete problems~\eqref{eq:26} and~\eqref{eq:27}, respectively.
\begin{itemize}
\item Suppose that $\nabla_{\bx} \cdot (c^2 \nabla_{\bx} u) \in L^2(\Omega) \otimes H^2(0,T)$. Then, the following estimate holds: 
\begin{align}\label{eq:31}
    \| (P_{\infty,j_t} - \Id^{\bx} \otimes \Pi_{j_t}^{\partial_t^2}) u \|_{\mathcal{V}(Q_T)} & \lesssim \| (\Id^{\bx} \otimes \Ppt_{j_t} - \Id) \nabla_{\bx} \cdot (c^2 \nabla_{\bx} u) \|_{L^2(Q_T)}.
\end{align}
\item Suppose that $u \in H^1(\Omega) \otimes H^2(0,T)$. Then, the following estimate holds: 
\begin{align} \label{eq:32}
    \| (P_{j_{\bx},\infty} - \Pi_{j_{\bx}}^{\nabla_{\bx}} \otimes \Id^t) u \|_{\mathcal{V}(Q_T)} & \lesssim \| (\Pg_{j_{\bx}} \otimes \Id^t - \Id) \partial_t^2 u \|_{{L^2(Q_T)}} + \|(\Pg_{j_{\bx}} - \Id^{\bx}) \partial_t u(\cdot,0)\|_{L^2(\Omega)}.
\end{align}
\end{itemize}
\end{proposition}
\begin{proof}
We start with~\eqref{eq:31}. We use the coercivity in \eqref{eq:8} and consistency to obtain
\begin{equation} \label{eq:33}
\begin{aligned}
    \| (P_{\infty,j_t} - \Id^{\bx} \otimes \Ppt_{j_t}) u \|^2_{\mathcal{V}(Q_T)} & \lesssim \A((P_{\infty,j_t} - \Id^{\bx} \otimes \Ppt_{j_t}) u, (P_{\infty,j_t} - \Id^{\bx} \otimes \Ppt_{j_t}) u)
    \\ & = \A((\Id - \Id^{\bx} \otimes \Ppt_{j_t}) u, (P_{\infty,j_t} - \Id^{\bx} \otimes \Ppt_{j_t}) u).
\end{aligned}
\end{equation}
Define $w_{\infty,j_t} := (P_{\infty,j_t} - \Id^{\bx} \otimes \Ppt_{j_t}) u$. Using the definitions of~$\A$ in \eqref{eq:6} and $\Ppt_{j_t}$ in \eqref{eq:12}, we write
\begin{equation} \label{eq:76}
\begin{aligned}
      \A((\Id - \Id^{\bx} \otimes \Ppt_{j_t}) u, w_{\infty,j_t}) & = (\partial_t^2 (\Id - \Id^{\bx} \otimes \Ppt_{j_t}) u, \partial_t w_{\infty,j_t})_{L_e^2(Q_T)}
      \\ & \quad + (\partial_t (\Id - \Id^{\bx} \otimes \Ppt_{j_t}) u(\cdot,0), \partial_t w_{\infty,j_t}(\cdot,0))_{L^2(\Omega)}
     \\ & \quad + (c^2\nabla_{\bx} (\Id - \Id^{\bx} \otimes \Ppt_{j_t}) u, \nabla_{\bx} \partial_t w_{\infty,j_t})_{L_e^2(Q_T)}
     \\ & = (c^2\nabla_{\bx} (\Id - \Id^{\bx} \otimes \Ppt_{j_t}) u, \nabla_{\bx} \partial_t w_{\infty,j_t})_{L_e^2(Q_T)}.
\end{aligned}
\end{equation}
Integration by parts in space and the commutativity of $\Ppt_{j_{t}}$ with $\nabla_{\bx} \cdot (c^2 \nabla_{\bx}\cdot )$ yield
\begin{equation} \label{eq:77}
\begin{aligned}
     (c^2 \nabla_{\bx} (\Id - \Id^{\bx} \otimes \Ppt_{j_t}) u, \nabla_{\bx} \partial_t w_{\infty,j_t})_{L_e^2(Q_T)} & = - (\nabla_{\bx} \cdot(c^2 \nabla_{\bx} (\Id-\Id^{\bx} \otimes \Ppt_{j_t}) u),\partial_t w_{\infty,j_t})_{L_e^2(Q_T)}
    \\ & = - ((\Id - \Id^{\bx} \otimes \Ppt_{j_t}) \nabla_{\bx} \cdot(c^2 \nabla_{\bx} u), \partial_t w_{\infty,j_t})_{L_e^2(Q_T)}.
\end{aligned}
\end{equation}
Combining \eqref{eq:76} and \eqref{eq:77} with the Cauchy--Schwarz inequality, and recalling the definition of the~$\mathcal{V}(Q_T)$-norm in~\eqref{eq:5}, we obtain
\begin{equation*}
\begin{aligned}
    & \A((\Id - \Id^{\bx} \otimes \Ppt_{\jt}) u, w_{\infty,j_t}) \lesssim \|(\Id - \Id^{\bx} \otimes \Ppt_{\jt}) \nabla_{\bx} \cdot (c^2 \nabla_{\bx} u) \|_{L^2(Q_T)} \| w_{\infty,\jt} \|_{\mathcal{V}(Q_T)}. 
     \end{aligned}
\end{equation*}
This, together with \eqref{eq:33}, yields \eqref{eq:31}.

\noindent 
For~\eqref{eq:32}, we use again the coercivity in \eqref{eq:8} and consistency to obtain
\begin{equation} \label{eq:78}
\begin{aligned}
    \| (P_{j_{\bx},\infty} - \Pi_{j_{\bx}}^{\nabla_{\bx}} \otimes \Id^t) u \|^2_{\mathcal{V}(Q_T)} 
   & \lesssim \A((P_{j_{\bx},\infty} - \Pi_{j_{\bx}}^{\nabla_{\bx}} \otimes \Id^t) u, (P_{j_{\bx},\infty} - \Pi_{j_{\bx}}^{\nabla_{\bx}} \otimes \Id^t) u)
   \\ & = \A((\Id-\Pi_{j_{\bx}}^{\nabla_{\bx}} \otimes \Id^t) u, (P_{j_{\bx},\infty} - \Pi_{j_{\bx}}^{\nabla_{\bx}} \otimes \Id^t) u).
\end{aligned}
\end{equation}
Define $w_{j_{\bx},\infty} := (P_{j_{\bx},\infty} - \Pi_{j_{\bx}}^{\nabla_{\bx}}\otimes \Id^t) u$. Using the definitions of~$\A$ in \eqref{eq:6} and $\Pi_{j_{\bx}}^{\nabla_{\bx}}$ in \eqref{eq:11}, we write
\begin{equation*}
\begin{aligned}
     \A((\Id-\Pi_{j_{\bx}}^{\nabla_{\bx}} \otimes \Id^t) u, w_{j_{\bx},\infty}) & = (\partial_t^2 (\Id-\Pi_{j_{\bx}}^{\nabla_{\bx}} \otimes \Id^t)u, \partial_t w_{j_{\bx},\infty})_{L_e^2(Q_T)}
     \\ & \quad + (\partial_t (\Id-\Pi_{j_{\bx}}^{\nabla_{\bx}} \otimes \Id^t)u(\cdot,0), \partial_t w_{j_{\bx},\infty} (\cdot,0))_{L^2(\Omega)}
     \\ & \quad + (c^2\nabla_{\bx} (\Id-\Pi_{j_{\bx}}^{\nabla_{\bx}} \otimes \Id^t)u, \nabla_{\bx} \partial_t w_{j_{\bx},\infty})_{L_e^2(Q_T)}
    \\ & = (\partial_t^2 (\Id-\Pi_{j_{\bx}}^{\nabla_{\bx}} \otimes \Id^t)u, \partial_t w_{j_{\bx},\infty})_{L_e^2(Q_T)}
     \\ & \quad + (\partial_t (\Id-\Pi_{j_{\bx}}^{\nabla_{\bx}} \otimes \Id^t)u(\cdot,0), \partial_t w_{j_{\bx},\infty}(\cdot,0))_{L^2(\Omega)}.
\end{aligned}
\end{equation*}
The commutativity of $\Pg_{j_{\bx}}$ with $\partial_t$ yields
\begin{equation*}
\begin{aligned}
    \A((\Id-\Pi_{j_{\bx}}^{\nabla_{\bx}} \otimes \Id^t) u, w_{j_{\bx},\infty}) & = ((\Id-\Pi_{j_{\bx}}^{\nabla_{\bx}} \otimes \Id^t) \partial_t^2  u, \partial_t w_{j_{\bx},\infty})_{L_e^2(Q_T)} 
    \\ & \quad + ((\Id-\Pi_{j_{\bx}}^{\nabla_{\bx}} \otimes \Id^t)\partial_t  u(\cdot,0), \partial_t w_{j_{\bx},\infty}(\cdot,0))_{L^2(\Omega)},
\end{aligned}
\end{equation*}
from which we deduce, using the Cauchy--Schwarz inequality,
\begin{equation*}
\begin{aligned}
    \A((\Id-\Pi_{j_{\bx}}^{\nabla_{\bx}} \otimes \Id^t) u, w_{j_{\bx},\infty}) \lesssim \Big( \|(\Id -\Pg_{j_{\bx}} \otimes \Id^t)\partial_t^2 u\|_{L_e^2(Q_T)} + \|(\Id^{\bx} -\Pg_{j_{\bx}}) \partial_t u(\cdot,0)\|_{L^2(\Omega)}\Big) \| w_{j_{\bx},\infty} \|_{\mathcal{V}(Q_T)}.
\end{aligned}
\end{equation*}
This, together with \eqref{eq:78}, yields \eqref{eq:32}.
\end{proof}
\begin{proposition}
Let $u \in \mathcal{V}(Q_T)$, and let $P_{\infty,j_t} u$ and $P_{j_{\bx},\infty} u$ be the solutions to the discrete problems \eqref{eq:26} and \eqref{eq:27}, respectively. If $u \in H^2(\Omega) \otimes H_{0,\bullet}^{p_t+3}(0,T)$, then
\begin{equation} \label{eq:79}
    \|(P_{\infty,j_t}-\Id) u\|_{L^2(Q_T)} \lesssim 2^{-j_t(p_t+1)} \|u\|_{H^2(\Omega) \otimes H^{p_t+3}(0,T)}.
\end{equation}
Similarly, assuming $u \in H_0^{p_{\bx}+1}(\Omega) \otimes H^2(0,T)$, we have
\begin{equation} \label{eq:80}
    \|(P_{j_{\bx},\infty}-\Id) u\|_{L^2(Q_T)} \lesssim 2^{-j_{\bx}(p_{\bx}+1)} \left( \|u\|_{H^{p_{\bx}+1}(\Omega) \otimes H^2(0,T)} + \|\partial_t u(\cdot,0)\|_{H^{p_{\bx}+1}(\Omega)} \right).
\end{equation}
\end{proposition}
\begin{proof} Applying the triangle inequality, the Poincar\'e inequality, and the bound \eqref{eq:31}, we obtain
\begin{align*}
    \| (P_{\infty,j_t}-\Id)u \|_{L^2(Q_T)} & \le \|(P_{\infty,j_t} - \Id^{\bx} \otimes \Pi_{j_t}^{\partial_t^2})u \|_{L^2(Q_T)} + \|(\Id^{\bx} \otimes \Pi_{j_t}^{\partial_t^2}-\Id) u\|_{L^2(Q_T)}  
    \\ & \lesssim \|(P_{\infty,j_t} - \Id^{\bx} \otimes \Pi_{j_t}^{\partial_t^2})u \|_{\mathcal{V}(Q_T)} + \|(\Id^{\bx} \otimes \Pi_{j_t}^{\partial_t^2}-\Id)u\|_{L^2(Q_T)}  
    \\ & \lesssim \| (\Id^{\bx} \otimes \Ppt_{j_t}-\Id) \nabla_{\bx} \cdot (c^2 \nabla_{\bx} u)\|_{L^2(Q_T)} + \|(\Id^{\bx} \otimes \Pi_{j_t}^{\partial_t^2}-\Id) u\|_{L^2(Q_T)}.
\end{align*}
Estimate~\eqref{eq:79} then follows from~\eqref{eq:14}.

\noindent
Similarly, by the triangle inequality, the Poincar\'e inequality, and the bound \eqref{eq:32}, we have
\begin{align*}
    \|(P_{j_{\bx},\infty}-\Id)u\|_{L^2(Q_T)} & \le \|(P_{j_{\bx},\infty} - \Pi_{j_{\bx}}^{\nabla_{\bx}}\otimes\Id^{t}) u\|_{L^2(Q_T)} + \| (\Pi_{j_{\bx}}^{\nabla_{\bx}}\otimes\Id^{t} - \Id) u\|_{L^2(Q_T)}
    \\ & \lesssim \|(P_{j_{\bx},\infty} - \Pi_{j_{\bx}}^{\nabla_{\bx}}\otimes\Id^{t}) u\|_{\mathcal{V}(Q_T)} + \| (\Pi_{j_{\bx}}^{\nabla_{\bx}}\otimes\Id^{t} - \Id) u\|_{L^2(Q_T)}
    \\ & \lesssim \|( \Pi_{j_{\bx}}^{\nabla_{\bx}}\otimes\Id^{t}-\Id)\partial_t^2 u\|_{L^2(Q_T)} + \|( \Pi_{j_{\bx}}^{\nabla_{\bx}}-\Id^{\bx})\partial_t u(\cdot,0)\|_{L^2(\Omega)} 
    \\ & \hspace{4.9cm} + \| (\Pi_{j_{\bx}}^{\nabla_{\bx}}\otimes\Id^{t} - \Id) u\|_{L^2(Q_T)}.
\end{align*}
Using~\eqref{eq:13}, we finally obtain~\eqref{eq:80}.
\end{proof}
\begin{remark}[Convergence of the discretization and semidiscretizations]
If $u$ satisfies the regularity assumption~\eqref{eq:15}, the error estimate in \eqref{eq:16} implies that
\begin{equation*}
    \lim_{(j_{\bx},j_t) \to (\infty,\infty)} P_{j_{\bx},j_t} u = u \quad \text{in~} L^2(Q_T).
\end{equation*}
Additionally, if~$u \in H^2(\Omega) \otimes H_{0,\bullet}^{p_t+3}(0,T)$, then estimate~\eqref{eq:79} ensures that
\begin{equation} \label{eq:81}
    \lim_{j_t \to \infty} P_{\infty,j_t} u = u\quad \text{in~} L^2(Q_T),
\end{equation}
and, if $u \in H_0^{p_{\bx}+1}(\Omega) \otimes H^2(0,T)$, estimate~\eqref{eq:80} ensures that
\begin{equation*}
    \lim_{j_{\bx} \to \infty} P_{j_{\bx},\infty} u = u\quad \text{in~} L^2(Q_T).
\end{equation*}
\eremk
\end{remark}

\bibliographystyle{plain}
\bibliography{bibliography}

\end{document}